\documentclass[a4paper, 11pt]{article}
\usepackage [a4paper,left=2.5cm,bottom=2.5cm,right=2.5cm,top=2.5cm]{geometry}
\usepackage[english]{babel}
\usepackage{amssymb}
\usepackage{amsmath,amsthm, dsfont}
\usepackage{subcaption}
\usepackage{tabularx,array}
\usepackage{graphicx, url}
\usepackage{enumitem} 

\theoremstyle{plain}
\theoremstyle{plain}\newtheorem{assumption}{}

\newtheorem{theorem}{Theorem}[section]
\newtheorem{lemma}[theorem]{Lemma}

\newtheorem{proposition}[theorem]{Proposition}
\newtheorem{example}[theorem]{Example}
\newtheorem{remark}[theorem]{Remark}

\usepackage{color}
\usepackage[dvipsnames]{xcolor}

\newcommand{\one}{\mathds{1}}

\newcommand{\E}{\mathbb{E}}
\newcommand{\R}{\mathbb{R}}

\renewcommand{\P}{\mathbb{P}}

\title{Parameter estimation  of discretely observed 
interacting particle systems}
\author{Chiara Amorino$^{*}$, Akram Heidari$^{*}$, Vytaut\.e Pilipauskait\.e$^{\dagger}$,  Mark Podolskij$^{*}$}

\date{}

\begin{document}
\maketitle

\footnote{
The authors gratefully acknowledge financial support of ERC Consolidator Grant 815703 “STAMFORD: Statistical Methods for High Dimensional Diffusions”.

$^{*}$Unit\'e de Recherche en Math\'ematiques, Universit\'e du Luxembourg%, %\url{chiara.amorino@uni.lu}. \\

{$^\dagger$Department of Mathematical Sciences, Aalborg University, Denmark}
}

\begin{abstract}
In this paper, we consider the problem of joint parameter estimation for drift and
diffusion coefficients of a stochastic McKean-Vlasov
equation and for the associated system of interacting particles. The analysis is provided in a general framework, as both coefficients depend on the solution 
%of the process 
and on the law of the solution itself. Starting from discrete observations of the interacting particle system over a fixed interval $[0, T]$, we propose a contrast function based on a pseudo likelihood approach. We show that the associated estimator is consistent  when the discretization step ($\Delta_n$)
and the number of particles ($N$) satisfy $\Delta_n \rightarrow 0$ and $N \rightarrow \infty$, and asymptotically normal when additionally  the condition $\Delta_n N \rightarrow 0$ holds. \\

\noindent
 \textit{Keywords:} Asymptotic normality, consistency, interacting particle systems, McKean-Vlasov equation, nonlinear diffusion, parameter estimation. \\
 
\noindent
\textit{AMS 2010 subject classifications:} 62F12, 62E20, 62M05, 60G07, 60H10.

\end{abstract}

\section{Introduction}
In this paper we focus on parametric estimation of interacting particle system of the form 
\begin{equation}
\begin{cases}
 d X_t^{\theta,i, N} = b \big(\theta_1, X_t^{\theta,i, N}, \mu_t^{\theta, N} \big) dt + a \big( \theta_2, X_t^{\theta,i, N}, \mu_t^{\theta, N} \big) d W_t^i,  \qquad i = 1, ... , N, \quad t\in [0, T], \\[1.5 ex]
 \mathcal{L} \big( X_0^{{\theta,} 1, N}, ... , X_0^{{\theta,} N, N} \big) : = \mu_0 \times ... \times \mu_0.
\end{cases}
 \label{eq: model}
\end{equation}
Here the unknown parameter $\theta
:= (\theta_1, \theta_2)$ belongs to the set $\Theta := \Theta_1 \times \Theta_2$, where $\Theta_j\subset \R^{p_j}${, $j=1,2$,} are compact {and convex} sets; we set 
$p: = p_1 + p_2$. The processes $(W^i_t)_{t \in [0, T]}$, $i=1,\dots,N$, are  independent $\mathbb{R}$-valued Brownian motions, independent of the initial {value $(X^{\theta,1,N}_0, \dots, X^{\theta,N,N}_0)%^\top
$ of the system}
%distribution $\mu_0$, 
and $\mu_t^{\theta, N}$ is the empirical measure of the system at time $t$, i.e.\
\[
\mu_t^{\theta, N} := \frac{1}{N} \sum_{i = 1}^N \delta_{X_t^{\theta, i, N}}.
\]
The model coefficients are functions $b: U_1 \times \R \times \mathcal{P}_2 \rightarrow \R$ and $a: U_2 \times \R \times \mathcal{P}_2 \rightarrow \R$, where $U_1$ and $U_2$ are two open sets containing $\Theta_1$ and $\Theta_2$, respectively, and $\mathcal{P}_2$ denotes the set of probability measures on $\R$ with a finite second moment, endowed with the Wasserstein 2-metric
\begin{equation}
W_2(\mu, \nu) := \Big( \inf_{m \in \Gamma (\mu, \nu)} \int_{\R^2} |x - y|^2 m(dx, dy) \Big)^{\frac 1 2},
\label{eq: wass}
\end{equation}
and $\Gamma(\mu, \nu)$ denotes the set of probability measures on $\R^2$ with marginals $\mu$ and $\nu$. The underlying observations
are
$$% 
\big(X_{t_{j,n}}^{\theta,i,N}\big)_{j= 1, \dots, n}^{i=1, \ldots, N}, $$
where 
%$t_{j,n} \in [0,T]$ and $\Delta_n:= t_{j+1,n} - t_{j,n}$ is the discretization frequency {\color{blue}(depends on $j$ unless observations are equidistant over $[0,T]$? 
$t_{j,n} := {T j/n}$ and $\Delta_n := {T/n}$ is the discretization %frequency
{step}.
%)}. 
We assume 
that the time horizon $T$ is fixed, and %$N\to \infty$, $\Delta_n\to 0$. 
${N,n \to \infty}$.

The interacting particle system is naturally associated to its
mean field equation as $N \rightarrow \infty$. The latter is described by  the 1-dimensional McKean-Vlasov SDE
\begin{equation}
d \bar{X}_t^{\theta} = b \big(\theta_1, \bar{X}_t^{\theta}, \bar{\mu}_t^{\theta} \big) dt + a \big(\theta_2, \bar{X}_t^{\theta}, \bar{\mu}_t^{\theta} \big) d W_t,  \quad t\in [0, T],
\label{eq: McK}
\end{equation}
where $\bar{\mu}_t^{\theta}$ is the law of $\bar{X}_t^{\theta}$ and $(W_t)_{t \in [0, T]}$ is a standard Brownian motion{, independent of the initial value $\bar X^\theta_0$ having the law $\bar \mu_0^\theta := \mu_0$}.
This equation is non-linear in the sense of McKean, see e.g.\ \cite{60imp,61imp,79imp}. It means, in particular, that the coefficients {depend not only on the current state but also on the current distribution of the solution.}
%not only depend on the solution, but they also on the law of the solution. 
It is well known that, under appropriate assumptions on the coefficients $a$ and $b$, it is possible to obtain a phenomenon commonly named \textit{propagation of chaos} (see e.g.\ \cite{79imp}). It implies that the empirical law $\mu_t^{\theta, N}$ weakly converges to $\bar{\mu}_t^{\theta}$ as $N \rightarrow \infty$. The McKean-Vlasov SDE in \eqref{eq: McK} links to a non-linear non-local partial differential equation on the space of probability measures (see e.g.\ \cite{Cat08}), which naturally arises in several applications in statistical physics. Indeed, stochastic systems of interacting particles and the associated McKean non-linear Markov processes have been introduced in 1966 in \cite{60imp} starting from statistical physics, to model the dynamics of plasma. Their importance has increased in time, and a huge number of probabilistic tools have been progressively developed in this context (see  \cite{Cat08,Fer97,Mal01,Mel96}, just to name a few).

On the other hand, however, statistical inference in this framework remained out of reach for many years (except for the early work of Kasonga in \cite{Kas90}), mainly as microscopic particle systems derived from statistical physics are not directly observable. Later on, McKean-Vlasov models found applications in several other fields, in which the data is observable. Nowadays, these models are used in finance (smile calibration in \cite{smile}; systemic risk in \cite{Fou13}) as well as social sciences (opinion dynamics in \cite{Cha17}) or mean-field games (see e.g.\ \cite{Car19, Zanella, Giesecke}). Moreover, some applications in neuroscience and population dynamics can be found respectively in \cite{Bal12} and \cite{Mog99}. At the same time, the interest in analysis of statistical models related to PDEs has gradually increased. A clear illustration of that is provided by the works on nonparametric Bayes and uncertainty quantification for inverse problems, as in \cite{AbrNic,Nickl1,Nickl2}.

Motivated by the increasing interest in statistical inference for McKean-Vlasov processes, we aim at estimating jointly the parameters ${\theta_1, \theta_2}$ starting from the discrete observations of the interacting particle systems \eqref{eq: model} over a fixed time interval $[0, T]$. 
Despite recent interest in the study of the McKean-Vlasov SDEs, the problem of parameter estimation for this class has received relatively little attention. In \cite{Wen} the authors established asymptotic consistency and normality of the maximum likelihood estimator for a class of McKean-Vlasov SDEs with constant %volatility
{diffusion coefficient}, based on the continuous observation of the trajectory. This has been extended to the path dependent case in \cite{Liu}. The mean field regime has been firstly considered by Kasonga in \cite{Kas90}, who studied a system of interacting diffusion processes depending linearly in the drift coefficient on some unknown parameter. Starting from continuous observation of the system over a fixed time interval $[0, T]$, he showed that the MLE is consistent and asymptotically normal as $N \rightarrow \infty$. This has been extended in \cite{Sharrock} to the case where the parametrisation is not linear, while Bishwal \cite{Bis} extended it to the case where only discrete observations of the system are available and the parameter to be estimated is a function of time. In \cite{Giesecke} the authors develop an asymptotic inference approach based on the approximation of the likelihood function for mean-fields models of large interacting financial systems. Moreover, Chen \cite{24 imp} has established the optimal convergence rate for the MLE in the large $N$ and large $T$ case. Even in this work the drift coefficient is linear and the diffusion coefficient is constant.

Let us also mention the works  \cite{GenLar1, GenLar2}, where parametric inference for a particular class of nonlinear self-stabilizing SDEs is studied, starting from continuous observation of the non-linear diffusion. Some different asymptotic regimes are considered, such as the small noise and the long time horizon. The problem of the semiparametric estimation of the drift coefficient starting from the observation of the particle system at time $T$, for $T \rightarrow \infty$ is studied in \cite{Vyt}, while \cite{Marc} considers non-parametric estimation of the drift term in a McKean-Vlasov SDE, based on the continuous observation of the associated interacting particle system over a fixed time horizon. 

None of these works, however, consider the problem of the joint estimation of the drift and %volatility 
{diffusion} coefficients. Moreover, not only we are not aware of any work about parameter estimation for interacting particle system where the %volatility 
{diffusion} coefficient can depend on the solution and on the law of the solution itself, but in the majority of the above mentioned work the %volatility 
{diffusion coefficient} is directly assumed to be constant. We consider a more general model, as in \eqref{eq: model}, motivated by several applications in which the %volatility function 
{diffusion coefficient} depends on the law. For example, this is the case in mathematical finance for the calibration of local and stochastic volatility models, with applications connected to the Dupire’s local volatility function (see \cite{Bos,Gyo,Lac}). Moreover, they are used to capture the diversity of a financial market, as in \cite{Alm}. 

We underline that the joint estimation of the two parameters introduces some significant difficulties: since the drift and
the {diffusion coefficient}
%volatility 
parameters are not estimated at the same rate, we have to deal with asymptotic properties in two different regimes. Another challenge comes from the fact that both  coefficients depend on the empirical law of the process. This introduces some complexity compared to the case where $a$ is constant. 
%Indeed, in the latter setting the independence of the Brownian motions leading the dynamics of the different particles is enough to obtain the independence of the stochastic integrals, which makes easy to deal with some (centered) terms. Considering our general model \eqref{eq: model}, instead, a more detailed analysis of these terms is needed.

A natural approach to estimation of unknown parameters in our context would be to use a maximum likelihood
estimation. However, the likelihood function based on the discrete sample is not tractable in this setting, since
it depends on the transition densities of the process, which are not explicitly known. To overcome this difficulty several methods have been developed, in the case of high frequency estimation for discretely observed classical SDEs. A widely-used method is to consider a pseudo likelihood
function, for instance based on the high frequency approximation of the dynamic of the process
by the dynamic of the Euler scheme, see for example \cite{FloZmi,Kes97,Yos92}. 

Our statistical analysis is based upon minimisation of a contrast function, which is similar in spirit to the methods \cite{FloZmi, Kes97,Yos92} that have been proposed in the setting of classical SDEs.
The main result of the paper is the consistency and asymptotic normality of the resulting estimator, which is showed by using a central limit theorem for martingale difference triangular arrays. The convergence rates for estimation of the two parameters are different,  which leads us to the study of the asymptotic properties of the contrast function in two different asymptotic schemes.
 Moreover, to illustrate our main results,  we present numerical experiments for %simulate 
two %examples 
models of interacting particle systems. %to illustrate our main results. 
Specifically, the first 
model is linear, %example is a linear mean-field model, 
while the second is a stochastic opinion dynamics model. 
%Our analysis reveals that w
While it is feasible to express the estimator explicitly for the linear %mean-field 
model, the estimator for the stochastic opinion dynamics model is implicit and can only be obtained numerically. Our results show that the proposed estimators perform well in both cases. 

We emphasize that our inference is made on the time horizon $[0, T]$ with $T$ being fixed. It is well known that it is impossible to estimate the drift parameter of a classical SDE on a finite time horizon. However, due to increasing number of particles, we are able  to consistently estimate the drift even when $T$ is fixed. 
%Indeed, in this framework the number of particles $N$ plays the role played by $T$ for the parameters estimation of the coefficients for stochastic differential equations. A clear illustration of that is the fact that the condition $\Delta_n T \rightarrow 0$ for $n, T \rightarrow \infty$, needed in \cite{FloZmir} in order to prove the asymptotic normality of the estimator for classical SDEs in long time horizon is now replaced by $\Delta_n N \rightarrow 0$ for $n,N \rightarrow \infty$ (see remark below Theorem \ref{th: normality} fur further details). \\
Moreover, it is worth remarking that our %study contains, 
results apply to 
%as a special case, 
the 
system of $N$ independent copies of a diffusion process as a special case. 
%with drift and diffusion coefficients dependent on unknown parameters.
%model of $N$ iid diffusions discretely observed on a fixed sample interval.
Non-parametric statistical inference for this type of system %model 
can be found for example in \cite{ComGen, MarRos, Den21} (see also references therein). 
Closer to the purpose 
of %this 
our work, 
%we refer instead to 
\cite{Mix, Den20} %for 
 discuss parameter estimation %for discretely observed 
from discrete observations of independent copies of a diffusion process
with mixed (or fixed) %random 
effects. 
Specifically, 
joint estimation of a fixed effect in the diffusion coefficient and parameters of the special distribution of a random effect (or a fixed effect) in the drift coefficient of the SDE is shown possible with the same rates of convergence in the same asymptotic framework as ours. 
%in the same asymptotic framework 
%the same rates of convergence as ours are proven
%for a joint estimator of a fixed effect and %fixed effects or 
%parameters of the special distribution of a random effect (or a fixed effect) in the diffusion and drift coefficient respectively of the SDE.
%The i
Interested readers can find further references about %diffusions 
SDEs with random %mixed
effects in the %above mentioned 
aforementioned papers.

The outline of the paper is as follows. In Section \ref{s: ass} we present the estimation approach,
list the required assumptions and demonstrate some examples. Section \ref{s: main} is devoted to  main results of the paper, which include consistency and asymptotic normality of the estimator.   
Section \ref{s: num} is devoted to  numerical experiments.
In Section \ref{s: tec} we provide the technical lemmas we will use in order to show our main results. The proofs of the main results are collected in Section \ref{s: proof main} while the technical results are shown in Section \ref{s: proof technical}.

\subsection*{Notation}

Throughout the paper all positive constants are denoted by $C$ or $C_q$ if they depend on an external parameter $q$. 
%{\color{blue} The transpose of a vector $\theta$ is denoted by $\theta^\top$.} 
{All vectors are row vectors, $\| \cdot \|$ denotes the Euclidean norm for vectors.%unless it is stated otherwise.
}
{ We write $f(\theta) = f(\theta_1,\theta_2)$ for $\theta%^\top 
= (\theta_1%^\top
,\theta_2%^\top
)$.}
%We denote by $C^{r}(\Theta_j ;\R)$ the set of functions $f: \Theta_j  \to \R$, which are $r\in \mathbb{N}$ times continuously differentiable in the direction $\theta_j$, $j=1,2$. 
{For $r=0,1,\dots$, we denote by $C^{r}(X;\R)$ the set of %$r\in \mathbb{N}$ 
$r$ times continuously differentiable functions $f: X \to \R$. We denote by $\partial_x f$ the partial derivative of a function $f(x,y,\dots)$ with respect to $x$. We denote by $\nabla_{\theta_j} f$ the vector $(\partial_{\theta_{j,1}} f, \dots, \partial_{\theta_{j,p_j}} f)$, $j=1,2$, and $\nabla_\theta f = (\nabla_{\theta_1} f, \nabla_{\theta_2} f)$.}
We say that a function $f:\R \times \mathcal{P}_l \to \R$ has \textit{polynomial growth} if 
\begin{align} \label{eq: pol growth}
|f(x,\mu)| \le C (1 +|x|^k +W_2^l(\mu,\delta_0) )
\end{align}
for some {$k,l =0,1,\dots$ and all $(x,\mu)\in \R\times \mathcal{P}_l$, where $\mathcal{P}_l$ denotes the set of probability measures on $\R$ with a finite $l$-th absolute moment}. { For $p \in [1,\infty)$, the Wasserstein $p$-metric between two probability measures $\mu$ and $\nu$ in $\mathcal{P}_p$
%on $\R$ with finite $p$-th absolute moments 
is given as
$$W_p(\mu, \nu) := \Big( \inf_{m \in \Gamma (\mu, \nu)} \int_{\R^2} |x - y|^p m(dx, dy) \Big)^{\frac 1 p};$$
where $\Gamma(\mu, \nu)$ denotes the set of probability measures on $\R^2$ with marginals $\mu$ and $\nu$. Finally,} 
%In the sequel, 
we suppress the dependence 
of several objects on the true parameter $\theta_0$. In particular, we write $\P := \P^{\theta_0}$, $\E := \E^{\theta_0}$,  
$X_t^{i, N}:=X_t^{\theta_0, i, N}$, $\bar X_t := \bar X^{\theta_0}_t$, $\mu_t^N:=\mu_t^{\theta_0, N}$ and 
$\bar \mu_t := \bar \mu^{\theta_0}_t$.
Furthermore, we denote by $\xrightarrow{\mathbb{P}}$, $\xrightarrow{\mathcal{L}}$, 
$\xrightarrow{L^p}$ the convergence in probability, in law, in $L^p$ respectively. We also denote the value $a^2(\theta_2,x,\mu)$ as $c(\theta_2,x,\mu)$. \\
% {\rev Finally, in analogy to the Wasserstein 2-metric introduced in \eqref{eq: wass}, we introduce the Wasserstein $p$ metric, for any $p \ge 1$, as 
% $$W_p(\mu, \nu) := \Big( \inf_{m \in \Gamma (\mu, \nu)} \int_{\R^2} |x - y|^p m(dx, dy) \Big)^{\frac 1 p};$$
% where $\Gamma(\mu, \nu)$ denotes the set of probability measures on $\R^2$ with marginal $\mu$ and $\nu$.}

\section{Minimal contrast estimator, assumptions and examples}{\label{s: ass}}

We aim at estimating the unknown parameter $\theta_0 
= (\theta_{0,1}, \theta_{0,2}) 
\in \Theta^{\circ}$ given equidistant discrete observations of the system introduced in \eqref{eq: model}. %{\color{blue}Since $T$ is fixed, we set $\Delta_n := T/n$.}
%Since $T$ is fixed, we will assume without loss of generality that  $T=1$ and set $n=\Delta_n^{-1}$ {\color{blue}(There is no need to fix $T= 1$: it suffices to replace $\int_0^1 \dots$ by $\int_0^T \dots$ everywhere and $n$ by $1/\Delta_n$ in the normalization in Thm \ref{th: normality}?)}. 
We study the asymptotic regime $N,n\to \infty$. 

The estimator we propose is based upon a contrast function, which originates from the Gaussian quasi-likelihood. 
Starting from discrete observations of the model there are  difficulties due to the fact that the transition density of the process is unknown. A common way to overcome this issue is to base the inference on a discretization of the continuous likelihood (see for example \cite{Gen90}, \cite{Kes97} and \cite{Yos92} where classic SDEs are considered). 
This motivates us to consider the following contrast function:
\begin{align}
S^N_n (\theta)
%(\theta_1,\theta_2) 
:= \sum_{i=1}^N \sum_{j=1}^n \Bigg\{ &\frac{\big(X^{i,N}_{t_{j,n}} - X^{i,N}_{t_{j-1,n}} - \Delta_n b\big(\theta_1, X^{i,N}_{t_{j-1,n}}, \mu^N_{t_{j-1,n}}\big)\big)^2}{\Delta_n c \big(\theta_2,X^{i,N}_{t_{j-1,n}}, \mu^N_{t_{j-1,n}}\big)} %\nonumber \\%%[1.5 ex]
%\qquad+ &
+\log c \big(\theta_2, X^{i,N}_{t_{j-1,n}}, \mu^N_{t_{j-1,n}}\big) \Bigg\},
\label{eq:def contrast}
\end{align}
%\begin{align}
%S^N_n (\theta)
%%(\theta_1,\theta_2) 
%&:= \sum_{i=1}^N \sum_{j=1}^n \Bigg\{ \frac{\left(X^{i,N}_{t_{j,n}} - X^{i,N}_{t_{j-1,n}} - \Delta_n b\left(\theta_1, X^{i,N}_{t_{j-1,n}}, \mu^N_{t_{j-1,n}}\right)\right)^2}{\Delta_n c \left(\theta_2,X^{i,N}_{t_{j-1,n}}, \mu^N_{t_{j-1,n}}\right)} \nonumber \\%[1.5 ex]
%&\qquad+ \log c \left(\theta_2, X^{i,N}_{t_{j-1,n}}, \mu^N_{t_{j-1,n}}\right) \Bigg\}.
%\label{eq:def contrast}
%\end{align}
{for $\theta = (\theta_1,\theta_2)$.} The estimator $\hat{\theta}_n^N = (\theta^N_{n,1},\theta^N_{n,2})$ of $\theta_0$ %= (\theta_{0,1}, \theta_{0,2})$ 
is obtained as 
\begin{equation*}
\hat{\theta}_n^N \in \mathop{\mathrm{arg\,min}}_{\theta \in \Theta} S_n^N(\theta).
\end{equation*}
%\begin{align*}
%\hat{\theta}_n^N = (\hat{\theta}_{n,1}^N, \hat{\theta}_{n,2}^N)%\textcolor{red}{^\top}
%\in \mathop{\mathrm{arg\,min}}_{(\theta_1, \theta_2) \in \Theta} S_n^N(\theta_1, \theta_2).
%\end{align*}
 Comparing $S_n^N(\theta)$ with the contrast function for parameter estimation for classical SDEs, the main difference consists in the fact that we have now an extra sum over the number of interacting diffusion processes. %particles. 
The interaction depends on the empirical measure of the system. The dependence of the drift and diffusion coefficients on the measure can take a general form.
%Even if it may not seems a big complication, it actually brings into play many challenges. Indeed, our analysis will be based on the study of second and fourth moments of the derivatives of the contrast (see the proof of Proposition \ref{p: norm L} below). The presence of the sum over the particles implies an extra $N^2$ and $N^4$ in the study of second and fourth moments, respectively. It entails some issues when one wants to prove the convergence of such objects to finite quantities. In order to solve them and 
In order to meet this challenge and prove some asymptotic properties for $\hat{\theta}_n^N$
we need to introduce a set of %conditions. 
assumptions.
The first two assumptions ensure the system's existence and uniqueness, while the next two 
%four conditions are 
impose additional regularity %assumptions 
conditions on the coefficients $a$ and~$b$.
\begin{assumption} \label{as1}
(\textit{Boundedness of moments})
\textit{For all $k \ge 1$,} %and $i \in \{ 1, ..., N \}$:}  
%\[ \E[|X_0^{ i, N}|^k] \leq C_k. \] 
$$
\int_{\R} |x|^k \mu_0 (d x) \le C_k.
$$
\end{assumption} 
%\textbf{A2}: \textit{The drift coefficient $b$ 
%$b: \Theta_1 \times \R \times \mathcal{P}_1 \rightarrow \R$ 
%and the diffusion coefficient $a$ 
%$a: \Theta_2 \times \R \times \mathcal{P}_1 \rightarrow \R$ 
%are 
%measurable with respect to $\mathcal{F}_{t} = \sigma \{(W_u^i)_{0 \le u \le t}, X_0^{i, N}; \quad i= 1, ... , N \}$ and 
%such that 
%$$\sup_{\theta_1 \in \Theta_1} |b(\theta_1, 0, \delta_0)|  + \sup_{\theta_2 \in \Theta_2} |a(\theta_2, 0, \delta_0)| < \infty.$$}
%\\

%We also require the Lipschitz condition on the coefficients. However, as they depend on the density, the Lipschitz continuity is formulated in a way that involves also the Wasserstein distance as introduced in \eqref{eq: wass}.
\begin{assumption} \label{as2}
(\textit{Lipschitz condition}) \textit{The drift and %volatility 
{diffusion} coefficients are Lipschitz continuous {in $(x,\mu)$}, i.e.\ {for all $\theta$ there exists $C$ such that} for all $(x,\mu), (y,\nu) \in \R \times {\cal P}_2$,
%$x,y \in \R$, $\mu,\nu \in {\cal P}_2$,
%: For all $x,y \in \R$ and $\mu,\nu \in {\cal P}_2$ it holds that
$$|b(\theta_1, x, \mu) - b(\theta_1, y, \nu)| + |a(\theta_2, x, \mu) - a(\theta_2, y, \nu)|  \le C ( |x - y| + W_2(\mu, \nu) ).$$
%uniformly in $\theta \in \Theta$. (To check if $C$ needs to be uniform in $\theta$)}
}
\end{assumption}
%{We remark it is possible to relax condition A2 on the drift coefficient requiring the global lipschitzianity only in the law. The global lipschitzianity in $x$ is instead replaced by the local lipschitzianity with polynomial growth in $x$ together with the one-sided lipschitzianity in $x$ (see Assumption 2.1 in \cite{DosReis}). The boundedness of the moments gathered in our Lemma \ref{l: moments} would in this case be replaced by Theorem 3.1 of \cite{AAP} and the propagation of the chaos needed in order to show Lemma \ref{l: Riemann} would derive in this case from Proposition 3.1 of \cite{DosReis}. In both these results the Wasserstein 2-metric is considered instead of the Wasserstein 1-metric we define in \eqref{eq: wass}. However, mimicking the proofs, it is easy to see that the above mentioned results still hold true in this context.}
\begin{assumption} \label{as3}
(\textit{Regularity of the diffusion coefficient})
\textit{The diffusion coefficient is uniformly bounded away from $0$:
$$\inf_{(\theta_2, x, \mu) \in \Theta_2 \times \R \times {\cal P}_2} c(\theta_2, x, \mu) >0. %\ge c.
$$}
\end{assumption}
\begin{assumption} \label{as4}
(\textit{Regularity of the derivatives}) \textit{(I) {For all $(x,\mu)$}, the functions $b(\cdot, x,\mu)$, %and 
$a(\cdot, x,\mu)$ are in $C^3(U_1;\R)$,
%and 
$C^3(U_2;\R)$ respectively. Furthermore, all their {partial} derivatives up to order three have polynomial growth, in the sense
of \eqref{eq: pol growth},
 uniformly in $\theta$. %{ (Enough that $\sup_{\theta_1 \in \Theta_1} |\partial_{\theta_{1,h}} b(\theta_1,x,\mu) |$, $\sup_{\theta_2 \in \Theta_2} | \partial_{\theta_{2,\tilde h}} a(\theta_2,x,\mu) |$, $h=1,\dots,p_1$, $\tilde h = 1,\dots, p_2$, have polynomial growth?)}
 \\
(II) The {first and} second order derivatives in $\theta$ are locally Lipschitz in $(x,\mu)$ with polynomial weights, i.e.\ {for all $\theta$ there exists $C >0$, $k,l =0,1,\dots$ such that for all $r_1 + r_2 =1, 2$, $h_1, h_2 = 1, ... , p_1$, $\tilde{h}_1, \tilde{h}_2 = 1, ... , p_2$,
$(x,\mu), (y,\nu) \in \R \times {\cal P}_2$,}
\begin{align*}
&\big|\partial_{\theta_{1,h_1}}^{r_1} \partial_{\theta_{1,h_2}}^{r_2} b (\theta_1, x, \mu)- \partial_{\theta_{1,h_1}}^{r_1} \partial_{\theta_{1,h_2}}^{r_2} b (\theta_1, y, \nu) \big| + \big|\partial_{\theta_{2,\tilde{h}_1}}^{r_1} \partial_{\theta_{2,\tilde{h}_2}}^{r_2} a (\theta_2, x, \mu)-\partial_{\theta_{2,\tilde{h}_1}}^{r_1} \partial_{\theta_{2,\tilde{h}_2}}^{r_2} a (\theta_2, y, \nu) \big|
\\%[1.5 ex]
&\qquad \le C (|x-y| + W_2 (\mu,\nu)) \big( 1 + |x|^k + |y|^k + W_2^l(\mu,\delta_0) + W_2^l(\nu,\delta_0) \big).
\end{align*}
%for some $k,l {\color{blue}=0,1,\dots}%%>0
%$ and uniformly in $\theta$, and all $r_1 + r_2 =1, 2$, $h_1, h_2 = 1, ... , p_1$, $\tilde{h}_1, \tilde{h}_2 = 1, ... , p_2${\color{blue}, %%$x, y\in \R$, $\mu, \nu \in {\cal P}_2$
%$(x,\mu), (y,\nu) \in \R \times {\cal P}_2$}.
} 
\end{assumption}

\begin{remark} \rm
(i) It is possible to relax assumption \textbf{\ref{as2}} on the drift coefficient to allow for a locally Lipschitz condition in $x$ with polynomial weights, cf.\ \cite[Assumption 2.1]{DosReis}. In this setting 
the boundedness of moments shown in our Lemma \ref{l: moments} can  be replaced by \cite[Theorem 3.3]{AAP} and the propagation of chaos needed in order to prove Lemma \ref{l: Riemann} would follow from \cite[Proposition 3.1]{DosReis}. As a consequence the main results of this paper still hold. \\ \\
(ii) 
%Condition (I) of 
\textbf{\ref{as4}(I) }
is sufficient to show consistency of the estimator $\hat{\theta}_n^N$. We require the additional condition \textbf{(II)} of \textbf{\ref{as4}} to prove the asymptotic normality. \qed
\end{remark}

\noindent
%It will be proven in the sequel that, as a consequence of the properties gathered in \ref{as2} and \ref{as4}, it follows 
%the polynomial growth of the coefficients. In particular, the following inequality holds true for some $c > 0$, $k, l \in \mathbb{N}$ and for any $x \in \R$ and any $\mu \in \mathcal{P}_1$
%\begin{equation}
%|b(\theta_1, x, \mu)| + |a(\theta_2, x, \mu)|  \le c (1 + |x|^k+ %W_1^l(\mu, \delta_0)).
%\label{eq: pol growth}
%\end{equation}
We now state an assumption on the identifiability of the model
and some further conditions that are required to prove the asymptotic normality. For this purpose we define the functions $I : \Theta \to \R$, $J : \Theta_2 \to \R$ as
\begin{align}\label{def:I}
%I (\theta_1,\theta_2)
I (\theta)&:= \int_0^{{T}} \int_{\R} \frac{\left(b(\theta_1,x, {\bar{\mu}_t%^{\theta_0}
}) - b(\theta_{0,1}x, { \bar{\mu}_t%^{\theta_0}
})\right)^2}{c(\theta_2,x, { \bar{\mu}_t%^{\theta_0}
})} 
{ \bar{\mu}_t%^{\theta_0}
} (d x) d t,\\[1.5 ex] \label{def:J}
J (\theta_2) &:= \int_0^{{T}} \int_{\R} \Big( \frac{c(\theta_{0,2},x, { \bar{\mu}_t%^{\theta_0}
})}{c (\theta_2,x, { \bar{\mu}_t%^{\theta_0}
})} + \log c (\theta_2,x, {\bar{\mu}_t%^{\theta_0}
}) \Big) { \bar{\mu}_t%^{\theta_0}
} (d x) d t,
\end{align}
{where recall that $\bar \mu_t$ stands for $\bar{\mu}_t^{\theta_0}$.}
The next set of conditions are the following assumptions.
\begin{assumption} \label{as5}
 \textit{(Identifiability)} \textit{The functions $I, J$ defined above satisfy that for every $\varepsilon > 0$,
$$
\inf_{\theta \in \Theta: \|\theta_1-\theta_{0,1}\|\ge \varepsilon} I(\theta)>0\qquad \mbox{and } \inf_{\theta_2 \in \Theta_2: \|\theta_2-\theta_{0,2}\|\ge \varepsilon} (J(\theta_2)-J(\theta_{0,2}))>0.
$$}
\end{assumption}
\begin{assumption}\label{as6}
 (\textit{Invertibility}) \textit{We define a $p \times p$ block diagonal matrix 
$\Sigma(\theta_0) {:}=
\operatorname{diag}(\Sigma^{(1)}(\theta_0),\Sigma^{(2)}(\theta_0))$ whose main-diagonal blocks 
%whose main-diagonal block 
$\Sigma^{(j)}(\theta_0) = (\Sigma^{(j)}_{kl}(\theta_0))$ %, is a $p_j \times p_j$ matrix,
%$j=1,2$, 
are defined via
%has 
 \begin{align*}
%\Sigma_{k,l}
\Sigma^{(j)}_{kl} (\theta_0) {:}= \begin{cases}
\displaystyle 2 \int_0^{{T}} \int_\mathbb{R} \frac{\partial_{\theta_{1, k}} b(\theta_{0, 1}, x, { \bar{\mu}_t%^{\theta_0}
}) \, \partial_{\theta_{1, l}} b(\theta_{0, 1}, x, { \bar{\mu}_t%^{\theta_0}
})}{c(\theta_{0, 2}, x, { \bar{\mu}_t%^{\theta_0}
})} { \bar{\mu}_t%^{\theta_0}
}(dx) dt, \qquad &j=1, \, k, l = 1, \dots, p_1,
\\
%\Sigma_{p_1 + k, p_1 + l} 
%\Sigma_{kl}^{(2)}(\theta_0) &= 
\displaystyle \int_0^{{T}} \int_\mathbb{R} \frac{\partial_{\theta_{2, k}}c(\theta_{0, 2}, x, { \bar{\mu}_t%^{\theta_0}
}) \, \partial_{\theta_{2, l}}c(\theta_{0, 2}, x, { \bar{\mu}_t%^{\theta_0}
})}{c^2(\theta_{0, 2}, x, {\bar{\mu}_t%^{\theta_0}
})}  { \bar{\mu}_t%^{\theta_0}
}(dx) dt, \qquad &j=2, \, k,l = 1,\dots, p_2.
\end{cases}
\end{align*}
We assume that $\operatorname{det}(\Sigma^{{(j)}} (\theta_0)) \neq 0${, $j=1,2$. %(Maybe to bring definitions above the assumption?)
}
}
\end{assumption}

\begin{assumption} \label{as7}
 \textit{ %The volatility coefficient satisfies one of the following two conditions
%\begin{enumerate}
 %   \item The map $\mu \mapsto a(\theta_{2, 0}, x, \mu)$ admits a functional derivative of order 2 in the sense of the above definition. Moreover, the derivatives of order 1 and 2 have polynomial growth in the sense that
  %  $|\partial_{\mu} f(x, \mu)| + |\partial^2_{\mu} f(x, \mu)| \le c(1 + |x|^k + W_1^l(\mu, \delta_0))$ for some $c > 0$, $k, l \in \mathbb{N}$. 
  %  \item 
 (\textit{Integral condition on the diffusion coefficient}) {At $\theta_{0,2}$ for all $(x,\mu)$ t}he diffusion coefficient takes the form 
 $$
 a(\theta_{{0,}2}, x, \mu) {:} = \Tilde{a}
 \Big( %\theta_2, 
 x, \int_{\mathbb{R}} K%_{\theta_2} 
 (x, y) \mu(dy)\Big)
 $$
 {for some functions 
% $\tilde a$, $K$ mapping the value $(\theta_2,x,y) \in \Theta_2 \times \R^2$ to the values $\tilde a_{\theta_2}(x,y)$, $K_{\theta_2}(x,y) \in \R$ respectively. 
 %For any $\theta_2 \in \Theta_2$, the functions $\tilde a%_{\theta_2}
 %$, $K%_{\theta_2}
 %$ are in $
 $\tilde a, K \in C^2(\R^2; \R)$,
 which satisfy $|\partial^{r_1}_x \partial_y^{r_2} \tilde a%_{\theta_2} 
 (x,y)| + |\partial^{r_1}_x \partial^{r_2}_y K%_{\theta_2}
 (x,y)| \le C (1+|x|^{k}+|y|^l)$ for some $k,l = 0,1,\dots$ and all $r_1+r_2=%0,
 1,2$, 
 $(x,y) 
 \in \R^2
 $.
 }
 }

\end{assumption}

%\end{enumerate}}
%We observe that condition 2 is stronger than 1 as, knowing that the volatility admits an integral form, it follows that the map $\mu \mapsto a(\theta_{2, 0}, x, \mu)$ admits a functional derivative $\partial_{\mu} a (\theta_{2, 0}, x, y, \mu) = \Tilde{a}_{\theta_{2, 0}}(x, y)$ and this can be iterated to the order two. \\
\noindent
%We remark that condition \ref{as7}, which is only required to show asymptotic normality, implies condition \ref{as2}. Indeed, under \ref{as7} we have 
%\begin{align*}
%|a(\theta_{2}, x, \mu) - a(\theta_{2},y, \nu)| & = \left|\int_{\R} \left(\Tilde{a}_{\theta_{ 2}}(x, \cdot) - \Tilde{a}_{\theta_{2}}(y, \cdot) \right) d\mu - \int_{\R} \Tilde{a}_{\theta_{2}}(y, \cdot) d(\nu - \mu)\right| \\
%& \le C (|x - y| + W_2(\mu, \nu)),
%\end{align*}
%where we have used the Lipschitz property of $\Tilde{a}_{\theta_{2}}$ and the fact that 
%$$\int_{\R} \Tilde{a}_{\theta_{2}}(y, \cdot) d(\nu - \mu) \le C \sup_{|\phi|_{Lip} \le 1}\int_{\R} \phi d(\nu - \mu) = C W_1(\mu, \nu) \le C W_2(\mu,\nu).$$
{Assumptions \textbf{\ref{as1}}-\textbf{\ref{as5}} are required to prove the consistency of our estimator and are relatively standard in the literature for statistics of random processes. However, Assumption \textbf{\ref{as5}} deserves some extra attention, as the quantities $I(\theta)$ and $J(\theta)$ are not at all explicit due to the presence of $\bar{\mu}_t%^{\theta_0}
$. Hence, it may be difficult to check Assumption \textbf{\ref{as5}} in practice and the identifiability of all parameters may not always be possible.
In order to delve deeper into the topic, we refer to Section 2.4 in \cite{Hof2}, where the authors have provided a thorough analysis. More specifically, for estimating the drift from continuous observations, they have identified explicit criteria that enable obtaining both identifiability and non-degeneracy of the Fisher information matrix. Notably, for a certain type of likelihood, they have established a connection between global identifiability and non-degeneracy of the Fisher information, which is highlighted in  \cite[Proposition 16]{Hof2}. It could be interesting to understand if it possible to prove an analogous proposition in our context, even if this is out of the purpose of the paper and it is therefore left for further investigation.

The additional conditions \textbf{\ref{as6}}-\textbf{\ref{as7}} are needed to obtain the central limit theorem, even if they are not of the same type. Indeed, \textbf{\ref{as6}} is an invertibility condition which is always required when one wants to prove asymptotic normality. {In \textbf{\ref{as6}}, note that $\partial_{\theta_{1,k}} b (\theta_{0,1},x,\bar \mu_t)$ and 
$\partial_{\theta_{2,k}} c (\theta_{0,2},x,\bar \mu_t)$ are respectively $\partial_{\theta_{1,k}} b(\theta_{0,1},x,\mu) |_{\mu = \bar \mu_t}$ and $\partial_{\theta_{2,k}} c(\theta_{0,2},x,\mu) |_{\mu = \bar \mu_t}$, whereas
$\bar \mu_t$ stands for $\bar \mu_t^{\theta_0}$.} On the other hand, \textbf{\ref{as7}} is a technical condition needed in order to obtain the first statement of Lemma \ref{l: conditional expectation}. We shed light to the fact that the bounds in Lemma \ref{l: conditional expectation} are stated for %the real value of the parameter 
$\theta_0$ and similarly we ask to \textbf{\ref{as7}} to be valid exclusively for the 
%actual value of the parameter, $\theta_{2,0}$. As a result, neither $\tilde{a}$ nor $K$ are reliant on the value of $\theta$. }
{
%We ask to \ref{as7} to be valid exclusively for the 
true parameter value $\theta_{0,2}$. Naturally, both $\tilde{a}$ and $K$ in \textbf{\ref{as7}} can be functions on $\Theta_2 \times \R^2$ with the first argument fixed at $\theta_{0,2}$. %which result in a diffusion coefficient that satisfies our other assumptions.
}
}

We also remark that, in the case where the unknown parameter $\theta$ appears only in the drift coefficient, there is no need to add a further assumption on the derivatives of the diffusion coefficient to estimate it, even if the diffusion coefficient still depends on the law of the process. 
%Moreover, it is clearly possible to prove the asymptotic normality of the joint estimator without requiring \ref{as7} in the case where only the drift depends on $\mu$. \\

\begin{example} \rm
A number of interacting particle models (and associated mean field equations) have been analyzed in the literature. We highlight a few here to illustrate the scope of our paper. \\
We start by considering some examples where the diffusion coefficient is a constant on a compact set that does not include the origin. This case has several applications (see (i) and (ii)). After that, some more general examples are presented. \\
\\
{
(i) The Kuramoto model is the most classical model for synchronization phenomena in large populations of coupled oscillators such as a clapping crowd, a population of fireflies or a system of neurons (see Section 5.2 of \cite{Review prop} and references therein). Let $N$ oscillators be defined by $N$ angles $X_t^{i,N}$, $i=1,\dots,N$ (defined modulo $2 \pi$, in this way they can actually be considered as elements of the circle), evolving in $t \in [0,T]$ according to}
$$d X_t^{i, N}=-\frac{\theta_{0, 1}}{N}\sum_{j=1}^N \sin \big(X_t^{i, N} - X_t^{j, N} \big) d t + \theta_{0,2}d W^{{i}}_t.$$
{This variant of the model satisfies our assumptions.\\ 
%Another application of our framework is given by the Kuramoto model, which is the most classical model for synchronization phenomena between populations of oscillators. It may be used for example to model a clapping crowd, a population of fireflies or a system of neurons (see Section 5.2 of \cite{Review prop} and references therein). Let $N$ oscillators be defined by $N$ angles $X_t^{i,N}$ (defined modulo $2 \pi$, in this way they can actually be considered as elements of the circle) which satisfy
%$$d X_t^{i, N}=-\frac{\theta_{0, 1}}{N}\sum_{j=1}^N \sin(X_t^{i, N} - X_t^{j, N}) d t + \theta_{0,2}d W^{{\color{blue}i}}_t.$$
%This model also satisfies our conditions. \\
\\
%(iv) A similar model has been considered in Wen et al. \cite{Wen} for ion diffusion with coulomb interaction between particles. The mean field equation is described as
%$$ d \bar{X}_t=\left(-\theta_{0,1,1} \bar{X}_t-\int_\R \frac{\theta_{0,1,2}}{(\bar{X}_t- x )^2} \bar{\mu}(d x) \right) d t + \theta_{0,2} d W_t.$$
%It is a three-parameter model with Lipschitz continuous drift function. \\
%\\
(ii) A popular model for opinion dynamics (see e.g.\ \cite{Cha17,65imp}) takes the form
$$d X_t^{i, N}=-\frac{1}{N}\sum_{j=1}^N  \varphi_{\theta_{0, 1}} \big( \big| X_t^{i, N} - X_t^{j, N} \big| \big) \big( X_t^{i, N} - X_t^{j, N} \big) d t + \theta_{0,2}d W^{{i}}_t 
$$
for $i=1,\dots,N$, $t \in [0,T]$, where $\varphi_{\theta_{0,1}} (x) := \theta_{0,1,1} \one_{[0,\theta_{0,1,2}]} (x)$, $x \in \R$, is the influence function which acts on the ``difference of opinions'' between agents. To have our regularity assumptions hold true in practice we can replace the function $\varphi_{\theta_{0,1}}$ by its infinitely differentiable approximation as it is done in Section 5.2 of \cite{Sharrock}. In \cite{Sharrock} we also note that the proxy of $\varphi_{\theta_{0,1}}$ depends non-linearly on the parameter $\theta_{0,1,2}$.\\
\\
(iii) Another example is
$$
d X^{i,N}_t = \Big( \theta_{0,1,1} + \frac{\theta_{0,1,2}}{N} \sum_{j=1}^N X^{j,N}_t - \theta_{0,1,3} X^{i,N}_t \Big) d t + \theta_{0,2} \sqrt{ 1+\big( X^{i,N}_t \big)^2 } d W^i_t
$$
for $i=1,\dots,N$, $t \in [0,T]$. We note that in the case $\theta_{0,1,2}=0$ the interacting particle system reduces to $N$ independent samples of a special case of the Pearson diffusion, which has applications in finance, see \cite{For08} and references therein.\\
\\
(iv) We consider the dynamic of the system
$$
d X^{i,N}_t = \Big( \theta_{0,1,1} + \frac{\theta_{0,1,2}}{N} \sum_{j=1}^N X^{j,N}_t - \theta_{0,1,3} X^{i,N}_t \Big) d t + \Big( \theta_{0,2,1} + \theta_{0,2,2} \sqrt{ \frac{1}{N} \sum_{j=1}^N \big( X^{j,N}_t \big)^2 } \Big) d W^i_t
$$
for $i=1,\dots,N$ in $t\in [0,T]$, where both the coefficients $b$ and $a$ depend on the law argument. We remark that the mean field limit of the above interacting particle system is a time-inhomogeneous Ornstein-Uhlenbeck process. See \cite{Kas90%,Sharrock
} for the case %where the diffusion coefficient is constant, i.e.\ 
$\theta_{0,1,1} = \theta_{0,2,2} = 0$. }
\\
\\
{Some remarks are in order. Example (iv), where $\theta_{0,2,2} = 0$, has been thoroughly discussed in Section 4.1 of \cite{Hof2}, specifically, with regard to the restrictions on $\mu_0$ and $\theta_{0,1}$ that ensure the latter parameter satisfies A5, A6. In examples (i), (iii) and (iv), where either $\theta_{0,2,1}$ or $\theta_{0,2,2}$ is set to $0$, 
%the diffusion coefficient is a multiplicative function of $\theta_{0,2}$, hence, 
it is obvious that A5, A6 hold for $\theta_{0,2} \neq 0$.
%this parameter. 
Finally, we note that in examples (i), (iii), and (iv), where either $\theta_{0,2,1}$ or $\theta_{0,2,2}$ is set to $0$, the drift and diffusion coefficients are respectively linear and multiplicative functions of $\theta$, which allows us to solve our estimator in closed form.
}
\end{example}

\section{Main results}{\label{s: main}}
Our main results demonstrate the consistency and the asymptotic normality of the estimator $\hat{\theta}_n^N$. 

\begin{theorem}{(Consistency)}
Assume that \textbf{\ref{as1}}-\textbf{\ref{as5}} hold, with only condition \textbf{(I)} in \textbf{\ref{as4}}. Then the estimator $\hat{\theta}_n^N$ is consistent in probability:
$$\hat{\theta}_n^N \xrightarrow{\mathbb{P}} \theta_0 \quad \mbox{as } n, N \rightarrow \infty.$$
\label{th: consistency}
\end{theorem}
\noindent
In order to obtain the asymptotic normality of our estimator we need to add an assumption on the relation between the rates $N$ and $\Delta_n$. In particular, we require that $N{\Delta_n} \rightarrow 0$ as $N, n \rightarrow \infty$.

\begin{theorem}{(Asymptotic normality)}
Assume that \textbf{\ref{as1}}-\textbf{\ref{as7}} hold. If $N {\Delta_n} \rightarrow 0$ then 
%we deduce 
$$\big(\sqrt{N}(\hat{\theta}_{n,1}^N - \theta_{0,1}), \sqrt{N %n
/ {\Delta_n}}(\hat{\theta}_{n,2}^N - \theta_{0,2}) \big)
%^{\textcolor{red}{\top}} 
\xrightarrow{\mathcal{L}} {\cal N}\big(0, 2 ( \Sigma(\theta_0) )^{-1}\big) \quad \mbox{as } n, N\rightarrow\infty,$$
where 
$$ 2 ( \Sigma(\theta_0) )^{-1} {:}= 
2 \operatorname{diag} \big( ( \Sigma^{(1)}(\theta_0) )^{-1}, ( \Sigma^{(2)}(\theta_0) )^{-1} \big)$$
with $\Sigma^{(j)} (\theta_0)${, $j=1,2$,} being defined in \textbf{\ref{as6}}. 
%with entries
%\begin{equation}
%\begin{cases}
% K_{hh} = \Big( \int_0^1 \int_\mathbb{R} \Big( \frac{\partial_{\theta_{1, h}} b(\theta_{0,1}, x, \bar{\mu}_t)}{a(\theta_{0,2},x, \bar{\mu}_t)} \Big)^2 \bar{\mu}_t(dx) dt \Big)^{-1}, \qquad 1\leq h\leq p_1 
 %\\[2.5 ex]
% K_{p_1 + \tilde{h}, p_1 + \Tilde{h}}= 3 \Big(\int_0^1 \int_\mathbb{R} \Big( \frac{\partial_{\theta_2, \tilde{h}}a(\theta_{0,2},x, \bar{\mu}_t)}{a(\theta_{0,2},x, \bar{\mu}_t)} \Big)^2 \bar{\mu}_t(dx) dt \Big)^{-1}, \qquad 1\leq \tilde{h} \leq p_2.
%\end{cases}
%\end{equation}
\label{th: normality}
\end{theorem}

\noindent
As common in the literature on contrast function based methods, understanding the asymptotic behaviour of $S^N_n (\theta_1,\theta_2)$ and its derivatives is key to obtain the statements of Theorems \ref{th: consistency} and \ref{th: normality}. In particular, we show that, under proper normalisation,  the first derivative of 
$S^N_n (\theta_1,\theta_2)$ 
converges to a Gaussian law with mean $0$ and covariance matrix $2 \Sigma(\theta_0)$ (see Proposition \ref{p: norm L}), while the second derivative converges in probability to the matrix $\Sigma(\theta_0)$ defined in 
\textbf{\ref{as6}} (see Proposition \ref{p: second derivatives contrast}). These results lead to the statement of Theorem \ref{th: normality}. 

The condition on {the rate, at which the discretization step $\Delta_n$ converges to $0$, has}
%relation between the two convergence rates have 
been discussed in detail in the framework of classical SDEs. 
{In this context, one disposes discrete observations of the trajectory of only one particle up to a time}
%Therein only the discrete trajectory of one particle is observed with discretization step $\Delta_n\to 0$ and terminal time 
$T:=n\Delta_n \to \infty$. In \cite{FloZmi} the corresponding condition was $T \Delta_n = n \Delta_n^2 \rightarrow 0$
as $n\to \infty$, which has been later improved to $n \Delta_n^3 \rightarrow 0$  in \cite{Yos92} thanks to a correction introduced in the contrast function. Finally, Kessler \cite{Kes97} proposed  a contrast function based on a Gaussian approximation of the transition density, which allowed him to consider a weaker condition $n \Delta_n^p \rightarrow 0$ for an arbitrary integer $p$. Similar developments have been made in the setting of classical SDEs with jumps in \cite{Sjs,Joint,GLM,Shi}. 

One may wonder if it possible to weaken the condition on the discretization step in the %setting of
{context of interacting} particle systems. { For a system  of independent copies of a diffusion process with random and/or fixed effects, \cite{Del18,Mix,Den20} require it in the same asymptotic framework as ours. %to prove asymptotic normality of the estimator of parameters. 
In \cite{Mix} also the rates of convergence of the estimators towards the parameters $\theta_1$ of the distribution of a random effect in the drift coefficient, and the fixed effect $\theta_2$ in the diffusion coefficient, are %respectively 
shown to be the same as ours.%$\sqrt{N}$ and $\sqrt{N/\Delta_n}$, which match our results.
} On the one hand, the condition $N {\Delta}_n \rightarrow 0$ %is needed in order to 
{allows us to} approximate the derivative of the contrast function {with} %via
a triangular array of martingale increments, as it is the case for classical SDEs. 
{For this step, higher order approximations, similar to those in \cite{Kes97}, could potentially help us relax this condition.}
%In this step higher order approximation similar to \cite{Kes97} can potentially help to obtain a weaker condition.  
On the other hand, 
%our condition on the discretization step is also required due to 
{we need it because of the} correlation between particles and higher order approximation{s} do %es 
not seem to solve this issue. Thus, we leave this investigation for future research.

% {\rev {\color{blue}(Can we delete it?)} The comment above helps us to realize that the role played by $T$ in the classical SDE case is now played by the number of particles $N$. Such comment is useful also to compare our convergence rates with the ones for classical SDEs. Indeed, the convergence rate for the estimation of the drift in the diffusion case is $\sqrt{T}$, while the one for the estimation of the diffusion parameter is $\sqrt{\frac{T}{\Delta_n}} = \sqrt{n}$, which is consists with our results.}

A recent paper \cite{Hof2} establishes the \textit{LAN property} for
drift estimation in $d$-dimensional McKean-Vlasov models under continuous observations and with diffusion {coefficient }%component 
being a function of $(t,\bar X_t)$ {only}. The authors show that the Fisher information matrix is given as
\begin{equation}{\label{eq: Fisher info}}
\Bigg(\int_0^T \int_{\R^d} \partial_{\theta_{{1,}k}} (c^{- \frac{1}{2}} b) (\theta_{0,1}, {t}, x, \bar \mu_t)^{{\top}} \partial_{\theta_{{1,}l}} (c^{- \frac{1}{2}} b) (\theta_{0,1}, {t}, x, \bar \mu_t) \bar \mu_t (dx) dt \Bigg)_{1 \le k, l \le p_{1}}
\end{equation}
(cf.\ \cite{Sharrock} where the diffusion coefficient is {an identity matrix}). This is consistent with our Theorem \ref{th: normality} when restricted to drift estimation. In other words, our drift estimator is asymptotically efficient. When considering joint estimation of the drift and diffusion coefficients, the LAN property has not yet been shown, although the results of Gobet \cite{Gobet 2002} in the classical diffusion setting give some hope. {Indeed, Gobet
\cite{Gobet 2002} has shown that for classical SDEs, in the ergodic case, the Fisher information for the drift parameter is given by $$(\Gamma_b^{\theta_0})_{k,l} = \int_{\R} \frac{\partial_{\theta_{1,k}} b(\theta_{0,1},x)\, \partial_{\theta_{1,l}} b(\theta_{0,1},x)}{c (\theta_{0,2},x)} \pi(dx)$$ 
for $k,l = 1, \dots, p_1,
%\in \{1, ... , p_1 \}
$ while the one for the diffusion parameter is given by $$(\Gamma_a^{\theta_0})_{k,l} = \int_{\R} \frac{\partial_{\theta_{2,k}} c(\theta_{0,2},x)\, \partial_{\theta_{2,l}} c(\theta_{0,2},x)}{c^2 (\theta_{0,2},x)} \pi(dx)$$ for $k,l 
 = 1, ... , p_2
%\in \{1, ... , p_2 \}
$, where $\pi$ is the invariant density associated to the diffusion. As $\Gamma_b^{\theta_0}$ modifies to \eqref{eq: Fisher info} for McKean-Vlasov models, one could expect that $\Gamma_a^{\theta_0}$ modifies to our asymptotic variance as well. This is left for further investigation.}

%It is interesting to remark that, in case one aims at estimating the drift parameter only, the previously procedure clearly still provides the result. Moreover, up to assume that the drift coefficient takes the form in \ref{as7}, it is possible to relax the condition on the rate of convergence of $N $ and $\Delta_n$. This is based on the use of Ito's formula on the drift coefficient in a similar way as we did for the diffusion coefficient in the proof of Point 1 Lemma \ref{l: conditional expectation}. A detailed analysis of this case is left as future perspective.\\ %Such condition derives from the correlation between the interacting particles, while the condition needed to obtain a triangular array of martingale increments is $\sqrt{N} \Delta_n \rightarrow 0$ (as it is for classical SDEs) and it is therefore less restrictive. \\
%Similarly, if one consider the case where the diffusion coefficient is constant, it is then possible to improve the estimations showed in the proof of Proposition \ref{p: norm L}. It will result in a less restrictive condition %the condition $N^2 \Delta_n^{\frac{5}{2}} \rightarrow 0$ 
%on the number of particles $N$ and on the discretization step $\Delta_n$.

{
\section{Numerical examples}{\label{s: num}}
{ We will now 
%present the results of a Monte Carlo analysis of 
examine the finite-sample performance of the introduced estimator $\hat \theta^N_n$ on two examples of interacting particle systems. 

%In both cases, we approximate the solution of each system using the Euler method with a step size of $\delta = 0.01$.
}
% We will now present two examples of parameter estimation in systems of interacting particles to illustrate our main results. The first example is a linear mean-field model, while the second is a stochastic opinion dynamics model. We will observe that in the linear mean-field model, it is feasible to express the estimator explicitly, whereas in the stochastic opinion dynamics model, the estimator $\hat{\theta}_n$ will be implicit and can only be found numerically. In both cases, we generate sample paths and apply a standard Euler-Maruyama scheme to implement the contrast function.

\subsection{Linear model}
%Linear mean field}

Consider an interacting particle system of the form:
\begin{equation}\label{ex:linearIPS}
		d X^{i,N}_t =  - \Big( \theta_{1,1} X^{i,N}_t + \frac{\theta_{1,2}}{N} \sum_{j= 1}^N (X^{i,N}_t - X^{j,N}_t) \Big) d t + \sqrt{\theta_{2}} d W^i_t,
\end{equation}
where $i = 1, ... , N$, $t\in [0, T]$, for some  $\theta_1 = (\theta_{1,1},\theta_{1,2}) \in \R^2$,  $\theta_{1,1} \neq 0$, $\theta_{1,1} + \theta_{1,2} \neq 0$, $\theta_{2} > 0$ {and $\int_{\R} x \mu_0 (d x) \neq 0$.}
In this model, the parameter $\theta_{1,1}$ determines the intensity of attraction of each individual particle towards zero, while $\theta_{1,2}$ governs the degree of {interaction, which is the} attraction of each individual particle towards the empirical mean. Notably, 
{for $\theta_{1,2} =0$, the processes $(X^{i,N}_t)_{t \in [0,T]}$, $i=1, \dots,N$, are independent.}
%when $\theta_{1,2}$ is set to zero, the system of interacting particles simplifies to $N$ independent samples of this process. 

Recall that for ${\theta_2} = 1$, estimation of the parameter ${\theta_1}$ from a continuous observation of the system  has been studied in \cite{Kas90, Sharrock}. Since the drift and squared diffusion coefficients in \eqref{ex:linearIPS} are linear in $\theta := (\theta_1,\theta_2)$, it is possible to find our estimator $\hat \theta^N_n$ in the closed form similarly as in \cite{Kas90,  Sharrock}:
\begin{equation}
	\hat \theta^N_{n,1,1} = \frac{A^N_n - B^N_n}{D^N_n - C^N_n}, \qquad \hat \theta^N_{n, 1,2} = \frac{A^N_n D^N_n - B^N_n C^N_n}{(C^N_n)^2 - C^N_n D^N_n},
\end{equation}
where 
\begin{gather*}
A^N_n := \frac{1}{N} \sum_{i=1}^N \sum_{j=1}^{n} (X^{i,N}_{t_{j-1,n}} - \bar X^N_{t_{j-1,n}})(X^{i,N}_{t_{j,n}} - X^{i,N}_{t_{j-1,n}}), \ B^N_n := \frac{1}{N} \sum_{i=1}^N \sum_{j=1}^n X^{i,N}_{t_{j-1,n}} (X^{i,N}_{t_{j,n}}-X^{i,N}_{t_{j-1,n}}), \\
C^N_n := \frac{\Delta_n}{N} \sum_{i=1}^N \sum_{j=1}^n (X^{i,N}_{t_{j-1,n}}- \bar X^{N}_{t_{j-1,n}})^2, \quad D^N_n := \frac{\Delta_n}{N} \sum_{i=1}^N \sum_{j=1}^n (X^{i,N}_{t_{j-1,n}})^2
\end{gather*}
with $\bar X^N_{t_{j-1,n}} := N^{-1} \sum_{k=1}^N X^{k,N}_{t_{j-1,n}}$, and then
\begin{equation}
\qquad \hat \theta^N_{n,2} = \frac{1}{NT} \sum_{i = 1}^N \sum_{j=1}^n \Big( X^{i,N}_{t_{j,n}} - X^{i,N}_{t_{j-1,n}} + \Delta_n \Big( \hat \theta^N_{n,1,1} X^{i,N}_{t_{j-1,n}} + \frac{\hat \theta^N_{n,1,2}}{N} \sum_{j= 1}^N (X^{i,N}_{t_{j-1,n}} - X^{j,N}_{t_{j-1,n}}) \Big) \Big)^2.
\end{equation}

{To illustrate the finite sample performance of $\hat \theta^N_n$, we choose $\theta = (\theta_{1,1}, \theta_{1,2},\theta_2) = (0.5,1,1)$ and $\mu_0 = \delta_1$ as in \cite{Sharrock}. We simulate $1000$ solutions of the system given by \eqref{ex:linearIPS} using the Euler method with a step size of $0.01$. We obtain observations of the system --- data sets for all 
possible combinations of $T=50,100$, $\Delta_n = 0.1, 0.05, 0.01$ and $N =50, 100$. 
%We note that the discretization step $\Delta_n = 0.01$ is equal to our chosen step size for the Euler method.
Table~\ref{tab4} presents the effect of $N$, $\Delta_n$, $T$ on the performance of $\hat \theta^N_n$. 
As $N$ or $T$ increases, the sample RMSE and bias of $\hat \theta^N_{n,1}$ decrease, whereas that of $\hat \theta^N_{n,2}$ do not change significantly. However, as $\Delta_n$ gets smaller, the performance of $\hat \theta_{n,2}$ improves, as well as that of $\hat \theta^N_{n,1,2}$.}

% \begin{table}[htbp]
% \centering
% \begin{tabular}{l c c c c}
% 	\hline
% 	$(N,T) =$ & 
% 	$(50,50)$ & 
% 	$(100,50)$ &
% 	$(50,100)$ &
% 	$(100,100)$ \\
% 	\hline
% 	$\theta_{1,1}$ & 0.50 (0.10)& 0.50 (0.08)& 0.50 (0.08)& 0.50 (0.07)\\
% 	$\theta_{1,2}$ & 0.90 (0.11)& 0.90 (0.08)& 0.90 (0.09)& 0.90 (0.07)\\
% 	$\theta_2$ & 0.88 (0.01)& 0.88 (0.01)& 0.88 (0.01)&0.88 (0.00)\\
% \end{tabular}
%   \caption{Sample mean and standard deviation (in brackets) of $\hat \theta^N_n$ for $\theta = (0.5, 1, 1)$, $\Delta_n = 0.1$ and different values of $N,T$. The number of replications is $1000$.
%   }
% \label{tab1}
% \end{table}

\begin{table}[htbp]
\centering
\begin{tabular}{l rrr rrr rrr rr}
	\hline
%     $(N,T)=$ &&
% 	\multicolumn{2}{c}{$(50,50)$} && \multicolumn{2}{c}{$(100,50)$} &&
% 	\multicolumn{2}{c}{$(50,100)$} &&
% 	\multicolumn{2}{c}{$(100,100)$} \\
	\multicolumn{1}{r}{$N=$} &
	\multicolumn{2}{c}{50} && \multicolumn{2}{c}{100} &&
	\multicolumn{2}{c}{50} &&
	\multicolumn{2}{c}{100} \\
	\multicolumn{1}{r}{$(\Delta_n, T)=$} &
	\multicolumn{2}{c}{$(0.1,50)$} && \multicolumn{2}{c}{$(0.1,50)$} &&
	\multicolumn{2}{c}{$(0.1,100)$} &&
	\multicolumn{2}{c}{$(0.1,100)$} \\
	\hline 
%	&&&&&&&&&&\\
	$\hat \theta^N_{n,1,1}$ & 0.10 & (0.00) && 0.08 & (0.00) && 0.08 & (0.00) && 0.07 & (0.00) \\
	$\hat \theta^N_{n,1,2}$ & 0.15 & (-0.10) && 0.13 & (-0.10) && 0.13 & (-0.10) && 0.12 & (-0.10)\\
	$\hat \theta^N_{n,2}$ & 0.12 & (-0.12) && 0.12 & (-0.12) && 0.12 & (-0.12) &&0.12 & (-0.12)\\
	&&&&&&&&&&\\
	\multicolumn{1}{r}{$(\Delta_n, T)=$} &
	\multicolumn{2}{c}{$(0.05,50)$} && \multicolumn{2}{c}{$(0.05,50)$} &&
	\multicolumn{2}{c}{$(0.05,100)$} &&
	\multicolumn{2}{c}{$(0.05,100)$} \\
	\hline		
%	&&&&&&&&&&\\
	$\hat \theta^N_{n,1,1}$ & 0.10 & (0.01) && 0.08 & (0.01) && 0.08 & (0.01) && 0.07 & (0.00) \\
	$\hat \theta^N_{n,1,2}$ & 0.12 & (-0.05) && 0.10 & (-0.05) && 0.10 & (-0.05) && 0.09 & (-0.05)\\
	$\hat \theta^N_{n,2}$ & 0.06 & (-0.06) && 0.06 & (-0.06) && 0.06 & (-0.06) &&0.06 & (-0.06)\\	
	&&&&&&&&&&\\
	\multicolumn{1}{r}{$(\Delta_n, T)=$} &
	\multicolumn{2}{c}{$(0.01,50)$} && \multicolumn{2}{c}{$(0.01,50)$} &&
	\multicolumn{2}{c}{$(0.01,100)$} &&
	\multicolumn{2}{c}{$(0.01,100)$} \\
	\hline		
%	&&&&&&&&&&\\
	$\hat \theta^N_{n,1,1}$ & 0.11 & (0.01) && 0.08 & (0.01) && 0.09 & (0.01) && 0.07 & (0.01) \\
	$\hat \theta^N_{n,1,2}$ & 0.11 & (-0.02) && 0.09 & (-0.01) && 0.09 & (-0.01) && 0.07 & (-0.01)\\
	$\hat \theta^N_{n,2}$ & 0.00 & (0.00) && 0.00 & (0.00) && 0.00 & (0.00) &&0.00 & (0.00)
\end{tabular}
  \caption{Sample RMSE (and bias in brackets) of $\hat \theta^N_n$ for $\theta = (0.5, 1, 1)$ %, $\Delta_n = 0.1$ 
  and different values of $N$, $\Delta_n$, $T$.
  The number of replications is $1000$.
  }
\label{tab2}
\end{table}

{We note that the numerical results presented above for $\Delta_n = 0.01$ can be viewed as the maximum likelihood estimation. Indeed, our contrast function up to a negative constant is the log-likelihood function for the Euler approximation with the same step $\Delta_n$.}
{
%It should be noted that the results presented above for $\Delta_n = 0.01$ can be viewed as the Maximum Likelihood Estimate (MLE) obtained through the use of an Euler method with a step size of 0.01 in our simulations. By choosing $\Delta_n$ to be equal to the step size, we are essentially assuming continuous observations heuristically. 
%Hence, one could not do {\color{blue}much} better than the estimation provided in the last lines of Table \ref{tab2} above. 
{Therefore, it is difficult to improve upon the estimation provided in the last lines of Table \ref{tab2}. Interestingly,} %Notably,  
the performance of our estimator for $\Delta_n = 0.1$ and $\Delta_n = 
0.05$ is quite similar to that of $\Delta_n = 0.01$, particularly with respect to the RMSE for the estimation of $\hat \theta^N_{n,1,1}$ and $\hat \theta^N_{n,1,2}$. } \\
\\
One possible application of our Theorem \ref{th: normality} is to test the hypothesis of noninteraction of particles similarly as in \cite{Kas90}. Consider the null hypothesis $H_0 : \theta_{1,2} = 0$ and the alternative $H_1 : \theta_{1,2}  \neq 0$. 
%Then, $H_0$ asserts that the processes $(X^{i,N}_t)_{t \in [0,T]}$, $i=1,\dots,N$, are independent. 
%By 
{According to} Theorem \ref{th: normality}, if $N \Delta_n \to 0$, then
$$
\sqrt{N} (\hat \theta^N_{n,1,2} - \theta_{1,2}) \xrightarrow{\mathcal{L}} {\cal N}(0, V(\theta)),
$$
%\text{ with }
%\quad \text{as } n, N\rightarrow\infty,
%$$
%where 
%$$
%with 
{where} 
$$
V (\theta) := 2 \Sigma^{(1)}_{11} (\theta) / (\Sigma^{(1)}_{11}(\theta) \Sigma^{(1)}_{22}(\theta) - \Sigma^{(1)}_{12}(\theta)\Sigma^{(1)}_{21}(\theta)),
$$
{and}
%where 
for all $i,j = 1,2$,
\begin{align*}
\Sigma^{(1)}_{ij} (\theta) &:= 
\begin{cases}
\displaystyle 2 \theta_2^{-1} \int_0^T \int_{\R} x^2 \bar \mu_t (d x) d t, &i=j=1,\\
\displaystyle 2  \theta_2^{-1} \int_0^T \int_{\R}  \Big( x - \int_{\R} y \bar \mu_t (d y) \Big)^2 \bar \mu_t (d x) d t, &\text{else},
\end{cases}
%= 2 \theta_2^{-1} \int_0^T \Big( \int_{\R} x^2 \bar \mu^\theta_t (d x) - \Big( \int_{\R} x \bar \mu^\theta_t ( d x ) \Big)^2 \Big) d t\\
%&= 2 \theta_2^{-1} \int_0^T \int_{\R} x \Big( x - \int_\R y \bar \mu^\theta_t (d y) \Big) \bar \mu^\theta_t (d x) d t =: \Sigma^{(1)}_{12} (\theta) = \Sigma^{(1)}_{21} (\theta).
\end{align*}
can be %found 
{explicitly computed in terms of the model} %system 
parameters, see %similarly as in 
\cite{Kas90, Sharrock}. {By using } Lemma \ref{l: Riemann} and Theorem \ref{th: consistency}, { we have that}
$$
{V^N_n := } \hat \theta^N_{n,2} D^N_n/((D^N_n -C^N_n)C^N_n) %=: V^N_n 
\xrightarrow{\P} V(\theta) \quad \text{as } n, N\rightarrow\infty.
$$
{Therefore, if $N \Delta_n \to 0$, under $H_0$, we can conclude that} %We conclude that under $H_0$, if $N \Delta_n \to 0$, then
$$
{Z^N_n := } \hat \theta^N_{n,1,2}  \sqrt{N/V^N_n} \xrightarrow{\mathcal{L}} {\cal N}(0, 1) \quad \text{as } n, N\rightarrow\infty.
$$
{Thus, we reject $H_0$ if} %Then $H_0$ is rejected if 
$$
|{Z^N_n} %\hat \theta^N_{n,1,2}  \sqrt{N/V^N_n} 
| > z_{\alpha/2},
$$
where $\alpha \in (0,1)$ is the chosen level of significance and $z_{\alpha}$ denotes the $\alpha$-quantile of the standard normal distribution.

{Next, we examine the performance of the test statistic $Z^N_n$. We simulate $1000$ solutions of the system given by \eqref{ex:linearIPS} with $\mu_0=\delta_1$, using the Euler method with a step size of $0.01$. Table \ref{tab3} reports the rejection rates of $H_0$ in favor of $H_1$ at a significance level of $\alpha=5\%$ using $Z^N_n$ for all possible combinations of $N, T=50, 100$, $\Delta_n = 0.1$ and $\theta=(0.5, \theta_{1,2},1)$, where $\theta_{1,2} = 0, 0.1, 0.25, 0.5,$ or $1$. The empirical size is quite well observed. Rejection rates of incorrect $H_0$ increase with increasing $\theta_{1,2}$ or $N$ and $T$.}

\begin{table}[htbp]
\centering
\begin{tabular}{l r r r r r}
	\hline
	$\theta_{1,2}$ &$(N,T)=$ & 
	$(50,50)$ & 
	$(100,50)$ &
	$(50,100)$ &
	$(100,100)$ \\
	\hline
	0 && 4.8 & 4.6 & 4.2 & 4.1\\
	0.1 && 17.8 & 22.5 & 21.4 & 28.9\\
	0.25 && 61.3 & 78.2 & 75.6 & 87.0\\
	0.5 && 97.2 & 99.7 & 99.8 & 99.9\\
	1 && 100.0 & 100.0 & 100.0 & 100.0\\
\end{tabular}
  \caption{Rejection rates (in $\%$) of $H_0 : \theta_{1,2} = 0$ vs. $H_1 : \theta_{1,2} \neq 0$ at level $\alpha = 5 \%$ with $Z^N_n$ for $\theta = (0.5, \theta_{1,2}, 1)$, $\Delta_n = 0.1$ and different values of $N,T$.
  %and $(N,\Delta_n)$.
  The number of replications is $1000$.
  }
\label{tab3}
\end{table}

% \begin{table}[htbp]
% \centering
% \begin{tabular}{l r r r r r}
% 	\hline
% 	$\theta_{1,2}$ &$(N,T)=$ & 
% 	$(50,50)$ & 
% 	$(100,50)$ &
% 	$(50,100)$ &
% 	$(100,100)$ \\
% 	\hline
% 	0 && 5.3 & 5.4 & 4.7 & 5.0\\
% 	0.1 && 13.9 & 12.1 & 17.2 & 19.8\\
% 	0.25 && 43.5 & 51.4 & 65.7 & 71.1\\
% 	1 && 99.9 & 100.0 & 100.0 & 100.0\\
% \end{tabular}
%   \caption{Rejection rates (in $\%$) of $H_0 : \theta_{1,2} = 0$ vs. $H_1 : \theta_{1,2} \neq 0$ at level $\alpha = 5 \%$ with $Z^N_n$ for $\theta = (0.5, \theta_{1,2}, 2)$, $\Delta_n = 0.1$ and different values of $N,T$.
%   %and $(N,\Delta_n)$.
%   The number of replications is $1000$.
%   }
% \end{table}

{
\subsection{Stochastic opinion dynamics model}
%Stochastic opinion dynamics}

% We consider now a stochastic opinion dynamics model in terms of system of interacting particles, which is given by 
We now consider an interacting particle system that can model opinion dynamics:
\begin{equation}\label{ex:opinionIPS}
dX_t^{i,N} =
%- \frac{{\color{blue}\theta_{1,2}}}{N} \sum_{j = 1}^N \exp(- \frac{0.01}{1 - (|X_t^{i,N} - X_t^{j,N}|- {\color{blue}\theta_{1,1}})^2}) {\color{blue}\one ( (|X_t^{i,N} - X_t^{j,N}|-\theta_{1,1})^2 < 1)} 
- \frac{1}{N} \sum_{j=1}^N {\varphi_{\theta_1} ( |X^{i,N}_t - X^{j,N}_t | )}
( X_t^{i,N} - X_t^{j,N} ) dt + \sqrt{\theta_2} dW_t^i,
% $$dX_t^{i,N} = - \frac{{\color{blue}\theta_{1,2}}}{N} \sum_{j = 1}^N 
% \exp(- \frac{0.01}{1 - (|X_t^{i,N} - X_t^{j,N}|- {\color{blue}\theta_{1,1}})^2}) {\color{blue}\one ( (|X_t^{i,N} - X_t^{j,N}|-\theta_{1,1})^2 < 1)} (X_t^{i,N} - X_t^{j,N}) dt + \sqrt{\theta_2} dW_t^i,
\end{equation}
where $i = 1, ... , N$, $t\in [0, T]$, and 
$$
\varphi_{\theta_1} (x) := \theta_{1,2} \exp \Big( - \frac{0.01}{1-(x-\theta_{1,1})^2} \Big) \one_{[\theta_{1,1}-1, \theta_{1,1}+1]} (x), \quad x \in \R,
$$
for some  
%$\theta_1 = (\theta_{1,1},\theta_{1,2}) \in \R^2$, 
$-1 < \theta_{1,1} \le 1$, $\theta_{1,2} > 0$, $\theta_{2} > 0$.
The %exponential 
interaction kernel $\varphi_{\theta_1}(x)$ %above is referred to as interaction kernel, and 
provides an infinitely differentiable approximation to the scaled indicator function $\theta_{1,2} \one_{[0, \theta_{1,1} + 1]} (x)$, $x\ge 0$. We interpret that $\theta_{1,1}$ governs the intensity of attraction of each individual particle 
%each one of the $N$ particles moves 
towards the scaled empirical mean of all the others within a distance $\theta_{1,1}+1$. The position of each particle represents its opinion, and over time, the opinions of particles merge into metastable "soft clusters". For further information on this stochastic opinion dynamics model, see \cite{Sharrock} and references therein.
% One way to understand the two parameters, $\theta_{1,1}$ and $\theta_{1,2}$, is to interpret the first as a scaling factor that governs the intensity of the attraction between particles, while the second represents a range parameter that defines the maximum distance between particles within which interactions can occur.

% This type of model appears in various fields, such as biology and social sciences, where the interaction kernel plays a crucial role in determining how the behavior of one particle can affect that of others. For instance, in a social context, the interaction kernel could represent how an individual's opinion may impact the opinions of others in a network. \\
% Deterministic models of this type have been extensively studied, and it is well known that, in the long run, the particles tend to merge into clusters. The number of these clusters depends on both the interaction kernel (i.e., the strength and range of the interactions) and the initial conditions. In contrast, stochastic models with random noise do not form exact clusters, but rather converge into metastable `soft clusters', which can be observed even in the absence of explicit clustering mechanisms. For further information on these models, see \cite{13Imp, 65imp} and related references. \\
% \\
%With this notation, we estimate the parameters $(\theta_{1,1},\theta_{1,2}, \theta_2)$ by minimization of the contrast function in \eqref{eq:def contrast}. \\

Note that
the squared diffusion coefficient is a multiplicative function of 
%in 
$\theta_2$ 
which enables us to express $\hat \theta^N_{n,2}$ in terms of $%\hat \theta^N_{n,1}=
(\hat \theta^N_{n,1,1},\hat \theta^N_{n,1,2})$. However, the latter estimator
%estimator $\hat \theta^N_{n,1}$ 
is implicit and can only be found using a numerical method. 
To illustrate the performance of $\hat \theta^N_n = (\hat \theta^N_{n,1,1},\hat \theta^N_{n,1,2}, \hat \theta^N_{n,2})$
%For the simulations, 
we choose the parameter $\theta = (\theta_{1,1}, \theta_{1,2}, \theta_2) = (-0.5,2,0.04)$ as in \cite{Sharrock}, and the initial distribution $\mu_0 = \mathcal{N}(0,1)$ for each individual particle. We simulate $1000$ solutions of the system given by \eqref{ex:opinionIPS} using the Euler method with a step size of $0.01$. We obtain $1000$ data sets for $\Delta_n = 0.1$ and all possible combinations of $N,T = 50, 100$ as in the previous subsection. Table~\ref{tab4} presents the effect of $N$, $T$ on the performance of $\hat \theta^N_n$. As $N$ increases, the sample RMSE and bias of $\hat \theta^N_n$ decrease, whereas they do not change that much with increasing $T$. We can also see that $\hat \theta^N_{n,1,1}$ is more accurate than $\hat \theta^N_{n,1,2}$.
%We estimate the bias and the standard deviation of our estimator using a Monte Carlo method based on $1000$ replications. 

% Results in Table \ref{t: } show that in this scenario, our proposed estimator exhibits excellent performance, as indicated by the significantly small bias and standard deviation. This observation suggests that the numerical results obtained from our estimator align with the theoretical predictions. Therefore, the estimator we have proposed, which minimizes the contrast function, performs exceedingly well. 

\begin{table}[htbp]
\centering
\begin{tabular}{l rrr rrr rrr rr}
	\hline
    $(N,T)=$ &
	\multicolumn{2}{c}{$(50,50)$} && \multicolumn{2}{c}{$(100,50)$} &&
	\multicolumn{2}{c}{$(50,100)$} &&
	\multicolumn{2}{c}{$(100,100)$} \\
	\hline
	$\hat \theta^N_{n,1,1}$ & 0.0340 & (0.0159) && 0.0263 & (0.0145) && 0.0280 & (0.0154) && 0.0206 & (0.0137) \\
	$\hat \theta^N_{n,1,2}$ & 0.1652 &(-0.1378) && 0.1503 & (-0.1347) && 0.1526 & (-0.1420) && 0.1472 & (-0.1416) \\
	$\hat \theta^N_{n,2}$ & 0.0027 & (-0.0026) && 0.0026 & (-0.0025) && 0.0033 & (-0.0032) && 0.0033 & (-0.0033) \\
\end{tabular}
  \caption{Sample RMSE (and bias in brackets) of $\hat \theta^N_n$ for $\theta =(-0.5,2,0.04)$, $\Delta_n = 0.1$ and different values of $N,T$.
  The number of replications is $1000$.
  }
\label{tab4}
\end{table}
}

% \begin{table}[htbp]
% \centering
% \begin{tabular}{l r r r r}
% 	\hline
%     $(N,T)=$ & 
% 	$(50,50)$ & 
% 	$(100,50)$ &
% 	$(50,100)$ &
% 	$(100,100)$ \\
% 	\hline
% 	$\theta_{1,1}$ & $-0.48$ $(0.03)$ & $-0.49$ $(0.02)$ & $-0.48$ $(0.02)$ & $-0.49$ $(0.02)$\\
% 	$\theta_{1,2}$ & $1.86$ $(0.09)$ & $1.87$ $(0.07)$ & $1.86$ $(0.06)$ & $1.86$ $(0.04)$\\
% 	$\theta_2$ & $0.04$ $(0.00)$ & $0.04$ $(0.00)$ & $0.04$ $(0.00)$ & $0.04$ $(0.00)$\\
% \end{tabular}
%   \caption{Sample mean and standard deviation (in brackets) of $\hat \theta^N_n$ for $\theta =(-0.5,2,0.04)$, $\Delta_n = 0.1$ and different values of $N,T$.
%   The number of replications is $1000$.
%   }
% \label{tab5}
% \end{table}

}

\section{Technical lemmas}{\label{s: tec}}
Before proving the main statistical results stated in previous section, we need to introduce some additional notations and to state some lemmas which will be useful in the sequel. 

Define $\mathcal{F}_t^N := \sigma \{(W_u^k)_{u \in [0,t]}, \, X_0^{%{\color{blue}\theta,}
k, N}; \, k= 1, ... , N \}$ and $\E_t [\cdot] := \E[\cdot | \mathcal{F}^N_t]$. For a set $(Y^{i,N}_{t,n})$ of random variables and $\delta \ge 0$, the notation 
$$
Y^{i,N}_{t,n} = R_t^i (\Delta_n^\delta)
$$ 
means that $Y^{i,N}_{t,n}$ is $\mathcal{F}_t^N$-measurable and the set $(Y_{t,n}^{i,N}/\Delta_n^\delta)$ is 
bounded in $L^q$ for all $q \ge 1$, uniformly in $t, i, n, N$. That is \
$$
\E \big[ \big|Y_{t,n}^{i,N}/\Delta_n^\delta \big|^q \big]^{1/q}  \le C_q
$$ 
for all $t, i, n, N$, $q \ge 1$.

%In the following, for $\delta \ge 0$, we will denote by $R_t(\Delta_n^\delta )$ any random variable which is $\mathcal{F}_{t}^{\color{blue}N} = \sigma \{(W_u^i)_{u \in [0,t]}, \, X_0^{i, N}; \, i= 1, ... , N \}$ measurable and such that, for any $q \ge 1$, 
%\begin{equation}
 %\quad  \left \| \frac{R_t(%\theta, 
 %\Delta_n^\delta )}{\Delta_n^\delta } \right \|_{L^q} \le C_{\color{blue}q}.
%\label{eq: definition R}
%\end{equation}
%$R_t$ represent the term of rest and have the following useful property, consequence of the just given definition:
%\begin{equation}
%R_t(\Delta_{n}^\delta)= \Delta_{n}^\delta R_t({\color{blue}1}%\Delta_{n}^0
%).
%\label{propriety power R}
%\end{equation}
%We point out that it does not involve the linearity of $R_t$, since the random variables $R_t$ on the left and on the right side are not necessarily the same but only two on which the control (\ref{eq: definition R}) holds with $\Delta_{n}^\delta $ and $\Delta_{n}^0 {\color{blue}=1}$, respectively. \\ \\
We will repeatedly use some moment inequalities gathered in the following lemma. 
\begin{lemma}
%Suppose that the process $X^{\theta,i,N}$ satisfies 
Assume \textbf{\ref{as1}-\ref{as2}}. Then, for all $p \ge 1$, $0 \le s < t \le T$ such that $t-s \le 1$, $i \in \{ 1, ..., N \}$, $N \in \mathbb{N}$, the following %estimations 
hold true.
\begin{enumerate}
    \item $\sup_{t \in [0,T]} \E[|X_t^{%\theta, 
    i, N}|^p] < C$,
    %For all $p \ge 1$ 
    %there exists $C_p > 0$ such that, 
    %for all $\theta \in \Theta$, 
    %for all $i \in \{ 1, ..., N \}$ %and for all $n \in \mathbb{N}$ 
    %$\sup_{\color{blue}t \in [0,T]} \E[|X_t^{\theta, i, N}|^p] < C$.
    moreover, $\sup_{t \in [0,T]} \E[W_p^q(\mu_t^{%{\color{blue}\theta,}
    N}, \delta_0)] < C$ for $p \le q$.
    \item %For all $p \ge 2$ and $0< s < t$ it is 
    $\E[|X_t^{%\theta, 
    i, N} - X_s^{%\theta, 
    i, N}|^p] \le C (t - s)^{\frac{p}{2}}$.
    \item %For all $p \ge 2$ and $0< s < t$ it is 
    $\E_s [|X_t^{%\theta, 
    i, N} - X_s^{%\theta, 
    i, N}|^p %\mathcal{F}_{s} 
    ] 
    \le C (t - s)^\frac{p}{2}R_s^i(%\theta,
    1)$. 
    \item %For all $p \ge 1$ and $0< s < t$ it is 
    $\E[W_2^p(\mu_t^{%\theta,
    N}, \mu_s^{%\theta,
    N})] \le C (t - s)^{\frac{p}{2}}$.
    \item %For all $p \ge 1$ and $0< s < t$ it is 
    $\E_s [W_2^p(\mu_t^{%\theta,
    N}, \mu_s^{%\theta,
    N}) | %\mathcal{F}_{s}
    ] \le C (t - s)^{\frac{p}{2}}R_s (%\theta, 
    1)$.
\end{enumerate}
\label{l: moments}
\end{lemma}
The asymptotic properties of the estimator are deduced by the asymptotic behaviour of our contrast function. To study it, the following lemma will be useful.

\begin{lemma}
Assume \textbf{\ref{as1}-\ref{as2}}. 
Let $f : \R \times {\cal P}_l \to \R$ satisfy for some $C>0$, $k,l =0,1,\dots$ and all $(x,\mu), (y,\nu) \in \R \times {\cal P}_l$,
%$x,y\in \R$, $\mu,\nu \in {\cal P}_l$,
%$(x,\mu), (x',\mu') \in \R \times {\cal P}_l$,
%Let $f : \R \times {\cal P}_l \to \R$ satisfy the following condition for some $C>0$, $k,l \in \mathbb{N}$: for all $(x,\mu), (x',\mu') \in \R \times {\cal P}_l$,
\begin{equation}\label{cond:Riemann}
|f(x,\mu)-f(y,\nu)| \le C (|x-y| + W_2(\mu,\nu))(1 + |x|^k + |y|^k + W_l^l(\mu,\delta_0) + W_l^l(\nu,\delta_0)).%,
\end{equation}
%for all $(x,\mu), (x',\mu') \in \R \times {\cal P}_l$. 
Moreover, let
%moreover, 
the mapping $(x,t) \mapsto f(x,\bar \mu%^{\theta}
_t)$ be integrable with respect to $\bar \mu%^\theta
_t (d x) d t$ over $\R \times [0,T]$.
%$$
%\int_0^1 \int_{\R} |f(x,\bar \mu^{\theta}_t)| \bar \mu^{\theta}_t (d x) d t < \infty.
%$$
Then 
$$
\frac{\Delta_n}{N} \sum_{i=1}^N \sum_{j=1}^n f(X^{%\theta,
i,N}_{t_{j-1,n}},\mu^{%\theta,
N}_{t_{j-1,n}}) \xrightarrow{\mathbb{P}} \int_0^T \int_{\R} f(x,\bar \mu%^{\theta}
_t) \bar \mu%^{\theta}
_t (d x) d t \quad \text{as } n,N \to \infty.
$$
\label{l: Riemann}
\end{lemma}

{ It is worth underlining that the boundedness of the moments and the convergence of the Riemann sums, which are obtained almost for free in the classical SDE case, are  more complex in our setting. In particular, the proof of Lemma \ref{l: Riemann} consists now in three steps, the first deals with the convergence of the proper Riemann sums, in the second step we move from the %system of 
interacting particle system to the iid 
system though the propagation of chaos property, while the third step is an application of the law of large numbers.\\
Another challenge compared to the classical SDE case is gathered in next lemma. Indeed,}
our main results heavily rely on the study of derivatives of our contrast function and so on the moment %s 
{bounds} of its numerator. 
{To accomplish this, we need to use It\^o's lemma on the squared diffusion coefficient as a function of the particle system's state. Therefore, we must}
{
%In order to obtain a good estimate on the derivatives of the contrast with respect to the diffusion parameter, we will need to use Ito's formula on $a$, with respect to both the process and the measure. Hence, we have to 
understand how to express derivatives of %the diffusion coefficient with respect to $\mu$. 
{$a$ with respect to the measure argument.} That is the purpose of the extra hypothesis \textbf{\ref{as7}}, thanks to which the problem reduces to study the derivatives of $K$.}\\ We recall that, in the sequel, we will denote by $c(\theta_2, x, \mu)$ the %function 
value $a^2(\theta_2, x, \mu)$.

\begin{lemma}
Assume \textbf{\ref{as1}-\ref{as2}}.
%Suppose that Assumptions \ref{as1}-\ref{as4} (I) hold. 
Then, the following hold true. 
\begin{enumerate}
    \item If also \textbf{\ref{as7}} is satisfied, then\\ 
    $\mathbb{E}_{t_{j,n}} [(X_{t_{j+1,n}}^{%\theta,
    i, N} - X_{t_{j,n}}^{%\theta,
    i, N} - \Delta_n b(\theta_{0,1}%1%{0, 1}
    , X_{t_{j,n}}^{%\theta,
    i, N}, \mu_{t_{j,n}}^{%\theta,
    N}) )^2] = \Delta_{n} c(\theta_{0,2}%2%{0, 2}
    , X_{t_{j,n}}^{%\theta,
    i, N},\mu_{t_{j,n}}^{%\theta,
    N}) + R_{t_{j,n}}^i(%\theta_0, 
    \Delta_{n}^{2}).$
    \item $\mathbb{E}_{t_{j,n}} [(X_{t_{{j+1},n}}^{%\theta,
    i, N} - X_{t_{j,n}}^{%\theta,
    i, N} - \Delta_n b(\theta_{0,1}, X_{t_{j,n}}^{%\theta,
    i, N}, \mu_{t_{j,n}}^{%\theta,
    N}) )^4] = 3 \Delta_{n}^2 c^2(\theta_{0, 2}
    , X_{t_{j,n}}^{i, N},\mu_{t_{j,n}}^{%\theta,
    N}) + R_{t_{j,n}}^i(%\theta, 
    \Delta_{n}^{\frac{5}{2}}).$
    \item $|\mathbb{E}_{t_{j,n}} [ X_{t_{j+1,n}}^{%\theta,
    i, N} - X_{t_{j,n}}^{%\theta,
    i, N} - \Delta_n b(\theta_{0,1}, X_{t_{j,n}}^{%\theta,
    i, N}, \mu_{t_{j,n}}^{%\theta,
    N}) ] | %\le
    = R_{t_{j,n}}^i(%\theta, 
    \Delta_{n}^{\frac{3}{2}}).$
   % \item $\mathbb{E}_j[(X_{t_{j+1}}^{i, N} - X_{t_{j}}^{i, N} - \Delta_n b(\theta_1, X_{t_{j}}^{i, N}, \mu_{t_j}^N) )^3] \le R_{t_j}(\theta, \Delta_{n}^{2}).$
    %\item Moreover, for $p \ge 2$, it is $\mathbb{E}_j[(X_{t_{j+1}}^{i, N} - X_{t_{j}}^{i, N} - \Delta_n b(\theta_1, X_{t_{j}}^{i, N}, \mu_{t_j}^N) )^p] \le R_{t_j}(\theta, \Delta_{n}^{\frac{p}{2}}).$
\end{enumerate}
\label{l: conditional expectation}
\end{lemma}
%We remark that the first and the second points here above are some particular cases of the last one, studied more in detail, where we have identified the main contribution. %The third and fourth points provide better estimation than the last one. This is due to the fact that in the odd moments we do not have the contribution of the Brownian part anymore, which is the one that provides the error of size $\Delta_n^{\frac{p}{2}}$ in the last point. \\
{ We underline that \textbf{\ref{as7}} is needed in order to prove that the size of the remainder function in the first point is $\Delta_n^2$. Without it, the size of the rest function would have been $\Delta_n^\frac{3}{2}$, which would not have been enough to obtain the asymptotic normality as in Proposition \ref{p: norm L} (see the proof of \eqref{e: mtg array}).}
The proof of the lemmas stated in this section can be found in Section \ref{s: proof technical}.

\section{Proofs}{\label{s: proof main}}
	
\subsection{Consistency}		
	
%Let us omit notation for the dependence on $\theta_0$. In particular, write $\P$ for $\P^{\theta_0}$, $X^N_t = (X^{1,N}_t, \dots, X^{N,N}_t)$ for $X^{\theta_0,N}_t = (X^{\theta_0,1,N}_t, \dots, X^{\theta_0,N,N}_t)$, $\mu^N_t$ for $\mu^{\theta_0,N}_t$. 

%Our proof follows that of Theorem 1 in \cite{Kes97}. 

Let us prove the (asymptotic) consistency of $\hat \theta^N_n = (\hat \theta^N_{n,1},\hat \theta^N_{n,2})$ component-wise. Our approach is similar to that taken in the proof of \cite[Theorem 5.7]{van98}. 
In particular, we consider a criterion function $\theta \mapsto S^N_n (\theta)$ 
as a random element taking values in $(C(\Theta;\R), \| \cdot \|_\infty)$. The uniform convergence of criterion functions is proved in the following lemma.

\begin{lemma}\label{lemma:cns}
%Denote by $\bar \mu_t$ the law of $\bar X_t = \bar X^{\theta_0}_t$ for $t \in [0,1]$ and by $\P$ the underlying probability measure $\P^{\theta_0}$. Then as 
Assume \textbf{\ref{as1}}-\textbf{\ref{as3}}, \textbf{{\ref{as4}(I)}}, \textbf{\ref{as5}}. {Then a}s $N,n \to \infty$,
\begin{align}\label{lim:con2}
&\sup_{(\theta_1,\theta_2) \in \Theta} \Big| \frac{\Delta_n}{N} S^N_n (\theta_1,\theta_2) - J (\theta_2) \Big| \xrightarrow{\mathbb{P}} 0,\\ \label{lim:con1}
&\sup_{(\theta_1,\theta_2) \in \Theta} \Big| \frac{1}{N} (S^N_n (\theta_1,\theta_2)-S^N_n (\theta_{0, 1},\theta_2)) - I (\theta_1,\theta_2) \Big| \xrightarrow{\mathbb{P}} 0,    
\end{align}
where the functions $I, J$ are defined in \eqref{def:I}, \eqref{def:J} respectively.
\end{lemma}

\begin{proof}
It suffices to show the following steps:
\begin{enumerate}
\item $\frac{\Delta_n}{N} S^N_n (\theta_1,\theta_2) \xrightarrow{\mathbb{P}} J (\theta_2)$ for every $(\theta_1,\theta_2) \in \Theta$.
\item The sequence $(\theta_1,\theta_2) \mapsto \frac{\Delta_n}{N} S^N_n(\theta_1, \theta_2)$ is tight in $(C(\Theta;\R), \| \cdot \|_\infty)$.
\item $\frac{1}{N} (S^N_n (\theta_1,\theta_2)-S^N_n (\theta_{0,1},\theta_2)) \xrightarrow{\mathbb{P}} I (\theta_1,\theta_2)$ for every $(\theta_1,\theta_2) \in \Theta$,
\item The sequence $(\theta_1,\theta_2) \mapsto \frac{1}{N} (S^N_n (\theta_1,\theta_2)-S^N_n (\theta_{0,1},\theta_2))$ is tight in $(C(\Theta;\R), \| \cdot\|_\infty)$.
\end{enumerate}

\noindent
Let us omit the notation for dependence on $N,n$, in particular, write $X^i_t$ for $X^{i,N}_t$, $\mu_t$ for $\mu^N_t$, $t_j$ for $t_{j,n}$.
%In order to lighten the notation, we will prove this lemma for $p_1= p_2 =1$. However, its proof naturally extends to the multidimensional parameters case. 
Denote $f(\cdot, X^i_t, \mu_t)$ by $f^i_t (\cdot)$ for a function $f$, for example equal to %$h,g$ or $d$ 
$h$ or $g$ defined as
%equal to $a$, $a^2$, $b$, $d$, $g$, where  
\begin{equation}\label{def:gd}
h (\theta, x,\mu) = \frac{%d^2(x, \mu, \theta_1)
(b(\theta_{0,1}, x, \mu) - b(\theta_1, x, \mu))^2}{c (\theta_2, x, \mu)}, \qquad g(\theta, x, \mu) = \frac{%d(x, \mu, \theta_1)
b(\theta_{0,1}, x, \mu) - b(\theta_1, x, \mu)}{c (\theta_2, x, \mu)}%, \qquad d (x, \mu, \theta_1) := b(x, \mu, \theta_{0,1}) - b(x, \mu, \theta_1)
\end{equation}
for all $\theta = (\theta_1,\theta_2) \in \Theta_1 \times \Theta_2 = \Theta$, $x \in \R$, $\mu \in {\cal P}_2$. \\

\noindent 
$\bullet$ Step 3. We start proving that for every $\theta = (\theta_1,\theta_2) \in \Theta_1 \times \Theta_2 = \Theta$,
$$
\frac{1}{N} (S^N_n(\theta_1,\theta_2) - S^N_n(\theta_{0,1},\theta_2)) \xrightarrow{\P} I(\theta) = \int_0^{T} \int_{\R} h(\theta, x,\bar \mu_t) \bar \mu_t (d x) d t.
$$
Let us first decompose the left hand side as a sum of a main term and remainder. %Use
%$$
%X^i_{t_{j-1}} - X^i_{t_{j-1}} - \Delta b^i_{t_{j-1}}(\theta_1) = R^i_j + \Delta %%d^i_{t_{j-1}}(\theta_1)
%(b^i_{t_{j-1}}(\theta_{0,1}) - b^i_{t_{j-1}}(\theta_1)),
%$$ 
%where 
We have
$$
S^N_n(\theta_1,\theta_2) = \sum_{i=1}^N \sum_{j=1}^n \frac{({H}^i_j + \Delta_n (b^i_{t_{j-1}}(\theta_{0,1}) - b^i_{t_{j-1}}(\theta_1)) %d^i_{t_{j-1}}(\theta_1) 
)^2}{\Delta_n c^i_{t_{j-1}} (\theta_2) } + (\log c)^i_{t_{j-1}}(\theta_2),
$$
%where note $d^i_{t_{j-1}}(\theta_{0,1}) = 0$ for all $i,j$, hence,
where
${H}^i_j = X^i_{t_{j-1}} - X^i_{t_{j-1}} - \Delta_n b^i_{t_{j-1}}(\theta_{0,1})$
for all $i, j$. We decompose
\begin{equation}\label{cfdeComp}
\frac{1}{N} (S^N_n(\theta_1,\theta_2) - S^N_n(\theta_{0,1},\theta_2)) = I^N_n (\theta) + 2 \rho^N_n (\theta),
\end{equation}
where
\begin{equation}\label{def:INnrhoNn}
I^N_n (\theta) = \frac{\Delta_n}{N} \sum_{i=1}^N \sum_{j=1}^n
h^i_{t_{j-1}}(\theta), 
\qquad \rho^N_n (\theta) = \frac{1}{N} \sum_{i=1}^N \sum_{j=1}^n
g^i_{t_{j-1}} (\theta) {H}^i_j.
\end{equation}
Then
$$
I^N_n (\theta) \xrightarrow{\P} I(\theta)
$$
follows from Lemma~\ref{l: Riemann} if the function $h(\theta,\cdot)$
%the function in the integral $I(\theta)$ 
is locally Lipschitz with polynomial growth. To check this assumption we note
that the functions $
%d(\cdot, \theta_1) = 
b(\theta_{0,1},\cdot) - b(\theta_1,\cdot)$, $a(\theta_2,\cdot)$ are Lipschitz continuous and have linear growth by \textbf{\ref{as2}}. We also recall that $\inf_{x,\mu} c (\theta_2,x,\mu) > 0$ by \textbf{\ref{as3}}. Hence, $h(\theta,\cdot)$ %in $I(\theta)$
satisfies the assumption of Lemma~\ref{l: Riemann}.

%if
%$\inf_{x,\mu} a^2 (x,\mu) > 0$ by A4 and $a(\cdot,\theta_2)$, $b(\cdot,\theta_1)$ are Lipschitz continuous for every $\theta_2 \in \Theta_2$, $\theta_1 \in \Theta_1$ respectively by A3. Also, Lipschitz continuity implies linear growth of $a(\cdot,\theta_2)$, $b(\cdot,\theta_1)$ for every $\theta_2 \in \Theta_2$, $\theta_1 \in \Theta_1$ respectively. 
%Indeed, under the above-stated assumptions and in view of
%$$
%\frac{x_1}{y_1} - \frac{x_2}{y_2} = \frac{x_1-x_2}{y_1} + \frac{x_2 (y_2-y_1)}{y_1y_2}
%$$
%and $x_1^2-x_2^2 = (x_1-x_2)(x_1+x_2)$ for all $x_k,y_k>0$, $k=1,2$, 
%conditions of the technical lemma are satisfied.

It remains to show that 
\begin{equation}\label{lim:rhoNn}
\rho^N_n (\theta) \xrightarrow{\P} 0.
\end{equation}
With 
%Using $g(\cdot,\theta)$ defined in \eqref{def:gd} and 
${H}^i_j = B^i_j + A^i_j$, where
$$
B^i_j = \int_{t_{j-1}}^{t_j} (b^i_s(\theta_{0,1}) - b^i_{t_{j-1}} (\theta_{0,1})) d s, \qquad A^i_j = \int_{t_{j-1}}^{t_j} a^i_s (\theta_{0,2}) d W^i_s,
$$
for all $i,j$, let us further decompose
%For all $i,j$ with
%$$
%B^i_j = \int_{t_{j-1}}^{t_j} (b^i_s(\theta_{0,1}) - b^i_{t_{j-1}} (\theta_{0,1})) d s, \qquad A^i_j = \int_{t_{j-1}}^{t_j} a^i_s (\theta_{0,2}) d W^i_s,
%$$
%use
%$R^i_j = B^i_j + A^i_j$ and $g(\cdot,\theta)$ defined in \eqref{def:gd} to get
\begin{equation}\label{rhoNndeComp}
\rho^N_n (\theta) = \rho^N_{n,1} (\theta) + \rho^N_{n,2} (\theta),
\end{equation}
where
%correspondingly
%$\rho^N_n (\theta)$ as a sum of 
$$
\rho^N_{n,1} (\theta)= \frac{1}{N} \sum_{i=1}^N \sum_{j=1}^n g^i_{t_{j-1}}(\theta) B^i_j, \qquad
\rho^N_{n,2} (\theta) = \frac{1}{N} \sum_{i=1}^N \sum_{j=1}^n g_{t_{j-1}}^i (\theta) A^i_j.
$$
%where 
%$$
%g(x,\mu,\theta) = \frac{d(x,\mu,\theta_1)}{a^2(x,\mu,\theta_2)}
%$$ 
%for all $(x,\mu,\theta) \in \R \times {\cal P} \times \Theta$ as before. 
%correspondingly. 
It is enough to show that
\begin{equation}\label{lim:rhoNnp}
\rho^N_{n,k} (\theta) \xrightarrow{L^k} 0, \qquad k=1,2.
\end{equation}
First, let us show \eqref{lim:rhoNnp} in case $k=2$. Note that for all $i_1 = i_2$ and $j_1 \neq j_2$, 
\begin{equation}\label{covgA}
\E [ g^{i_1}_{t_{j_1-1}}(\theta) A^{i_1}_{j_1} g^{i_2}_{t_{j_2-1}}(\theta) A^{i_2}_{j_2} ] = 0
\end{equation}
follows from $\E_{t_{j_1-1}} [A^{i_1}_{j_1}] = 0$, whereas independence of Brownian motions implies \eqref{covgA} for all $i_1 \neq i_2$ and $j_1, j_2$.  We conclude that
\begin{equation}\label{varrhoNn2}
\E [(\rho^N_{n,2} (\theta))^2] = \frac{1}{N^2} \sum_{i=1}^N \sum_{j=1}^n \E [ (g^i_{t_{j-1}}(\theta) A^i_j)^2 ]. 
\end{equation}
Next, the It\^o isometry gives 
$$
\E [ (g^i_{t_{j-1}}(\theta) A^i_j)^2 ] = \int_{t_{j-1}}^{t_j} \E [ (g^2)^i_{t_{j-1}}(\theta) c^i_s(\theta_{0,2}) ] d s,% = O(\Delta)
$$
where $\E [ (g^2)^i_{t_{j-1}}(\theta) c^i_s(\theta_{0,2}) ] = O(1)$ 
%whence we can see that $\E [ (g^i_{t_{j-1}}(\theta) A^i_j)^2 ] = O(\Delta)$
uniformly in $t_{j-1} \le s \le t_j, j,i$ thanks to  $\inf_{x,\mu}c(\theta_2,x,\mu)>0$ by \textbf{\ref{as3}}, linear growth of $a(\theta_{0,2},\cdot)$, $b(\theta_1,\cdot)$ %for every $\theta_1$ 
by \textbf{\ref{as2}} %, which follow from Lipschitz continuity, 
and moment bounds in  Lemma~\ref{l: moments}(1). We conclude that $\E [ (g^i_{t_{j-1}}(\theta) A^i_j)^2 ] = O(\Delta_n)$ uniformly in $i,j$, which in turn implies 
$$
\E [(\rho^N_{n,2}(\theta))^2] = O \big(N^{-1} \big).
$$
Finally, let us show \eqref{lim:rhoNnp} in case $k=1$. For this purpose, use
$$
\E[ |g^i_{t_{j-1}}(\theta) B^i_j| ] \le \int_{t_{j-1}}^{t_j} \E [|g^i_{t_{j-1}}(\theta) (b^i_s(\theta_{0,1}) - b^i_{t_{j-1}}(\theta_{0,1}))|] d s
$$ 
and then the Cauchy–Schwarz inequality. %As before,
Note $\E [ (g^2)^i_{t_{j-1}} (\theta) ] = O(1)$ uniformly in $j,i$ follows in the same way
%from $\inf_{x,\mu}a^2(x,\mu,\theta_2)>0$ by A4, linear growth of $b(\cdot,\theta_1)$ %for every $\theta_1$ 
%by A3 and moment bounds in part 1.\ of Lemma~\ref{l: moments}, 
as 
%so does $\E [ (g^2)^i_{t_{j-1}} (\theta) (a^2)^i_{s} (\theta_{0,2}) ] = O(1)$ uniformly in $t_{j-1} \le s \le t_j, j,i$ 
above. Lipschitz continuity of  $b(\theta_{0,1},\cdot)$ by \textbf{\ref{as2}} and moment bounds in  Lemma~\ref{l: moments}(2) and (4)
%established {\color{blue}(to do)} in the technical lemma 
imply $\E [ (b^i_s(\theta_{0,1}) - b^i_{t_{j-1}}(\theta_{0,1}))^2 ] = O(\Delta_n)$ uniformly in $t_{j-1} \le s \le t_j, j, i$. We conclude that 
$$
\E [|\rho^N_{n,1}|] = O(\Delta_n^{\frac{1}{2}}).
$$ 
This completes the proof of Step 3. \\

\noindent
$\bullet$ %Proof of Lemma~\ref{lemma:cns} 
Step 4. 
Recall the decomposition \eqref{cfdeComp}, \eqref{rhoNndeComp}.
%We have that
%$$
%S^N_n (\theta_1,\theta_2) - S^N_n (\theta_{1,0},\theta_2) = \sum_{i=1}^N \sum_{j=1}^n (r^i_j (\theta_{1,0}) + r^i_j(\theta_1)) g^i_{t_{j-1}} (\theta),
%$$
%where 
%$$
%r^i_j (\theta_1) := X^i_{t_j}-X^i_{t_{j-1}} - \Delta b^i_{t_{j-1}} (\theta_1) = r^i_{0,j} + r^i_{1,j}(\theta_1) + r^i_{2,j}
%$$ 
%with
%\begin{align*}
%r^i_{0,j} := \int_{t_{j-1}}^{t_j} b^i_s (\theta_{1,0}) d s - \Delta b^i_{t_{j-1}} (\theta_{1,0}),\qquad
%r^i_{1,j}(\theta_1) := \Delta d^i_{t_{j-1}}(\theta_1),\qquad
%r^i_{2,j} := \int_{t_{j-1}}^{t_j} a^i_s (\theta_{2,0}) d W^i_s.
%\end{align*}
%Note that $r^i_{1,j}(\theta_{1,0}) = 0$. 
%It is enough to show tightness of 
%$$
%\theta \mapsto s^N_{1,n} (\theta) := \frac{1}{N} \sum_{i=1}^N \sum_{j=1}^n r^i_{1,j} (\theta_1) 
%g^i_{t_{j-1}} (\theta), \qquad
%\theta \mapsto s^N_{k,n} (\theta) := \frac{1}{N} \sum_{i=1}^N \sum_{j=1}^n r^i_{k,j} g^i_{t_{j-1}} (\theta), \qquad k=0,2.
%$$
It is enough to show tightness of 
$$
\theta \mapsto I^N_n (\theta), \qquad
\theta \mapsto \rho^N_{n,k} (\theta), \qquad k=1,2.
$$
Our approach to showing tightness of both sequences are based upon \cite[Theorem 14.5]{Kall}. We need to show that for all $N,n$:
\begin{equation}\label{ineq:tight1}
\E \big[ \sup_{\theta} \| \nabla_\theta I^N_n (\theta) \| \big] \le C, \qquad \E \big[ \sup_{\theta} \| \nabla_\theta  \rho^N_{n,1} (\theta) \| \big] \le C.
\end{equation}
%$$
%\E \Big[ \sup_{\theta} \|\nabla_\theta s^N_{k,n}(\theta) \| \Big] \le C < \infty, \qquad k = 0,1.
%$$
%In case $k=1$ it follows from 
The above bounds follow if for all $N,n$, and $i,j,t_{j-1} \le s \le t_j$,
\begin{equation}\label{ineq:tight10}
\E \big[ \sup_\theta \| \nabla_\theta h^i_{t_{j-1}}(\theta) \| \big] \le C,
%$$
%\E \Big[ \sup_\theta \Big| \nabla_\theta \frac{(d^2)^i_{t_{j-1}}(\theta_1)}{(a^2)^i_{t_{j-1}}(\theta_2)} \Big|  \Big] \le C
%$$
%that is
%$$
%\E \Big[ \sup_\theta \Big| \frac{\partial_{\theta_1} (d^2)^i_{t_{j-1}}(\theta_1)}{(a^2)^i_{t_{j-1}}(\theta_2)} \Big|  \Big] \le C, \qquad 
%\E \Big[ \sup_\theta \| \frac{(d^2)^i_{t_{j-1}}(\theta_1) \partial_{\theta_2} (a^2)^i_{t_{j-1}}(\theta_2)}{ (a^4)^i_{t_{j-1}}(\theta_2)} \|  \Big] \le C,
%$$
%$$
%\E \Big[ \sup_\theta | \partial_{\theta_1} d^i_{t_{j-1}}(\theta_1) g^i_{t_{j-1}}(\theta) |  \Big] \le C, \qquad 
%\E \Big[ \sup_\theta \| d^i_{t_{j-1}}(\theta_1) \nabla_\theta g^i_{t_{j-1}}(\theta) \|  \Big] \le C,
%$$
\qquad 
\E \big[ | b^i_s(\theta_{0,1}) | \sup_\theta \|\nabla_\theta g^i_{t_{j-1}}(\theta) \|  \big] \le C,
\end{equation}
%where
%$h(x,\mu,\theta) = d^2 (x,\mu, \theta_1)/a^2 (x,\mu, \theta_2)$, $(x,\mu, \theta) \in \R \times {\cal P}_1 \times \Theta$, and $g, d$ have been 
where $h,g : \Theta \times \R \times {\cal P}_2\to \R$ are defined by \eqref{def:gd}. In $\nabla_{\theta_k} h, \nabla_{\theta_k} g$, $k=1,2$,
we note $\nabla_{\theta_1} (b(\theta_{0,1},\cdot) - b(\theta_1,\cdot)) = - \nabla_{\theta_1} b(\theta_1,\cdot)$. Moreover, by the mean value theorem, $|b(\theta_{0,1},\cdot) - b(\theta_1,\cdot)| \le C \sup_{\theta_1} \|\nabla_{\theta_1} b(\theta_1,\cdot)\|$ for all $\theta_1 \in \Theta_1$, since $\Theta_1$ is convex, bounded.
%%Consider $\sup_\theta \| \nabla_\theta h (\cdot, \theta)\|$, $\sup_\theta \| \nabla_\theta g(\cdot, \theta) \|$. 
%by the mean value theorem, $b(\cdot,\theta_{0,1}) - b(\cdot, \theta_1) = (\theta_{0,1}-\theta_1) \int_0^1 \partial_{\theta_1} b(\cdot, \theta_1 + u (\theta_{0,1}-\theta_1)) d u$ for all $ \theta_1 \in \Theta_1$, where $\Theta_1$ is %%bounded, convex, 
%a closed interval,
%which implies
%$|b(\cdot,\theta_{0,1}) - b(\cdot, \theta_1)| \le C \sup_{\theta_1} |\partial_{\theta_1} b(\cdot,\theta_1)|$, moreover, we have $\partial_{\theta_1} (b(\cdot,\theta_{0,1}) - b(\cdot, \theta_1)) = - \partial_{\theta_1} b(\cdot, \theta_1)$ for all $\theta_1$. 
Additionally using $\inf_{\theta_2,x,\mu} {c}(\theta_2,x,\mu) >0$ by \textbf{\ref{as3}}, we get
$$
\|\nabla_{\theta_1} g(\theta,\cdot)\| \le C \sup_{\theta_1} \|\nabla_{\theta_1} b(\theta_1,\cdot) \|, \qquad \|\nabla_{\theta_2} g (\theta,\cdot)\| \le C \sup_{\theta_1} \| \nabla_{\theta_1} b(\theta_1,\cdot) \| \sup_{\theta_2} \|\nabla_{\theta_2} a(\theta_2,\cdot)\|,
$$ 
and
$$
\|\nabla_{\theta_1} h (\theta,\cdot)\| \le C \sup_{\theta_1} \|\nabla_{\theta_1} b(\theta_1,\cdot) \|^2, \qquad \|\nabla_{\theta_2} h (\theta,\cdot)\| \le C \sup_{\theta_1} \| \nabla_{\theta_1} b(\theta_1,\cdot)\|^2 \sup_{\theta_2} \|\nabla_{\theta_2} a(\theta_2,\cdot)\|
$$
for all $\theta$. We have the polynomial growth of $\sup_{\theta_1} \|\nabla_{\theta_1}b(\theta_1,\cdot)\|$, $\sup_{\theta_2} \|\nabla_{\theta_2} a(\theta_2,\cdot)\|$ thanks to assumption \textbf{\ref{as4}} and linear growth of $b(\theta_{0,1},\cdot)$ thanks to \textbf{\ref{as2}}. 
The Cauchy-Schwarz inequality and
moment bounds in Lemma~\ref{l: moments}(1) yield \eqref{ineq:tight10} and so \eqref{ineq:tight1}.  
%follow 
%desired result follows because 
%$\sup_t \E [ |X^i_t|^p ] \le C_p < \infty$ for all $N, i$ and $p>0$. Moreover, for $p  \le  q$, we have $W_p^q ( \mu_t,\delta_0) \le ( \frac{1}{N} \sum_{i=1}^N |X^i_t|^p )^{\frac{q}{p}}$, hence, using the Minkowski inequality,
%$$
%\E [ W_p^q ( \mu_t,\delta_0 ) ] \le  \Big( \frac{1}{N} \sum_{i=1}^N \big( \E [ | X^i_t |^q ] \big)^{\frac{p}{q}} \Big)^{\frac{q}{p}} \le C_q < \infty
%$$
%for all $t$.

Following the approach of \cite[Theorem 20 in Appendix 1]{IbrHas}, we want to show that  for all $N,n$ and $\theta, \theta' \in \Theta$,
%$$
%\E [ | s^N_{2,n}(\theta) |^2 ] \le C, \qquad \E [ |s^N_{2,n}(\theta)- s^N_{2,n}(\theta')|^2 ] \le C \| \theta - \theta' \|^2_2.
%$$
$$
\E [ | \rho^N_{n,2}(\theta) |^2 ] \le C, \qquad \E [ |\rho^N_{n,2}(\theta) - \rho^N_{n,2}(\theta')|^2 ] \le C \| \theta - \theta' \|^2_2.
$$
%Note 
%$\rho^N_{n,2}(\theta) = 0$ with $\theta_1 = \theta_{0,1}$.
%%$s^N_{2,n}(\theta) = 0$ with $\theta_2 = \theta_{0,2}$.
We note that the second relation implies the first one because $\rho^N_{n,2}(\theta) = 0$ with $\theta_1 = \theta_{0,1}$ and $\Theta_2$ is bounded. 
%Let ${\cal F}_t := \sigma (X^i_0, W^{i}_s, 0 \le s \le t, 1\le i\le N)$. Since
%$$
%\E [  A^i_j g^i_{t_{j-1}} (\theta) | {\cal F}_{t_{j-1}} ] = g^i_{t_{j-1}} (\theta) \E [ A^i_j | {\cal F}_{t_{j-1}} ] = 0,
%$$
%%$$
%%\E [  r^i_{2,j} g^i_{t_{j-1}} (\theta) | {\cal F}_{t_{j-1}} ] = g^i_{t_{j-1}} (\theta) \E [ r^i_{2,j} | {\cal F}_{t_{j-1}} ] = 0,
%%$$
%we see that 
%%$\sum_{i=1}^N r^i_{2,j} (g^i_{t_{j-1}} (\theta)-g^i_{t_{j-1}} (\theta'))$, 
%$\sum_{i=1}^N A^i_j (g^i_{t_{j-1}} (\theta)-g^i_{t_{j-1}} (\theta'))$,
%$1 \le j \le n$, form a martingale difference sequence. 
In the same way as in \eqref{varrhoNn2} we get 
%$$
%\E [ |\rho^N_{n,2}(\theta)-\rho^N_{n,2}(\theta') |^2 ] = \frac{1}{N^2} \sum_{j=1}^n \E \Big[ \Big| \sum_{i=1}^N A^i_j (g^i_{t_{j-1}}(\theta) - g^i_{t_{j-1}} (\theta'))  \Big|^2 \Big],
%$$
%%$$
%%\E [ |s^N_{2,n}(\theta)-s^N_{2,n}(\theta') |^2 ] = \frac{1}{N^2} \sum_{j=1}^n \E \Big[ \Big| \sum_{i=1}^N r^i_{2,j}(g^i_{t_{j-1}}(\theta) - g^i_{t_{j-1}} (\theta'))  \Big|^2 \Big],
%%$$
$$
\E [ |\rho^N_{n,2}(\theta)-\rho^N_{n,2}(\theta') |^2 ] = \frac{1}{N^2} \sum_{i=1}^N \sum_{j=1}^n \E [ | (g^i_{t_{j-1}}(\theta) - g^i_{t_{j-1}} (\theta')) A^i_j |^2 ],
$$
where the It{\^o} isometry gives
$$
\E [ |  (g^i_{t_{j-1}}(\theta)-g^i_{t_{j-1}}(\theta')) A^i_j |^2 ] = \int_{t_{j-1}}^{t_j} \E [  (g^i_{t_{j-1}}(\theta)-g^i_{t_{j-1}}(\theta'))^2 c^i_s (\theta_{0,2}) ] d s. 
$$
%$$
%\E [ | r^i_{2,j}  (g^i_{t_{j-1}}(\theta)-g^i_{t_{j-1}}(\theta')) |^2 ] = \int_{t_{j-1}}^{t_j} \E [ (a^2)^i_s (\theta_{2,0}) (g^i_{t_{j-1}}(\theta)-g^i_{t_{j-1}}(\theta'))^2] d s, 
%$$
By the mean value theorem, 
$$
|g(\theta,\cdot)-g(\theta',\cdot)| \le \| \theta - \theta' \| \sup_{\theta} \| \nabla_\theta g(\theta,\cdot) \| 
$$
%$$
%|g^i_{t_{j-1}}(\theta)-g^i_{t_{j-1}}(\theta')| \le \sup_{\theta \in \Theta} \| \nabla_\theta g^i_{t_{j-1}} (\theta) \| \| \theta - \theta' \|
%$$
since $\Theta$ is convex. Then 
$$
\E \big[ \sup_\theta \| \nabla_\theta g^i_{t_{j-1}} (\theta) \|^2 c^i_s (\theta_{0,2}) \big] \le C
$$ 
for all $t_{j-1} \le s \le t_j, j, i$ and $N,n$ follows in a similar way as the second bound in \eqref{ineq:tight10} does using, in addition, linear growth of $a(\theta_{0,2},\cdot)$, which follows from its Lipschitz continuity by \textbf{\ref{as2}}. \\

\noindent
$\bullet$ %Proof of Lemma \ref{lemma:cns} 
Step 1. 
We want to prove that for every $\theta \in \Theta$,
\begin{equation}\label{lim:intJ}
\frac{\Delta_n}{N} S^N_n (\theta) \xrightarrow{\P} J (\theta_2) = \int_0^{{T}} \int_{\R} f (\theta_2,x,\bar \mu_t) \bar \mu_t(d x) d t,
\end{equation}
where 
$$
f (\theta_2, x, \mu) = \frac{c (\theta_{0,2},x,\mu)}{c (\theta_2,x,\mu)} + \log c (\theta_2,x,\mu)
$$
for every $(\theta_2,x,\mu) \in \Theta_2 \times \R \times {\cal P}_2$. For this purpose, in $\Delta_n S^N_n(\theta)$ let us decompose every term as 
\begin{equation}\label{def:rij}
\frac{(X^i_{t_j} - X^i_{t_{j-1}} - \Delta_n b^i_{t_{j-1}} (\theta_1))^2}{c^i_{t_{j-1}} (\theta_2)} + \Delta_n (\log c)^i_{t_{j-1}} (\theta_2) = \Delta_n f^i_{t_{j-1}}(\theta_2)+ r^i_j.
\end{equation}
We can decompose $r^i_j$ further with 
\begin{equation}\label{eq:incrX}
X^i_{t_j} - X^i_{t_{j-1}} - \Delta_n b^i_{t_{j-1}} (\theta_1) = B^i_j (\theta_1) + A^i_j,    
\end{equation}
where 
\begin{equation}\label{eq:incrXBA}
B^i_j (\theta_1) = \int_{t_{j-1}}^{t_j} b^i_s(\theta_{0,1}) d s - \Delta_n b^i_{t_{j-1}} (\theta_1), \qquad A^i_j = \int_{t_{j-1}}^{t_j} a^i_s (\theta_{0,2}) d W^i_s,
\end{equation}
note
$$
\E_{t_{j-1}} [ (A^i_j)^2 ] = \int^{t_j}_{t_{j-1}} c^i_s (\theta_{0,2}) d s.%,
$$
We get
%\begin{equation}\label{eq:incrX}
%X^i_{t_j} - X^i_{t_{j-1}} - \Delta b^i_{t_{j-1}} (\theta_1) = B^i_j (\theta_1) + A^i_j,    
%\end{equation}
%where 
%\begin{equation}\label{eq:incrXBA}
%B^i_j (\theta_1) := \int_{t_{j-1}}^{t_j} b^i_s(\theta_{0,1}) d s - \Delta b^i_{t_{j-1}} (\theta_1), \qquad A^i_j := \int_{t_{j-1}}^{t_j} a^i_s (\theta_{0,2}) d W^i_s,
%\end{equation}
%note
%$$
%\E_{t_{j-1}} [ (A^i_j)^2 ] = \int^{t_j}_{t_{j-1}} (a^2)^i_s (\theta_{0,2}) d s.%,
%$$
%%where $\E_t [\cdot]$ denotes $\E [ \cdot | {\cal F}_t ]$ and ${\cal F}_t = \sigma(X^i_0, W^i_s, 0 \le s \le t, 1\le i\le N)$ for every $t, N$. 
%Accordingly, in $\Delta S^N_n(\theta)$ split every term
%\begin{align*}
%\frac{(X^i_{t_j} - X^i_{t_{j-1}} - \Delta b^i_{t_{j-1}} (\theta_1))^2}{(a^2)^i_{t_{j-1}} (\theta_2)} + \Delta (\log a^2)^i_{t_{j-1}} (\theta_2) = \Delta f^i_{t_{j-1}}(\theta_2)+ \sum_{k=0}^2 r^i_{j,k},
%\end{align*}
%where
\begin{equation}\label{def:rijk}
r^i_j = \sum_{k=0}^2 r^i_{j,k}, \qquad \text{where } r^i_{j,k} = \frac{H^i_{j,k}}{c^i_{t_{j-1}} (\theta_2)}, \qquad k=0,1,2,
\end{equation}
and 
\begin{align*}
&H^i_{j,2} = (A^i_j)^2 - \E_{t_{j-1}} [ (A^i_j)^2 ],\qquad H^i_{j,1} = 2 A^i_j B^i_j(\theta_1) + (B^i_j (\theta_1))^2,\\
&\qquad H^i_{j,0} = \E_{t_{j-1}} [ (A^i_j)^2 ] - \Delta_n c^i_{t_{j-1}} (\theta_{0,2}). 
\end{align*}
Our proof of \eqref{lim:intJ} consists of the following steps: 
\begin{equation}
\frac{\Delta_n}{N} \sum_{i=1}^N \sum_{j=1}^n f^i_{t_{j-1}} (\theta_2) \xrightarrow{\P} J(\theta_2), \qquad \frac{1}{N} \sum_{i=1}^N \sum_{j=1}^n r^i_{j,k} \xrightarrow{L^1} 0, \quad k=0,1,2.\label{cns:vol}
\end{equation}
Let us start from the convergence in  \eqref{cns:vol} for $k=2$. It is enough to show that 
$\sup_i \E [ ( \sum_j r^i_{j,2} )^2] = o(1)$. 
We note that $\E [ r^i_{j_1,2} r^i_{j_2,2} ] = 0$, $j_1 \neq j_2$, since $\E_{t_{j-1}} [ r^i_{j,2}] = 0$. We are left to show that $\sup_i \sum_j \E [ (r^i_{j,2})^2 ] = o(1)$.
Thanks to assumption \textbf{\ref{as3}} %(more specifically, $\inf_{x,\mu} a^2 (x,\mu,\theta_2)> 0$) 
it reduces to showing $\sup_i \sum_j \E [ (H^i_{j,2})^2 ] = o(1)$,
where $\E_{t_{j-1}} [ (H^i_{j,2})^2 ] = \E_{t_{j-1}} [ (A^i_j)^4 ] -  (\E_{t_{j-1}} [ (A^i_j)^2 ])^2$ leads to $\E [ (H^i_{j,2})^2 ] \le \E [ (A^i_j)^4 ]$ for all $i,j$. Furthermore, by the Burkholder-Davis-Gundy inequality and Jensen's inequality,
\begin{equation}\label{ineq:A4}
\E [ (A^i_j)^4 ] \le C \E \Big[ \Big( \int_{t_{j-1}}^{t_j} c^i_s (\theta_{0,2}) d s \Big)^2 \Big] \le C \Delta_n \int_{t_{j-1}}^{t_j} \E [ (c^2)^i_s (\theta_{0,2}) ] d s = O(\Delta_n^2)
\end{equation}
uniformly in $i,j$, where the last relation follows thanks to linear growth of $a(\theta_{0,2},\cdot)$ by \textbf{\ref{as2}} and moment bounds in Lemma~\ref{l: moments}(1). %,
%$| a(x,\mu,\theta_{2,0}) | \le C (1 + |x| + W_1(\mu,\delta_0))$ for all $x,\mu$, 
%which in turn follows from assumption A3 about its Lipschitz continuity. 
We conclude that $\sup_{i,j} \E [(R^i_{j,2})^2] = O (\Delta_n^2)$.

We now turn to the convergence in \eqref{cns:vol} for $k=1$. It is enough to show that $n \sup_{i,j}  \E [ | r^i_{j,1} | ] = o(1)$. Assumption 
\textbf{\ref{as3}} %, i.e.\ $\inf_{x,\mu} a^2 (x,\mu,\theta_2) > 0$, 
implies $\E [|r^i_{j,1}|] \le C \E [ |H^i_{j,1}| ]$ for all $i,j$, where $\sup_{i,j} \E [ ( A^i_j )^2 ] = O (\Delta_n)$ follows from \eqref{ineq:A4}. Moreover, by Jensen's inequality,
$$
\E [ (B^i_j (\theta_1))^2 ]  %\le 2 \E \Big[ \Big( \int_{t_{j-1}}^{t_j} b^i_s (\theta_{0,1}) d s \Big)^2 \Big] + 2 \Delta^2 \E [ (b^i_{t_{j-1}}(\theta_1))^2 ] 
\le 2 \Delta_n \int_{t_{j-1}}^{t_j} \E [ ( b^i_s (\theta_{0,1})  )^2 ] d s + 2 \Delta_n^2 \E [ (b^i_{t_{j-1}}(\theta_1))^2 ] = O(\Delta_n^2)
$$
uniformly in $i,j$, where the last relation follows thanks to linear growth of $b(\theta_1,\cdot)$ for every $\theta_1$ by \textbf{\ref{as2}} and moment bounds in  Lemma~\ref{l: moments}(1).
We conclude that
$\sup_{i,j} \E [ |r^i_{j,1}| ] = O (\Delta_n^{\frac 3 2})$.
 
Next, we consider the convergence in \eqref{cns:vol} for $k=0$. It is enough to show that $n \sup_{i,j}  \E [ | r^i_{j,0} | ] = o(1)$. Assumption 
\textbf{\ref{as3}} 
%, i.e.\ $\inf_{x,\mu} a^2 (x,\mu,\theta_2) > 0$, 
implies $\E [| r^i_{j,0} |] \le C \E [ |H^i_{j,0}| ]$, where
$$
\E [ | H^i_{j,0}| ] \le \int^{t_j}_{t_{j-1}} \E [ | c^i_s (\theta_{0,2}) - c^i_{t_{j-1}} (\theta_{0,2}) | ] d s.
$$
Lipschitz continuity of $a(\theta_{0,2},\cdot)$ and Lemma~\ref{l: moments}(2) and (4) imply $\E [ (a^i_s (\theta_{0,2}) - a^i_{t_{j-1}} (\theta_{0,2}))^2 ] = O(\Delta_n)$ uniformly in $t_{j-1} \le s \le t_j,j,i$. 
%Indeed, $\E [ (X^i_s - X^i_{t_{j-1}})^2] = O (\Delta)$ {\color{blue}(to do)} and $\E [W_1^2 (\mu_s,\mu_{t_{j-1}})] = O (\Delta)$ uniformly in $t_{j-1} \le s \le t_j,j,i$. To see the last relation we use $W_1 (\mu_s,\mu_{t_{j-1}}) \le \frac{1}{N} \sum_{i} |X^i_s - X^i_{t_{j-1}}|$ and the Minkowski inequality. 
Finally, linear growth of $a(\theta_{0,2},\cdot)$ and moment bounds in Lemma~\ref{l: moments}(1) guarantee $\E [ (a^i_s (\theta_{0,2}) + a^i_{t_{j-1}} (\theta_{0,2}))^2 ] = O(1)$ uniformly in $t_{j-1} \le s \le t_j, j, i$. We conclude
by Cauchy-Schwarz inequality that $\E [|c^i_s (\theta_{0,2}) - c^i_{t_{j-1}} (\theta_{0,2})|] = O(\Delta_n^{\frac{1}{2}})$ uniformly in $t_{j-1} \le s \le t_j, j,i$, whence
$\sup_{i,j} \E [|r^i_{j,0}|] = O (\Delta_n^{\frac 3 2})$. 

The first relation in \eqref{cns:vol} follows from Lemma \ref{l: Riemann} if the function $f(\theta_2,\cdot)$ is locally Lipschitz with polynomial growth. To check this assumption, use 
%{\color{blue}(VP wishes that Lemma \ref{l: Riemann} implies the first convergence in \eqref{cns:vol1} so that it remains only to verify its conditions.)} The first convergence in \eqref{cns:vol1} follows from 
%$$
%n \sup_{i,j} \E \Big[ \Big| \int^{t_j}_{t_{j-1}} ( f^i_{t_{j-1}}(\theta_2) - f^i_s (\theta_2) ) d s \Big| \Big] = o(1).
%$$
%Consider 
%$\E [|f^i_{t_{j-1}}(\theta_2)-f^i_s(\theta_2)|]$ for all $t_{j-1} \le s \le t_j$, $j,i$.
%Use
$
%\frac{x}{y} - \frac{x'}{y'} = \frac{x-x'}{y} + \frac{x' (y'-y)}{y'y}, \qquad 
|\log y_1 - \log y_2| \le |y_1-y_2|/ \min(y_1, y_2) 
$
for $y_1,y_2>0$ and assumption \textbf{\ref{as3}}. Note $b(\theta_1,\cdot)$, $a(\theta_2,\cdot)$ are Lipschitz continuous and have linear growth by \textbf{\ref{as2}}. Hence, the function $f (\theta_2,\cdot)$ satisfies the assumption of Lemma~\ref{l: Riemann}. \\
\noindent
$\bullet$ %Proof of Lemma \ref{lemma:cns} 
Step 2. We want to prove that the sequence $\frac{\Delta_n}{N} S^N_n(\theta)$ in $(C(\Theta;\R), \| \cdot \|_\infty)$ is tight. So we have to show that  for all $N,n$,
$$
\frac{\Delta_n}{N} \E \Big[\sup_\theta {\sum_{k=1}^2} \| \nabla_{\theta_k} S^N_n (\theta) \| \Big] \le C.
$$
We have 
$$ 
\nabla_{\theta_k} S^N_n (\theta)=\sum_{i=1}^N \sum_{j=1}^n \zeta_{j,k}^i (\theta), \qquad k =1,2,
$$
where
\begin{align*}
\zeta_{j,1}^i (\theta) &= - \frac{2(X^i_{t_j} - X^i_{t_{j-1}} - \Delta_n b^i_{t_{j - 1}}(\theta_1) )}{c^i_{t_{j-1}}(\theta_2)} \nabla_{\theta_1} b^i_{t_{j - 1}}(\theta_1),\\    
\zeta_{j,2}^i (\theta) &= - \frac{(X^i_{t_j} - X^i_{t_{j-1}} - \Delta_n b^i_{t_{j - 1}}(\theta_1) )^2}{ \Delta_n (c^2)^i_{t_{j - 1}}(\theta_2)} \nabla_{\theta_2} c^i_{t_{j - 1}}(\theta_2) + \frac{1}{c^i_{t_{j - 1}} (\theta_2)} \nabla_{\theta_2} c^i_{t_{j - 1}}(\theta_2).
\end{align*}
It suffices to show that for all $N,n$ and $i,j$,
\begin{equation}\label{ineq:Esupxi}
\E \big[ \sup_\theta \|\zeta_{j,k}^i (\theta)\| \big] \le C, \qquad k =1,2. 
\end{equation}
Using \textbf{\ref{as3}} and the Cauchy-Schwarz inequality, we get 
\begin{align*}
\E \big[ \sup_\theta \|\zeta_{j,1}^i (\theta)\|\big] &\le C \big(\E \big[ \sup_{\theta_1} |X^i_{t_j}-X^i_{t_{j-1}} - \Delta_n b^i_{t_{j-1}}(\theta_1)|^2 \big] \big)^{\frac{1}{2}} \big( \E \big[ \sup_{\theta_1} \| \nabla_{\theta_1} b^i_{t_{j-1}}(\theta_1) \|^2 \big] \big)^{\frac 1 2},\\
\E \big[ \sup_\theta \|\zeta_{j,2}^i (\theta)\|\big] &\le \frac{C}{\Delta_n} \big( \E \big[ \sup_{\theta_1} |X^i_{t_j} - X^i_{t_{j-1}} - \Delta_n b^i_{t_{j-1}} (\theta_1)|^4 \big] \big)^{\frac{1}{2}} \big( \E \big[ \sup_{\theta_2} \| \nabla_{\theta_2} a^i_{t_{j - 1}}(\theta_2) \|^2 \big] \big)^{\frac{1}{2}} \\
&\qquad + C \E \big[ \sup_{\theta_2} \| \nabla_{\theta_2} a^i_{t_{j - 1}} (\theta_2) \| \big].
\end{align*}
%\begin{align*}
%\E \big[ \sup_\theta \|\xi_{j,1}^i(\theta)\|\big] &\le C \big( \E [ |X^i_{t_j} - X^i_{t_{j-1}}|^2 ] + \Delta^2 \E \big[ \sup_{\theta_1} |b^i_{t_{j-1}} (\theta_1)|^2 \big] \big)^{\frac{1}{2}} \big( \E \big[ \sup_{\theta_1} |\partial_{\theta_1} b^i_{t_{j - 1}}(\theta_1)|^2 \big] \big)^{\frac{1}{2}},\\
%\E \big[ \sup_\theta \|{\xi}_{j,2}^i(\theta)\|\big] &\le \frac{C}{\Delta} \big( \E [ |X^i_{t_j} - X^i_{t_{j-1}}|^4 ] + \Delta^4 \E \big[ \sup_{\theta_1} |b^i_{t_{j-1}} (\theta_1)|^4 \big] \big)^{\frac{1}{2}} \big( \E \big[ \sup_{\theta_2} | \partial_{\theta_2} a^i_{t_{j - 1}}(\theta_2) | \big] \big)^{\frac{1}{2}} \\
%&\qquad + C \E \big[ \sup_{\theta_2} | \partial_{\theta_2} a^i_{t_{j - 1}} (\theta_2) | \big].
%\end{align*}
We use polynomial growth of $\sup_{\theta_1}\| \nabla_{\theta_1} b(\theta_1,\cdot)\|$, $\sup_{\theta_2} \|\nabla_{\theta_2} a(\theta_2,\cdot)\|$ and moment bounds in Lemma~\ref{l: moments}(1). Moreover, Lemma~\ref{l: moments}(2) gives $\sup_{i,j} \E [|X^i_{t_j} - X^i_{t_{j-1}}|^4] = O(\Delta_n^2)$. Finally, $b(\theta_{0,1},\cdot)$ has a linear growth and the mean value theorem implies 
$b(\theta_1,\cdot) - b(\theta_{0,1},\cdot) = \int_0^1 \nabla_{\theta_1} b(\theta_{0,1} + (\theta_1-\theta_{0,1})u,\cdot) d u \cdot (\theta_1-\theta_{0,1})$ 
for all $\theta_1$ in $\Theta_1$, where $\Theta_1$ is convex, bounded and we recall that $\sup_{\theta_1}\| \nabla_{\theta_1} b(\theta_1,\cdot)\|$ has polynomial growth. The moment bounds in Lemma~\ref{l: moments}(1) imply $\E [\sup_{\theta_1} |b^i_{t_{j-1}}(\theta_1)|^4] \le C$, completing the proof of \eqref{ineq:Esupxi}.
%We note that the mean value theorem implies $
%b(\cdot,\theta_1) - b(\cdot,\theta_{0,1}) = (\theta_1-\theta_{0,1}) \int_0^1 \partial_{\theta_1} b(\cdot, \theta_{0,1} + (\theta_1-\theta_{0,1})u) d u$ for all $\theta_1$ in $\Theta_1$, where $\Theta_1$ is convex, bounded. We also note linear growth of $b(\cdot, \theta_{0,1})$ and polynomial growth of $\sup_{\theta_1} |\partial_{\theta_1} b(\cdot,\theta_1)|$, $\sup_{\theta_2} |\partial_{\theta_2} a(\cdot,\theta_2)|$. Then moment bounds by part 1.\ and $\sup_{i,j} \E [|X^i_{t_j} - X^i_{t_{j-1}}|^4] = O(\Delta^2)$ by part 2.\ of Lemma~\ref{l: moments} complete the proof of \eqref{ineq:Esupxi}.
\end{proof}

\subsubsection{Proof of Theorem \ref{th: consistency}}
\begin{proof}
%We prove the result for $p_1 = p_2 = 1$. It is easy to see it still holds true when $p_1, p_2 > 1$. \\
Assumption \textbf{\ref{as5}} implies that for every $\varepsilon>0$ there exists $\eta>0$ such that $J(\theta_2) - J(\theta_{0,2}) > \eta$
for every $\theta_2$ with $\|\theta_2 - \theta_{0,2}\| \ge \varepsilon$. Thus  $\{ \|\hat \theta^N_{n,2}-\theta_{0,2}\| \ge \varepsilon \}\subseteq \{J(\hat \theta^N_{n,2}) - J(\theta_{0,2}) > \eta\}$. The probability of the latter event converges to $0$ in view of 
$$
J (\hat \theta^N_{n,2}) - J (\theta_{0,2}) = J^N_{n,0} + J^N_{n,1},
$$
where 
%$$
%\rho^N_{1,n} := \bar S_2(\hat \theta^N_{2,n}) - S^N_n(\hat \theta^N_{1,n}, \hat \theta^N_{2,n}) + S^N_n(\hat \theta^N_{1,n}, \theta_{0,2}) - \bar S_2(\theta_{0,2}), \qquad \rho^N_{0,n} := S^N_n(\hat \theta^N_{1,n}, \hat \theta^N_{2,n}) - S^N_n(\hat \theta^N_{1,n}, \theta_{0,2}).
%$$
%the uniform convergence of $\frac{\Delta_n}{N}S^N_n$ to $\bar S_2$ in probability 
%Lemma~\ref{lemma:cns}(i) 
the definition of $\hat \theta^N_n$ and \eqref{lim:con2} imply respectively
\begin{align*}
J^N_{n,0} &:= \frac{\Delta_n}{N} (S^N_n(\hat \theta^N_{n,1}, \hat \theta^N_{n,2}) - S^N_n(\hat \theta^N_{n,1}, \theta_{0,2})) \le 0,\\
J^N_{n,1} &:= J (\hat \theta^N_{n,2}) - J (\theta_{0,2}) - J^N_{n,0} \le 2 \sup_{(\theta_1,\theta_2) \in \Theta} \Big| \frac{\Delta_n}{N} S^N_n(\theta_1,\theta_2)-J(\theta_2) \Big| = o_\P(1).
\end{align*}
%\textcolor{red}{Same notation as below \eqref{cfdeComp} and in \eqref{rhoNndeComp}.}
%On the other hand, we have $\bar S_2 (\hat \theta^N_{2,n}) - \bar S_2(\theta_{2,0}) \ge 0$ since $y_0/y + \log y \ge 1 + \log y_0$ for all $y,y_0>0$. 
%Hence, $\bar S_2 (\hat \theta^N_{2,n}) \to_\P \bar S_2(\theta_{2,0})$.
Consistency of $\hat \theta^N_{n,1}$ follows in a similar way. 
%{\color{blue}In place of A5 assume that for every $\varepsilon>0$,
%$$
%\inf_{\theta: |\theta_1-\theta_{0,1}|\ge \varepsilon} I(\theta) > 0.
%$$
%Check if we can relax: $\inf_{\theta_1: |\theta_1-\theta_{1,0}|\ge \varepsilon} I (\theta_1,\theta_{2,0})$ maybe?} 
Assumption \textbf{\ref{as5}} implies that for every $\varepsilon>0$ there exists $\eta >0$ such that $I(\theta_1,\theta_2) > \eta$ for every $(\theta_1,\theta_2)$ with $\|\theta_1-\theta_{0,1}\| \ge \varepsilon$. Thus $\{\|\hat \theta^N_{n,1}-\theta_{0,1}\| \ge \varepsilon\}\subseteq \{I(\hat \theta^N_{n,1}, \hat \theta^N_{n,2}) > \eta \}$. The probability of the latter event converges to $0$ because 
$$
I(\hat \theta^N_{n,1}, \hat \theta^N_{n,2}) = I^N_{n,0} + I^N_{n,1},
$$
where the definition of $\hat \theta^N_n$ and \eqref{lim:con1} imply respectively
\begin{align*}
I^N_{n,0} &:= \frac{1}{N} (S^N_n (\hat \theta^N_{n,1},\hat \theta^N_{n,2}) - S^N_n(\theta_{0,1},\hat \theta^N_{n,2})) \le 0,\\
I^N_{n,1} &:= I(\hat \theta^N_{n,1},\hat \theta^N_{n,2}) - I^N_{n,0}
%\frac{1}{N} (S^N_n (\hat \theta^N_{n,1},\hat \theta^N_{n,2}) - S^N_n(\theta_{0,1},\hat \theta^N_{n,2}))
= o_\P(1).
\end{align*}
%{\color{blue} For other solutions check Genon-Catalot, Lar\'edo (2021, 2022)?}
\end{proof}

\subsection{Asymptotic normality}

The proof of the asymptotic normality of our estimator is obtained following a classical route. It consists in proving the asymptotic normality of the first derivative of the contrast function \eqref{eq:def contrast} (see for example  \cite[Section5a]{GenJac93}). We introduce in particular the appropriate normalization matrix
$$M_n^N := \operatorname{diag} \Big( \underbrace{\frac{1}{\sqrt{N}}, \dots , \frac{1}{\sqrt{N}}}_{\text{$p_1$ times}}, \underbrace{\sqrt{\frac{\Delta_n}{N}}, \dots , \sqrt{\frac{\Delta_n}{N}}}_{\text{$p_2$ times}} \Big).
$$
%vector
%$$L_n^N(\theta_0) := \begin{pmatrix}
%- \displaystyle \frac{1}{\sqrt{N}} \partial_{\theta_1} S_n^N (\theta_0) \\
%- \displaystyle \frac{1}{\sqrt{N\,n}} \partial_{\theta_2} S_n^N (\theta_0)
%\end{pmatrix}.$$
%{\color{blue}(no need of $-$?)} 
The proof of Theorem \ref{th: normality} is based on the following proposition.
\begin{proposition}
Assume %that 
\textbf{\ref{as1}}-\textbf{\ref{as4}(I)} and 
\textbf{(II)}, %and 
\textbf{\ref{as7}}. 
If $N\Delta_n \to 0$ then
as $N,n \rightarrow \infty$,
$$%M^N_n 
\nabla_\theta S^N_n(\theta_0) {M^N_n}
\xrightarrow{\mathcal{L}} {\cal N}(0, 2 \Sigma (\theta_0)),
$$
where 
{ $\Sigma (\theta_0)$ is a $p \times p$ matrix defined in \textbf{\ref{as6}}}.

\label{p: norm L}
\end{proposition}
\noindent
We observe that, as $\nabla_\theta S_n^N (\hat{\theta}_n^N) = 0$, by Taylor's formula we obtain
\begin{equation}
{(\hat \theta^N_n - \theta_0)} \int_0^1 \nabla^2_\theta S_n^N (\theta_0 + s (\hat{\theta}_n^N - \theta_0)) ds 
%(\hat \theta^N_n - \theta_0) %\begin{pmatrix}
%\hat{\theta}_{1, n}^N - \theta_{1,0} \\
%\hat{\theta}_{2, n}^N - \theta_{2,0}
%\end{pmatrix} 
\,  = - \nabla_\theta S_n^N (\theta_0).
\label{eq: taylor}
\end{equation}
Multiplying the equation \eqref{eq: taylor} by $M^N_n$,
%We then introduce the matrix 
%$$M_n^N := \begin{pmatrix}
%\displaystyle \frac{1}{\sqrt{N}} & 0 \\
%0 & \displaystyle \frac{1}{\sqrt{N n}}
%\end{pmatrix}$$
%From \eqref{eq: taylor} it follows
we obtain
\begin{equation}
{(\hat \theta^N_n - \theta_0) (M^N_n)^{-1}} \int_0^1 \Sigma_n^N (\theta_0 + s (\hat{\theta}_n^N - \theta_0)) d s %\, (M^N_n)^{-1}(\hat \theta^N_n - \theta_0)
= - %M^N_n 
\nabla_\theta S^N_n(\theta_0) {M^N_n},
\label{eq: Taylor2}
\end{equation}
where 
$$
\Sigma^N_n (\theta) := M_n^N \nabla^2_\theta  S_n^N (\theta) M_n^N = \begin{pmatrix}
\Sigma_n^{N,(1)} (\theta) &  \Sigma_n^{N,(12)} (\theta) \\
\Sigma_n^{N,(21)} (\theta) & \Sigma_n^{N,(2)} (\theta)
\end{pmatrix}
$$
with
\begin{equation*}
\begin{aligned}[c]
\Sigma_n^{N,(1)} (\theta) &= (1/N) \nabla^2_{\theta_1} S_n^N {(\theta)},\\
\Sigma_n^{N,(21)}{(\theta)} &= (\sqrt{\Delta_n}/N) \nabla_{\theta_2} \nabla_{\theta_1} S_n^N (\theta), 
\end{aligned}
\qquad
\begin{aligned}[c]
\Sigma_n^{N,(12)} {(\theta)} &= (\sqrt{\Delta_n}/N) \nabla_{\theta_1} \nabla_{\theta_2} S_n^N (\theta),\\ \Sigma_n^{N,(2)}{(\theta)}&= (\Delta_n/N) \nabla^2_{\theta_2} S_n^N (\theta).
\end{aligned}
\end{equation*}
The analysis of the second derivatives of the contrast function is gathered in the following proposition, which will be proven at the end of this section. 
\begin{proposition}
%Suppose that Assumptions
Assume \textbf{\ref{as1}}-\textbf{\ref{as5}} 
%{\color{blue}, (no A7?)} hold, 
with both \textbf{(I)} and 
\textbf{(II)} in \textbf{\ref{as4}}.
%, and that ${N \sqrt{\Delta_n}} \rightarrow 0$ for $n, N \rightarrow \infty$. 
Then
%, the following statements hold 
as $N,n \to \infty$,
\begin{enumerate}
\item $\Sigma_n^N (\theta_0) \xrightarrow{\mathbb{P}} \Sigma (\theta_0)$,
\item $\sup_{s \in [0,1]} \| \Sigma^N_n (\theta_0 + s (\hat \theta^N_n -\theta_0)) - \Sigma^N_n (\theta_0) \| \xrightarrow{\mathbb{P}} 0$, where $\| \cdot \|$ refers to the operator norm on the space of $p \times p$ matrices induced by the Euclidean norm for vectors.
\end{enumerate}
\label{p: second derivatives contrast}
\end{proposition}

\noindent
By Proposition \ref{p: second derivatives contrast} assumption \textbf{\ref{as6}} implies that the probability that
$
\int_0^1 \Sigma_n^N (\theta_0 + s (\hat{\theta}_n^N - \theta_0)) d s
%= C^N_n (\theta_0) + \int_0^1 \big( C_n^N (\theta_0 + s (\hat{\theta}_n^N - \theta_0)) - C^N_n(\theta_0) \big) d s% = B + o_\P(1)
$
is invertible tends to 1. Applying its inverse to the equation \eqref{eq: Taylor2}, %left and right, 
by Proposition~\ref{p: norm L} and the continuous mapping theorem, we get 
$$
{ \big( \sqrt{N} (\hat{\theta}_{n,1}^N - \theta_{0, 1} ), \sqrt{N/\Delta_n} (\hat{\theta}_{n,2}^N - \theta_{0, 2}) \big) }
= %(M^N_n)^{-1} 
(\hat \theta^N_n -\theta_0 ) {(M^N_n)^{-1} } \xrightarrow{{\cal L}} 
{\cal N} \big(0, 2 ( \Sigma(\theta_0) )^{-1} \big). $$

\subsection{Proof of Proposition \ref{p: norm L}}
\begin{proof}
As in the proof of consistency, we omit the notation for dependence on $N,n$. In particular, we write $X^i_t$ for $X^{i,N}_t$, $\mu_t$ for $\mu^N_t$, $t_j$ for $t_{j,n}$. Denote by $f^i_{t_{j-1}}(\theta)$ the values of $f(\theta, X^i_{t_{j-1}}, \mu_{t_{j-1}})$. %for $k=1, 2$. 
We note that $- %M^N_n 
\nabla_{\theta} S^N_n (\theta) {M^N_n}$ consists of
%From the definition \eqref{eq:def contrast} of $S_n^N$ we have 
$- \partial_{\theta_{1, h}} S_n^N(\theta)/\sqrt{N} =:\sum_{j = 1}^n \xi_{j,h}^{(1)} (\theta)$ and $- \sqrt{\Delta_n/N} \partial_{\theta_{2, \Tilde{h}}} S_n^N(\theta) =: \sum_{j = 1}^n {\xi}_{j,\Tilde{h}}^{(2)} (\theta)$, %, with 
where
\begin{align*}
\xi_{j,h}^{(1)} (\theta) &{:}= \frac{1}{\sqrt{N}} \sum_{i = 1}^N 2 \frac{\partial_{\theta_{1, h}} b^i_{t_{j - 1}}(\theta_1)}{c^i_{t_{j - 1}}(\theta_2)} (X^i_{t_j} - X^i_{t_{j-1}} - \Delta_n b^i_{t_{j - 1}}(\theta_1) ),\\
\xi_{j,\Tilde{h}}^{(2)} (\theta) &{:}= \sqrt{\frac{\Delta_n}{N}} \sum_{i = 1}^N \frac{\partial_{\theta_{2, \Tilde{h}}} c^i_{t_{j - 1}}(\theta_2)}{ \Delta_n (c^i_{t_{j - 1}}(\theta_2))^2} (X^i_{t_j} - X^i_{t_{j-1}} - \Delta_n b^i_{t_{j - 1}}(\theta_1) )^2 - \frac{\partial_{\theta_{2, \Tilde{h}}} c^i_{t_{j - 1}}(\theta_2) }{c^i_{t_{j - 1}}(\theta_2) }
\end{align*}
for $h = 1, \dots, p_1$, $\Tilde{h} = 1, \dots, p_2$.
To prove the asymptotic normality of $%L_n^N(\theta_0) := 
-%M^N_n 
\nabla_\theta S^N_n (\theta_0){M^N_n}$ we want to use a  central limit theorem for martingale difference arrays,
%argument for triangular array of martingale increments, 
in accordance with Theorems 3.2 and 3.4 of \cite{HalHey80}.
%Hence, to prove our result we need to show that $L_n^N(\theta_0)$ is a triangular array of martingale increments, i.e. that
Approximation of $-%M^N_n 
\nabla_\theta S^N_n (\theta_0){M^N_n}$ by a martingale array follows from 
\begin{equation}
\sum_{j = 1}^n \E_{t_{j - 1}}[\xi^{(1)}_{j,h} (\theta_0)] \xrightarrow{\mathbb{P}} 0, \qquad  \sum_{j = 1}^n \E_{t_{j - 1}}[\xi^{(2)}_{j,\Tilde{h}} (\theta_0)] \xrightarrow{\mathbb{P}} 0
\label{e: mtg array}
\end{equation}
for $h = 1, \dots, p_1$, $\Tilde{h} = 1, \dots, p_2$.
Moreover, application of the central limit theorem requires that for some $r > 0$ the following convergences hold:
%and that, for a constant $r > 0$, the following convergences hold:
\begin{align}
 &\sum_{j = 1}^n \E_{t_{j - 1}}[\xi^{(1)}_{j,h_1} (\theta_0)\xi^{(1)}_{j,h_2} (\theta_0)] \xrightarrow{\mathbb{P}} 4 \int_0^{{T}} \int_{\R}  \frac{\partial_{\theta_{1, h_1}} b(\theta_{0, 1}, x, \bar{\mu}_t)\partial_{\theta_{1, h_2}} b(\theta_{0, 1}, x, \bar{\mu}_t)}{c(\theta_{0, 2}, x, \bar{\mu}_t)}  \bar{\mu}_t(dx) dt,
\label{e: norm theta1}\\
 &\sum_{j = 1}^n \E_{t_{j - 1}}[{\xi}^{(2)}_{j,\Tilde{h}_1} (\theta_0){\xi}^{(2)}_{j,\Tilde{h}_2} (\theta_0)] \xrightarrow{\mathbb{P}} 2 \int_0^{{T}} \int_{\R} \frac{\partial_{\theta_{2, \Tilde{h}_1}} c(\theta_{0, 2}, x, \bar{\mu}_t)\partial_{\theta_{2, \Tilde{h}_2}} c(\theta_{0, 2}, x, \bar{\mu}_t)}{c^2(\theta_{0, 2}, x, \bar{\mu}_t)} \bar{\mu}_t(dx) dt,
\label{e: norm theta2}\\
 &\sum_{j = 1}^n  \E_{t_{j - 1}}[{\xi}^{(1)}_{j,h} (\theta_0) \xi^{(2)}_{j,\Tilde{h}} (\theta_0)]  \xrightarrow{\mathbb{P}} 0,
 \label{e: norm mixed}\\
 &\sum_{j = 1}^n \E_{t_{j - 1}}[|{\xi}^{(1)}_{j,h} (\theta_0)|^{2 + r}] \xrightarrow{\mathbb{P}} 0, \qquad  \sum_{j = 1}^n \E_{t_{j - 1}}[|{\xi}^{(2)}_{j,\Tilde{h}} (\theta_0)|^{2 + r}] \xrightarrow{\mathbb{P}} 0,
\label{e: negl theta}
\end{align}
where $h,h_1,h_2 = 1,\dots,p_1$, $\tilde h, \tilde h_1, \tilde h_2 = 1,\dots,p_2.$\\

\noindent
$\bullet$ Proof of \eqref{e: mtg array}. \\
%We start proving that $L_n^N(\theta_0)$ is a triangular array of martingale increments.
Assumptions \textbf{\ref{as3}} and \textbf{\ref{as4}(I)} imply that $F^i_{j,h} := 2 \partial_{\theta_{1, h}} b^i_{t_{j-1}}(\theta_{0,1})(c^i_{t_{j - 1}}(\theta_{0, 2}))^{-1}$ satisfies $|F^i_{j,h}|\le C (1 + |X_{t_{j - 1}}^i|^{k_1} + W_2^{l_1} (\mu_{t_{j - 1}}, \delta_0))$. Hence, from  Lemma \ref{l: moments}(1) it is easy to see that $F^i_{j,h} = R^i_{t_{j - 1}}(1)$.
%$|\partial_{\theta_1} b^i_{t_{j - 1}}(\theta_{0, 1}) (c^i_{t_{j - 1}}(\theta_{0, 2}))^{-1}| \le C (1 + |X_{t_{j - 1}}^i|^{k_1} + W_{1}^{l_1}(\mu_{t_{j - 1}}, \delta_0))$.
%Hence, from Point 1.\ of Lemma \ref{l: moments} it is easy to see that $\partial_{\theta_1} b^i_{t_{j - 1}}(\theta_{0, 1}) (c^i_{t_{j - 1}}(\theta_{0, 2}))^{-1} = R^i_{t_{j - 1}}(1)$. 
If $N \Delta_{{n}} \to 0$ then  Lemma 
\ref{l: conditional expectation}(3) implies 
\begin{align*}
\sum_{j=1}^n \E_{t_{j - 1}}[\xi^{(1)}_{j,h} (\theta_0)]
%&= \frac{1}{\sqrt{N}} \sum_{i = 1}^N \sum_{j = 1}^n R_{t_{j - 1}}^i (1) \E_{j - 1}[X^i_{t_j} - X^i_{t_{j-1}} - \Delta b^i_{t_{j - 1}}(\theta_1) ] \\
&= \frac{1}{N^{\frac 1 2}} \sum_{i=1}^N \sum_{j=1}^n R^i_{t_{j - 1}}(1) R^i_{t_{j - 1}}(\Delta_n^{\frac 3 2}) \xrightarrow{L^1} 0
%\\ &= \sqrt{N \Delta} \frac{1}{n\, N} \sum_{i = 1}^N \sum_{j = 1}^n R_{t_{j - 1}}(\theta_0, 1).
\end{align*}
and so the convergence in probability. %and as we have assumed $N \Delta^\frac{1}{2} \rightarrow 0$. 
In a similar way, using  Lemma \ref{l: conditional expectation}(1), we obtain 
\begin{align*}
 \sum_{j=1}^n \E_{t_{j - 1}}[{\xi}^{(2)}_{j,\Tilde{h}} (\theta_0)] & = { \Big( \frac{\Delta_n}{N} \Big)^{\frac 1 2} \sum_{i=1}^N 
 \sum_{j=1}^n \frac{\partial_{\theta_{2, \tilde{h}}} c^i_{t_{j - 1}}(\theta_{0, 2}) }{ \Delta_n (c^i_{t_{j - 1}}(\theta_{0, 2}))^2}
 (\Delta_n c^i_{t_{j - 1}}(\theta_{0, 2}) + R^i_{t_{j - 1}}(\Delta_n^2)) - \frac{\partial_{\theta_{2, \tilde{h}}} c^i_{t_{j - 1}}(\theta_{0, 2}) }{ c^i_{t_{j - 1}}(\theta_{0, 2})} }
 \\
 & = \Big( \frac{\Delta_n}{N} \Big)^{\frac 1 2} \sum_{i=1}^N 
\sum_{j=1}^n \frac{\partial_{\theta_{2, \tilde{h}}} c^i_{t_{j - 1}}(\theta_{0, 2}) }{ \Delta_n (c^i_{t_{j - 1}}(\theta_{0, 2}))^2} R^i_{t_{j - 1}}(\Delta_n^2)  \\
& = \Big( \frac{\Delta_n}{N} \Big)^{\frac 1 2} \sum_{i=1}^N 
\sum_{j=1}^n R^i_{t_{j - 1}}(\Delta_n), 
\end{align*}
which converges to $0$ in $L^1$ and so in probability if $N \Delta_n \to 0$.\\

\noindent
$\bullet$ Proof of \eqref{e: norm theta1}. \\
We have
\begin{equation}\label{eq: start norm theta1}
\E_{t_{j-1}} [\xi_{j,h_1}^{(1)} (\theta_0)\xi_{j,h_2}^{(1)} (\theta_0)] = \frac{1}{N} \sum_{i_1, i_2=1}^N \E_{t_{j-1}} [(A^{i_1}_j + B^{i_1}_j)(A^{i_2}_j + B^{i_2}_j)] F^{i_1}_{j,h_1} F^{i_2}_{j,h_2},   \end{equation}
where 
$$
F^i_{j,h} := 2 \frac{\partial_{\theta_{1, h}} b^{i}_{t_{j-1}}(\theta_{0,1})} {c^{i}_{t_{j-1}}(\theta_{0,2})} = R^{i}_{t_{j-1}}(1),
$$
and
\begin{equation}\label{def: BAnorm}
B^{i}_j := \int_{t_{j-1}}^{t_j} (b^{i}_s (\theta_{0,1}) - b^{i}_{t_{j-1}}(\theta_{0,1})) d s, \qquad A^{i}_j := \int_{t_{j-1}}^{t_j} a_s^{i}(\theta_{0,2}) d W^{i}_s.
\end{equation}
We have $\E_{t_{j-1}} [(B^{i}_j)^2] = R^{i}_{t_{j-1}}(\Delta_n^3)$ and $\E_{t_{j-1}} [(A^{i}_j)^2] = R^{i}_{t_{j-1}}(\Delta_n)$, whereas
if $i_1 \neq i_2$ then $\E_{t_{j-1}} [A^{i_1}_j A^{i_2}_j] = 0$ because of the independence of Brownian motions. 
%, hence,  $\E_{t_{j-1}} [B^{i_1}_j B^{i_2}_j] = R^{i_1,i_2}_{t_{j-1}} (\Delta^3)$ and 
%$\E_{t_{j-1}} [A^{i_1}_j B^{i_2}_j] = R^{i_1,i_2}_{t_{j-1}} (\Delta^2)$ by the Cauchy-Schwartz inequality.
Hence, by the Cauchy-Schwarz inequality, 
$$
\E_{t_{j-1}} [(A^{i_1}_j + B^{i_1}_j)(A^{i_2}_j + B^{i_2}_j)] = \E_{t_{j-1}}[(A^{i_1}_j)^2] \mathbf{1} (i_1=i_2) + R_{t_{j-1}}^{i_1,i_2} (\Delta_n^2).
$$ 
We get
$$
\sum_{j=1}^n \E_{t_{j-1}} [\xi_{j,h_1}^{(1)} (\theta_0)\xi_{j,h_2}^{(1)} (\theta_0)] = \frac{1}{N} \sum_{j=1}^n \sum_{i=1}^N \E_{t_{j-1}} [(A^i_j)^2] F^i_{j,h_1} F^i_{j,h_2} + \frac{1}{N} \sum_{j=1}^n \sum_{i_1{,} i_2=1}^N R^{i_1,i_2}_{t_{j-1}} (\Delta_n^2),
$$
where the last sum converges to $0$ in $L^1$ and so in probability if $N \Delta_n \to 0$. We can therefore focus on the first sum. We decompose the term $\E_{t_{j-1}}[(A^i_j)^2]$ into $\Delta_n c^i_{t_{j-1}}(\theta_{0,2})$
and 
$$
\E_{t_{j-1}} [(A^i_j)^2] - \Delta_n c^i_{t_{j-1}}(\theta_{0,2}) = \int_{t_{j-1}}^{t_j} \E_{t_{j-1}} [c^i_s(\theta_{0,2}) - c^i_{t_{j-1}}(\theta_{0,2})] d s = R^i_{t_{j-1}} (\Delta_n^{\frac 3 2}).%,
$$
%and even more equal to  $R^i_{t_{j-1}} (\Delta^2)$ as in the proof of Point 1.\ of Lemma \ref{l: conditional expectation}. 
The result follows from $\Delta_n \to 0$ and application of Lemma \ref{l: Riemann}.\\

\noindent
$\bullet$ Proof of \eqref{e: negl theta}, first convergence. \\
We want to show \eqref{e: negl theta} with $r=2$. We use the same notation as in \eqref{eq: start norm theta1} and consider the terms
\begin{equation}\label{eq:AB4}
\E_{t_{j-1}} [(A^{i_1}_j + B^{i_1}_j)(A^{i_2}_j + B^{i_2}_j)(A^{i_3}_j+B^{i_3}_j)(A^{i_4}_j +B^{i_4}_j)] F^{i_1}_{j,h} F^{i_2}_{j,h} F^{i_3}_{j,h} F^{i_4}_{j,h}.    
\end{equation}
We have $F^i_j = R^i_{t_{j-1}}(1)$, moreover, $\E_{t_{j-1}} [(A^{i}_j)^4] = R_{t_{j-1}}^{i} (\Delta_n^2)$, $\E_{t_{j-1}}[(B^{i}_j)^4] = R_{t_{j-1}}^{i} (\Delta_n^6)$ and so
$\E_{t_{j-1}} [(A^{i}_j + B^{i}_j)^4] = R^{i}_{t_{j-1}}(\Delta_n^2)$. Application of the Cauchy-Schwarz inequality shows that the term in \eqref{eq:AB4} is also $R^{i_1,i_2,i_3,i_4}_{t_{j-1}}(\Delta_n^2)$.
%In case $i_1=i_2=i_3=i_4$ we get
%$$
%\E_{t_{j-1}} [(A^{i_1}_j + B^{i_1}_j)^4] = R^{i_1}_{t_{j-1}}(\Delta^2)
%$$
%since $\E_{t_{j-1}} [(A^{i_k}_j)^4] = R_{t_{j-1}}^{i_k} (\Delta^2)$, $\E_{t_{j-1}}[(B^{i_k}_j)^4] = R_{t_{j-1}}^{i_k} (\Delta^6)$. In case where $i_1=i_3$, $i_2=i_4$ are distinct we get 
%$$
%\E_{t_{j-1}} [(A^{i_1}_j + B^{i_1}_j)^2(A^{i_2}_j + B^{i_2}_j)^2] = R^{i_1,i_2}_{t_{j-1}}(\Delta^2)
%$$
%by the Cauchy-Schwarz inequality and the same $\E_{t_{j-1}}[(A^{i_k}_j)^4] = R^{i_k}_{t_{j-1}}(\Delta^2)$, $\E_{t_{j-1}}[(B^{i_k}_j)^4] = R^{i_k}_{t_{j-1}}(\Delta^6)$. Similarly, in case where $i_1$, $i_2=i_3=i_4$ or $i_1, i_2$, $i_3=i_4$ are  pairwise distinct,
%$$
%\E_{t_{j-1}} [(A^{i_1}_j + B^{i_1}_j)(A^{i_2}_j + B^{i_2}_j)(A^{i_3}_j + B^{i_3}_j)^2] = R^{i_1,i_2,i_3}_{t_{j-1}} (\Delta^2).
%$$
In case where $i_1, i_2, i_3, i_4$ are pairwise distinct we decompose $A^i_j$ into 
\begin{equation}\label{def: A12norm}
A^{i}_{j,2} := \int_{t_{j-1}}^{t_j} (a^{i}_s(\theta_{0,2}) - a^{i}_{t_{j-1}}(\theta_{0,2})) d W^{i}_s,\qquad
A^{i}_{j,1} := \int_{t_{j-1}}^{t_j} a^{i}_{t_{j-1}} (\theta_{0,2}) d W^{i}_s,
\end{equation}
which satisfy $\E_{t_{j-1}} [(A^i_{j,k})^4] = R^i_{t_{j-1}}(\Delta_n^{2k})$, $k=1,2$. In particular the independence of the Brownian motions implies 
$$
\E_{t_{j-1}} [A^{i_1}_{j,1} A^{i_2}_{j,1} A^{i_3}_{j,1} A^{i_4}_{j,k}] F^{i_1}_{j,h} F^{i_2}_{j,h} F^{i_3}_{j,h} F^{i_4}_{j,h} = 0
$$ 
for $k=1,2$. The %worst 
term {converging to $0$ at the slowest rate} in \eqref{eq:AB4} is then, up to a permutation of the indices $i_1,i_2,i_3,i_4$, 
$$
\E_{t_{j-1}} [A^{i_1}_{j,1} A^{i_2}_{j,1} A^{i_3}_{j,2} A^{i_4}_{j,2} + A^{i_1}_{j,1} A^{i_2}_{j,1} A^{i_3}_{j,1} B^{i_4}_j] F^{i_1}_{j,h} F^{i_2}_{j,h} F^{i_3}_{j,h} F^{i_4}_{j,h} = R^{i_1,i_2,i_3,i_4}_{t_{j-1}} (\Delta_n^3).
$$
%$$
%\E_{t_{j-1}} [(A^{i_1}_j + B^{i_1}_j)(A^{i_2}_j + B^{i_2}_j)(A^{i_3}_j+B^{i_3}_j)(A^{i_4}_j +B^{i_4}_j)] = R^{i_1,i_2,i_3,i_4}_{t_{j-1}} (\Delta^3).
%$$
We get
\begin{equation}\label{eq:negl1}
\sum_{j=1}^n \E_{t_{j-1}} [(\xi_{j,h}^{(1)}(\theta_0))^4] = \frac{1}{N^2} \sum_{j=1}^n \Big( \sum_{i \in I} R^i_{t_{j-1}} (\Delta_n^3) + \sum_{i \in I^c} R^i_{t_{j-1}} (\Delta_n^2) \Big),
\end{equation}
where $I$ denotes a set of all $i = (i_1,i_2,i_3,i_4) \in \{1,\dots,N\}^4$ such that $i_1, i_2, i_3, i_4$ are pairwise distinct. We note that $\operatorname{card}(I) = O(N^4)$ and $\operatorname{card}(I^c) = O(N^3)$. We conclude that \eqref{eq:negl1} converges to $0$ in $L^1$ and so in probability if $N \Delta_n \to 0$.\\

\noindent
$\bullet$ Proof of \eqref{e: norm theta2}. \\
We rewrite the left hand side of \eqref{e: norm theta2} as
\begin{equation}
\frac{\Delta_n}{N} \sum_{j=1}^n \sum_{i_1, i_2=1}^N \Delta_n^{-2} C^{i_1}_{j,\Tilde{h}_1} C^{i_2}_{j,\Tilde{h}_2} \E_{t_{j-1}} [ D^{i_1}_j D^{i_2}_j],\label{eq: norm theta2 start}
%\frac{\partial_{\theta_2} c^{i_1}_{t_{j - 1}}(\theta_{0, 2})}{ \Delta (c^{i_1}_{t_{j - 1}}(\theta_{0, 2}))^2} \frac{\partial_{\theta_2} c^{i_2}_{t_{j - 1}}(\theta_{0, 2})}{ \Delta (c^{i_2}_{t_{j - 1}}(\theta_{0, 2}))^2}\nonumber \\
%\times \E_{t_{j - 1}} [  &((X^{i_1}_{t_j} - X^{i_1}_{t_{j-1}} - \Delta b^{i_1}_{t_{j - 1}}(\theta_{0, 1}) )^2 - \Delta c^{i_1}_{t_{j - 1}}(\theta_{0, 2})) \nonumber \\ 
%\times &((X^{i_2}_{t_j} - X^{i_2}_{t_{j-1}} - \Delta b^{i_2}_{t_{j - 1}}(\theta_{0,1}) )^2 - \Delta c^{i_2}_{t_{j - 1}}(\theta_{0, 2})  ) ].
\end{equation}
where
$$
C^i_{j,\Tilde{h}} := \frac{\partial_{\theta_{2, \Tilde{h}}} c^{i}_{t_{j-1}} (\theta_{0,2})}{(c^{i}_{t_{j-1}}(\theta_{0,2}))^2} = R^{i}_{t_{j-1}}(1), \qquad D^{i}_j :=  (X^{i}_{t_j}-X^{i}_{t_{j-1}}-\Delta_n b^{i}_{t_{j-1}}(\theta_{0,1}))^2 - \Delta_n c^{i}_{t_{j-1}}(\theta_{0,2}).
$$
We consider the term $\E_{t_{j-1}} [D^{i_1}_j D^{i_2}_j]$ in \eqref{eq: norm theta2 start}. By  Lemma \ref{l: conditional expectation}(1) %the conditional expectation
%{\color{green}$\E_{t_{j-1}}[D^{i_1}_j D^{i_2}_j]$} in \eqref{eq: norm theta2 start} equals
it equals
\begin{align}\label{eq:cexpected22}
&\E_{t_{j-1}} [ (X^{i_1}_{t_j} - X^{i_1}_{t_{j-1}} - \Delta_n b^{i_1}_{t_{j-1}}(\theta_{0,1}))^2 (X^{i_2}_{t_j} - X^{i_2}_{t_{j-1}} - \Delta_n b^{i_2}_{t_{j-1}}(\theta_{0,1}))^2]\\
&\qquad- \Delta_n c^{i_1}_{t_{j-1}} (\theta_{0,2}) \Delta_n c^{i_2}_{t_{j-1}} (\theta_{0,2}) + R^{i_1,i_2}_{t_{j-1}}(\Delta_n^3).\nonumber
\end{align}
%\begin{align*}
%\E_{t_{j-1}}[D^{i_1}_j D^{i_2}_j] = \E_{t_{j-1}} [ &(X^{i_1}_{t_j} - X^{i_1}_{t_{j-1}} - \Delta b^{i_1}_{t_{j-1}}(\theta_{0,1}))^2\\ 
%\times &(X^{i_2}_{t_j} - X^{i_2}_{t_{j-1}} - \Delta b^{i_2}_{t_{j-1}}(\theta_{0,1}))^2]
%- \Delta c^{i_1}_{t_{j-1}} (\theta_{0,2}) \Delta c^{i_2}_{t_{j-1}} (\theta_{0,2}) + R^{i_1,i_2}_{t_{j-1}}(\Delta^3).%\label{eq:cexpected22}
%\end{align*}
%We consider the conditional expectation on the right hand side above.
%in \eqref{eq:cexpected22}. 
If $i_1=i_2$ then  Lemma \ref{l: conditional expectation}(2) implies
$$
\E_{t_{j-1}} [ (X^{i}_{t_j} - X^{i}_{t_{j-1}} - \Delta_n b^{i}_{t_{j-1}}(\theta_{0,1}))^4] = 3 \Delta_n^{2} (c^{i}_{t_{j-1}}(\theta_{0,2}))^2 + R^{i}_{t_{j-1}} (\Delta_n^{\frac 5 2}),
$$
whence 
\begin{equation}\label{eq:condD2}
\E_{t_{j-1}}[(D^{i}_j)^2] = 2 \Delta_n^2  (c^{i}_{t_{j-1}}(\theta_{0,2}))^2 + R^{i}_{t_{j-1}} (\Delta_n^{\frac 5 2}).
\end{equation}
If $i_1 \neq i_2$ then to deal with the term in \eqref{eq:cexpected22} we decompose
$$
X^{i}_{t_j} - X^{i}_{t_{j-1}} - \Delta_n b^{i}_{t_{j-1}}(\theta_{0,1}) = A^i_{j,1} + A^i_{j,2} + B^i_j
$$
as in \eqref{def: BAnorm}, \eqref{def: A12norm}, where
%\begin{align*}
%A^{i}_{j,2} &:= \int_{t_{j-1}}^{t_j} (a^{i}_s(\theta_{0,2}) - a^{i}_{t_{j-1}}(\theta_{0,2})) d W^{i}_s,\qquad
%A^{i}_{j,1} := \int_{t_{j-1}}^{t_j} a^{i}_{t_{j-1}} (\theta_{0,2}) d W^{i}_s,\\
%B^{i}_j &:= \int_{t_{j-1}}^{t_j} (b^{i}_s(\theta_{0,1}) - b^{i}_{t_{j-1}}(\theta_{0,1})) d s,
%\end{align*}
%which satisfy
$\E_{t_{j-1}} [(A^{i}_{j,k})^{4}] = R^{i}_{t_{j-1}}(\Delta_n^{2k})$, $k=1,2$, and $\E_{t_{j-1}} [(B^{i}_j)^4] = R^{i}_{t_{j-1}}(\Delta_n^6)$.
We note that 
$$ 
\E_{t_{j-1}}[(A^{i_1}_{j,1})^2 (A^{i_2}_{j,1})^2] = \Delta_n c^{i_1}_j (\theta_{0,2}) \Delta_n c^{i_2}_j (\theta_{0,2}). 
$$ 
Moreover, we have

\begin{align*}
%&
\E_{t_{j-1}} [ (A^{i_1}_{j,1})^2 A^{i_2}_{j,1} A^{i_2}_{j,2} ] 
%\\&\qquad= \E_{t_{j-1}} \Big[ \Big( \int_{t_{j-1}}^{t_j} a^{i_1}_{t_{j-1}} (\theta_{0,2}) d W^{i_1}_s \Big)^2 \, \int_{t_{j-1}}^{t_j} a^{i_2}_{t_{j-1}} (\theta_{0,2}) d W^{i_2}_s \, \int_{t_{j-1}}^{t_j} (a^{i_2}_s (\theta_{0,2}) - a^{i_2}_{t_{j-1}} (\theta_{0,2}) ) d W^{i_2}_s \Big]\\&\qquad
= 
c^{i_1}_{t_{j-1}} (\theta_{0,2}) a^{i_2}_{t_{j-1}} (\theta_{0,2}) 
V^{i_1,i_2}_j,
\end{align*}
where independence of Brownian motions together with It\^o isometry implies 
\begin{align}
V^{i_1,i_2}_j &:= \E_{t_{j-1}} \Big[ (W^{i_1}_{t_j}-W^{i_1}_{t_{j-1}})^2 \int_{t_{j-1}}^{t_j} d W^{i_2}_s \int_{t_{j-1}}^{t_j} (a^{i_2}_s(\theta_{0,2}) - a^{i_2}_{t_{j-1}}(\theta_{0,2})) d W^{i_2}_s \Big]\nonumber\\
&= \int_{t_{j-1}}^{t_j} \E_{t_{j-1}} [ (W^{i_1}_{t_j}-W^{i_1}_{t_{j-1}})^2 (a^{i_2}_t (\theta_{0,2}) - a^{i_2}_{t_{j-1}}(\theta_{0,2}))] d t.\label{eq:ri1i2j}
\end{align}
Assumption \textbf{\ref{as7}} allows us to apply It\^o's lemma to $a_t^{i_2}(\theta_{0,2})$. We get that the conditional expectation in \eqref{eq:ri1i2j} equals

\begin{align*}
&\E_{t_{j-1}} \Big[ (W^{i_1}_{t_j}-W^{i_1}_{t_{j-1}})^2 \int_{t_{j-1}}^t \sum_{k{=1}}^{{N}} \Big( b^k_s (\theta_{0,1}) \partial_{x_k} a^{i_2}_s (\theta_{0,2}) + \frac{1}{2} c^k_s (\theta_{0,2}) \partial^2_{x_k} a^{i_2}_s (\theta_{0,2}) \Big) d s \Big] \\
&\qquad+ \E_{t_{j-1}} \Big[ (W^{i_1}_{t_j}-W^{i_1}_{t_{j-1}})^2 \int_{t_{j-1}}^t \sum_{k{=1}}^{{N}} a^{k}_s (\theta_{0,2}) \partial_{x_k} a^{i_2}_s (\theta_{0,2}) d W_s^k \Big].
\end{align*}
The first term is clearly a $R^{i_1,i_2}_{t_{j-1}} (\Delta_n^2)$ function. Regarding the second one, for $k \neq i_1$, the independence of the Brownian motions makes it directly equal to $0$. For $k = i_1$, instead, we have 
$$
\E_{t_{j-1}} \Big[ (W^{i_1}_{t_j}-W^{i_1}_{t_{j-1}})^2 \int_{t_{j-1}}^t a^{i_1}_s (\theta_{0,2}) \partial_{x_{i_1}} a^{i_2}_s (\theta_{0,2}) d W_s^{i_1} \Big],
$$
%Because of 
where under \textbf{\ref{as7}} we obtain 
$$
\partial_{x_{i_1}} a^{i_2}_s (\theta_{0,2}) := \partial_y \tilde a%_{\theta_{0,2}} 
\Big( X^{i_2}_s, \frac{1}{N} \sum_{l{=1}}^{{N}} K%_{\theta_{0,2}} 
(X^{i_2}_s,X^l_s) \Big) \frac{1}{N} \partial_y K%_{\theta_{0,2}} 
(X^{i_2}_s,X^{i_1}_s)
$$
with
%and 
$\partial_y \tilde a%_{\theta_{0,2}}
$, $\partial_y K%_{\theta_{0,2}}
$ having %the last function has 
polynomial growth.
Using the Cauchy-Schwarz inequality, it follows that the above quantity  is upper bounded by 
\begin{align*}
&\Big( 3\Delta_n^2 \E_{t_{j-1}} \Big[\Big( \int_{t_{j-1}}^t a^{i_1}_s (\theta_{0,2}) \partial_{x_{i_1}} a^{i_2}_s (\theta_{0,2}) d W_s^{i_1}\Big)^2 \Big] \Big)^\frac{1}{2} \\
&\qquad = \Big( 3 \Delta_n^2 %R^{i_1}_{t_{j-1}} (\Delta) \Big( 
\int_{t_{j-1}}^t \E_{t_{j-1}} [ (  a^{i_1}_s (\theta_{0,2}) \partial_{x_{i_1}} a^{i_2}_s (\theta_{0,2}) )^2 ] ds \Big)^\frac{1}{2} 
%&= R^{i_1}_{t_{j-1}} (\Delta) \Big(\int_{t_{j-1}}^t \frac{1}{N^2} R^{i_1, i_2,{\color{blue}s}}_{t_{j-1}} (1) ds \Big)^\frac{1}{2} \\
= \frac{1}{N} R^{i_1, i_2}_{t_{j-1}} (\Delta_n^\frac{3}{2}).
\end{align*}
It implies
\begin{equation}\label{eq:ito}
\E_{t_{j-1}}[(A^{i_1}_{j,1})^2 A^{i_2}_{j,1} A^{i_2}_{j,2}] = R^{i_1,i_2}_{t_{j-1}} (\Delta_n^3)  + \frac{1}{N} R^{i_1,i_2}_{t_{j-1}} (\Delta_n^\frac{5}{2}).
\end{equation}
We conclude that %the conditional expectation on the right hand side of \eqref{eq:cexpected22} equals
\begin{align*}
&\E_{t_{j-1}} [ (X^{i_1}_{t_j} - X^{i_1}_{t_{j-1}} - \Delta_n b^{i_1}_{t_{j-1}}(\theta_{0,1}))^2 (X^{i_2}_{t_j} - X^{i_2}_{t_{j-1}} - \Delta_n b^{i_2}_{t_{j-1}}(\theta_{0,1}))^2]\\ 
&\qquad = \Delta_n c^{i_1}_j (\theta_{0,2}) \Delta_n c^{i_2}_j (\theta_{0,2}) + R^{i_1,i_2}_{t_{j-1}}(\Delta_n^3)  + \frac{1}{N} R^{i_1,i_2}_{t_{j-1}} (\Delta_n^\frac{5}{2}),    
\end{align*}
whence 
\begin{equation}\label{eq:condD12}
\E_{t_{j-1}}[D^{i_1}_j D^{i_2}_j] = R^{i_1,i_2}_{t_{j-1}}(\Delta_n^3) + \frac{1}{N} R^{i_1,i_2}_{t_{j-1}} (\Delta_n^\frac{5}{2})    
\end{equation}
if $i_1 \neq i_2$. Finally, we plug \eqref{eq:condD2}, \eqref{eq:condD12} back into \eqref{eq: norm theta2 start}, where 
use of the conditions $N \Delta_n \to 0$, $\Delta_n \to 0$ and Lemma \ref{l: Riemann} completes the proof of the convergence in \eqref{e: norm theta2}.\\

\noindent
$\bullet$ Proof of \eqref{e: negl theta}, second convergence. \\
We prove it for $r = 2$. We use the same notation as in \eqref{eq: norm theta2 start} and rewrite the left hand side of \eqref{e: negl theta} as 
\begin{equation}\label{eq:sumCD}
\frac{\Delta_n^2}{N^2} \sum_{j=1}^n \sum_{i_1, i_2, i_3, i_4=1}^N \Delta_n^{-4} C^{i_1}_{j,\Tilde{h}} C^{i_2}_{j,\Tilde{h}} C^{i_3}_{j,\Tilde{h}} C^{i_4}_{j,\Tilde{h}} \E_{t_{j - 1}} [D_j^{i_1} D_j^{i_2} D_j^{i_3} D_j^{i_4}].    \end{equation}
We have $\E_{t_{j-1}}[(D^{i}_j)^4] = R^{i}_{t_{j-1}}(\Delta_n^{4})$ and $\operatorname{card}(I^c) = O(N^3)$, where $I$ denotes a set of all $i=(i_1,i_2,i_3,i_4) \in \{1,\dots,N\}^4$ such that $i_1,i_2,i_3,i_4$ are pairwise distinct. In \eqref{eq:sumCD} the sum over $i \in I^c$ converges to $0$ in $L^1$ and so in probability since $N\Delta_n \to 0$. In case $i \in I$  we use the decomposition 
\begin{align*}
D^i_j &= (A^i_{j,1} + A^i_{j,2} + B^i_j)^2 - \Delta_n c^i_{t_{j-1}}(\theta_{0,2})\\
&= (A^i_{j,2}+B^i_j)(2 A^i_{j,1}+A^i_{j,2} + B^i_j) + (A^i_{j,1})^2 - \Delta_n c^i_{t_{j-1}} (\theta_{0,2}).
\end{align*}
We note that
\begin{equation*}
\begin{gathered}
\E_{t_{j-1}} [((A^i_{j,1})^2- \Delta_n c^i_{t_{j-1}}(\theta_{0,2}))^4] = R^i_{t_{j-1}}(\Delta_n^4),\\
\E_{t_{j-1}} [(A^i_{j,k})^8] = R^i_{t_{j-1}} (\Delta_n^{4k}), \ k=1,2, \qquad
%\E_{t_{j-1}} [(A^i_{j,2})^8] = R^i_{t_{j-1}} (\Delta^8), \qquad 
\E_{t_{j-1}} [(B^i_j)^8] = R^i_{t_{j-1}} (\Delta_n^{12}).  
\end{gathered}
\end{equation*}
Moreover, because of the independence of Brownian motions, we have
\begin{align*}
\E_{t_{j-1}} \Big[ \prod_{k=1}^4 ((A^{i_k}_{j,1})^2- \Delta_n c^{i_k}_{t_{j-1}}(\theta_{0,2})) \Big] = 0
%[&((A^{i_1}_{j,1})^2- \Delta c^{i_1}_{t_{j-1}}(\theta_{0,2}))((A^{i_2}_{j,1})^2- \Delta c^{i_2}_{t_{j-1}}(\theta_{0,2}))\\
%\times &((A^{i_3}_{j,1})^2- \Delta c^{i_3}_{t_{j-1}}(\theta_{0,2}))((A^{i_4}_{j,1})^2- \Delta c^{i_4}_{t_{j-1}}(\theta_{0,2}))] = 0
\end{align*}
and in a similar manner as in \eqref{eq:ito} under \textbf{\ref{as7}} we have
\begin{align*}
&\E_{t_{j-1}} \Big[ A^{i_1}_{j,2} A^{i_1}_{j,1} \prod_{k=2}^4 ( (A^{i_k}_{j,1})^2 - \Delta_n c^{i_k}_{t_{j-1}} (\theta_{0,2}) ) \Big]\\ 
&\qquad = a^{i_1}_{t_{j-1}} (\theta_{0,2}) \prod_{k=2}^4 c^{i_{k}}_{t_{j-1}} (\theta_{0,2}) \int_{t_{j-1}}^{t_j} \E_{t_{j-1}} \Big[ (a^{i_1}_s (\theta_{0,2}) - a^{i_1}_{t_{j-1}} (\theta_{0,2})) \prod_{l=2}^4 ((W^{i_{l}}_{t_j} - W^{i_{l}}_{t_{j-1}})^2 - \Delta_n) \Big] d s\\ 
&\qquad= R_{t_{j-1}}^{i_1,i_2,i_3,i_4}(\Delta_n^5) + \frac{1}{N} R_{t_{j-1}}^{i_1,i_2,i_3,i_4} (\Delta_n^{\frac{9}{2}}),
\end{align*}
whence it follows
$$
\E_{t_{j-1}} [D^{i_1}_j D^{i_2}_j D^{i_3}_j D^{i_4}_j] = R^{i_1,i_2,i_3,i_4}_{t_{j-1}} (\Delta_n^5) + \frac{1}{N} R_{t_{j-1}}^{i_1,i_2,i_3,i_4} (\Delta_n^{\frac{9}{2}}).
$$
 We recall that $\operatorname{card}(I) = O(N^4)$. Since $N\Delta_n \to 0$, $\Delta_n \to 0$, the sum over $i \in I$ in \eqref{eq:sumCD} converges to $0$ in $L^1$ and so in probability.\\

%We prove \eqref{e: negl theta} for $r = 2$. We have %that 
%$$\sum_{j = 1}^n \E_{j - 1} [(\xi_j^i (\theta_0))^4] = \frac{1}{(Nn)^2} \sum_{j = 1}^n \sum_{i_1, i_2, i_3, i_4 = 1}^N \E_{j - 1} [f_j^{i_1}(\theta_0) f_j^{i_2}(\theta_0) f_j^{i_3}(\theta_0) f_j^{i_4}(\theta_0)],$$
%where $f_j^{i}(\theta_0) = \frac{(X^i_{t_j} - X^i_{t_{j-1}} - \Delta b^i_{t_{j - 1}}(\theta_{1, 0}) )^2}{ \Delta (c^i_{t_{j - 1})^2}(\theta_{2, 0})} \partial_{\theta_2} c^i_{t_{j - 1}}(\theta_{2, 0}) + \frac{\partial_{\theta_2} c^i_{t_{j - 1}}(\theta_{2, 0}) }{c^i_{t_{j - 1}}(\theta_{2, 0}) }$. We observe that for any $i \in \{ 1, ... , N \}$ and any $p \ge 2$ it is 
%\begin{align}{\label{e: bound f}}
%\E_{j - 1} [|f_j^{i}(\theta_0)|^p] & \le  R_{t_{j - 1}}(\theta_0, \Delta^{- p}) \E_{j - 1} [(X^i_{t_j} - X^i_{t_{j-1}} - \Delta b^i_{t_{j - 1}}(\theta_{1, 0}) )^{2p}] + R_{t_{j - 1}}(\theta_0, 1) \\
%& \le R_{t_{j - 1}}(\theta_0, 1), \nonumber
%\end{align}
%having used the fifth point of Lemma \ref{l: conditional expectation}. Using Cauchy-Schwarz inequality and \eqref{e: bound f} it follows that
%$$\sum_{j = 1}^n \E_{j - 1} [(\xi_j^i (\theta_0))^4] \le \frac{1}{(Nn)^2} \sum_{j = 1}^n \sum_{i_1, i_2, i_3, i_4 = 1}^N R_{t_{j - 1}}(\theta_0, 1).$$
%In norm $L^1$ it is bounded by $c N^2 \Delta = c (N \Delta^{\frac{1}{2}})^2$ and so it converges to $0$, as we have assumed $N \Delta^{\frac{1}{2}} \rightarrow 0$ for $N, n \rightarrow \infty$. \\
%\\
\noindent
$\bullet$ Proof of \eqref{e: norm mixed}. \\
We rewrite the left hand side of \eqref{e: norm mixed} as
\begin{equation}\label{eq:mixed}
\frac{\Delta_n^{\frac{1}{2}}}{N} \sum_{j=1}^n \sum_{i_1,i_2=1}^N \E_{t_{j-1}} [ (A^{i_1}_{j,1}+A^{i_1}_{j,2}+B^{i_1}_j) D^{i_2}_j ] \Delta_n^{-1} C^{i_2}_{j,\Tilde{h}} F^{i_1}_{j,h},
\end{equation}
where
\begin{align*}
D^{i}_j %&= (A^{i}_{j,1} + A^{i}_{j,2} +B^{i}_j)^2 - \Delta c^{i}_{t_{j-1}} (\theta_{0,2})\\ 
&= (A^{i}_{j,2} +B^{i}_j)(2 A^{i}_{j,1} + A^{i}_{j,2} +B^{i}_j) + (A^{i}_{j,1})^2 - \Delta_n c^{i}_{t_{j-1}} (\theta_{0,2})
\end{align*}
with the notations introduced above. We recall that 
$F^i_{j,h} = R^{i}_{t_{j-1}}(1)$,
$C^i_{j,\Tilde{h}} = R^{i}_{t_{j-1}}(1)$, 
$\E_{t_{j-1}} [(B^{i}_j)^4] = R^{i}_{t_{j-1}}(\Delta_n^6)$, $\E_{t_{j-1}}[(A^{i}_{j,k})^4] = R^{i}_{t_{j-1}}(\Delta_n^{2k})$, $k=1,2$, and so $\E_{t_{j-1}}[((A^{i}_{j,1})^2 - \Delta_n c^{i}_{t_{j-1}} (\theta_{0,2}))^2] = R^{i}_{t_{j-1}}(\Delta_n^2)$. We note that
$$
\E_{t_{j-1}} [A^{i_1}_{j,1}((A^{i_2}_{j,1})^2 - \Delta_n c^{i_2}_{t_{j-1}}(\theta_{0,2}))] = 0
$$
for all $i_1,i_2$.  This is a consequence of the independence of the Brownian motions for $i_1 \neq i_2$, while for $i_1 = i_2$ it derives from the fact that the odd moments are centered. Hence, in case $i_1=i_2 = i$ the term $\E_{t_{j-1}} [(A^{i}_{j,1})^2 A^{i}_{j,2}]$ makes the main contribution to
$$
\E_{t_{j-1}} [ (A^{i}_{j,1}+A^{i}_{j,2}+B^{i}_j) D^{i}_j ] = R^{i}_{t_{j-1}}(\Delta_n^2).
$$
Now we can see that the sum over $i_1=i_2$ in \eqref{eq:mixed} converges to $0$ in $L^1$ and so in probability. In case $i_1 \neq i_2$ we have
$$
\E_{t_{j-1}} [A^{i_1}_{j,2}((A^{i_2}_{j,1})^2 - \Delta_n c^{i_2}_{t_{j-1}}(\theta_{0,2}))] = 0. 
$$
Moreover,
\begin{align*}
%&
\E_{t_{j-1}} [A^{i_1}_{j,1} A^{i_2}_{j,1} A^{i_2}_{j,2}] %\\
%& = \E_{t_{j-1}} \Big[a^{i_1}_{t_{j-1}} (\theta_{0,2})(W^{i_1}_{t_j} - W^{i_1}_{t_{j-1}}) a^{i_2}_{t_{j-1}} (\theta_{0,2}) (W^{i_2}_{t_j} - W^{i_2}_{t_{j-1}}) \int_{t_{j-1}}^{t_j} (a^{i_2}_s (\theta_{0,2}) - a^{i_2}_{t_{j-1}} (\theta_{0,2}) dW_s^{i_2} \Big] \\&
= a^{i_1}_{t_{j-1}} (\theta_{0,2}) a^{i_2}_{t_{j-1}} (\theta_{0,2}) \int_{t_{j-1}}^{t_j} \E_{t_{j-1}} [(W^{i_1}_{t_j} - W^{i_1}_{t_{j-1}}) (a^{i_2}_s (\theta_{0,2}) - a^{i_2}_{t_{j-1}} (\theta_{0,2}) ] ds.
\end{align*}
The application of It\^o's lemma to $a_s^{i_2}(\theta_{0,2})$ under \textbf{\ref{as7}} similarly as in the proof of \eqref{eq:ito} provides 
$$\E_{t_{j-1}} [A^{i_1}_{j,1} A^{i_2}_{j,1} A^{i_2}_{j,2}] = R^{i_1,i_2}_{t_{j-1}} (\Delta_n^{\frac 5 2}) + \frac{1}{N} R^{i_1,i_2}_{t_{j-1}} (\Delta_n^{2}). $$
We conclude that
$$
\E_{t_{j-1}} [(A^{i_1}_{j,1}+A^{i_1}_{j,2}+B^{i_1}_j) D^{i_2}_j] = R^{i_1,i_2}_{t_{j-1}} (\Delta_n^{\frac 5 2})  + \frac{1}{N} R^{i_1,i_2}_{t_{j-1}} (\Delta_n^{2}).
$$
in case $i_1 \neq i_2$. Hence, the sum over $i_1 \neq i_2$ in \eqref{eq:mixed} converges to $0$ in $L^1$ and so in probability when $N \Delta_n \to 0$, $\Delta_n \to 0$. 
This concludes the proof of the asymptotic normality of $-%M^N_n
\nabla_{\theta}S^N_n (\theta_0) {M^N_n}$. 
\end{proof}

\subsection{Proof of Proposition \ref{p: second derivatives contrast}}
\begin{proof}
%[Proof of Point 1] 
The proof relies on the computation of the second derivatives of the contrast function. We have that, for any $k,l = 1, ... , p_1$,
\begin{align*}
\partial_{\theta_{1, k}} \partial_{\theta_{1, l}} S_n^N (\theta) = 2 \sum_{i = 1}^N \sum_{j = 1}^n \Big\{ \Delta_n &\frac{\partial_{\theta_{1, k}} b^i_{t_{j - 1}}(\theta_1) \partial_{\theta_{1, l}} b^i_{t_{j - 1}}(\theta_1)}{c^i_{t_{j - 1}}(\theta_2)}\\ 
- &\frac{\partial_{\theta_{1, k}} \partial_{\theta_{1, l}} b^i_{t_{j - 1}}(\theta_1)}{c^i_{t_{j - 1}}(\theta_2)} (X^i_{t_j} - X^i_{t_{j-1}} - \Delta_n b^i_{t_{j - 1}}(\theta_1) )\Big\},    
\end{align*}
where the last factor can further be decomposed into $\Delta_n (b^i_{t_{j-1}}(\theta_{0,1}) - b^i_{t_{j-1}}(\theta_1))$ and $X^i_{t_{j}} - X^i_{t_{j-1}} - \Delta_n b^i_{t_{j-1}}(\theta_{0,1})$. We can see that $\partial_{\theta_{1, k}} \partial_{\theta_{1, l}} S_n^N (\theta)/N$ converges to
\begin{align}\label{def:B11}
\Sigma^{(1)}_{kl} (\theta) := 2 \int_0^1 \int_{\R}  \Big\{ &\frac{\partial_{\theta_{1, k}} b(\theta_1, x,\bar \mu_t) \partial_{\theta_{1, l}} b(\theta_1, x,\bar \mu_t)}{c(\theta_2, x,\bar \mu_t)}  \\
- &\frac{\partial_{\theta_{1, k}} \partial_{\theta_{1, l}} b(\theta_1, x,\bar \mu_t)}{c(\theta_2, x,\bar \mu_t)} (b(\theta_{0,1}, x,\bar \mu_t) - b(\theta_1, x,\bar \mu_t))  \Big\} \bar \mu_t (d x) d t \nonumber
\end{align}
uniformly in $\theta$ in probability. Indeed, the proof follows along the lines of the proof of \eqref{lim:con2}. We refer to Steps 3, 4 of the proof of Lemma \ref{lemma:cns}, where in  \eqref{def:INnrhoNn} in $I^N_n(\theta)$, $\rho^N_n (\theta)$  it is enough to replace the functions $h(\theta,\cdot)$ {and} $g(\theta,\cdot)$ with the integrand of \eqref{def:B11} and $\partial_{\theta_{1, k}} \partial_{\theta_{1, l}} b (\theta_1,\cdot) /c(\theta_2,\cdot)$ respectively, and to check them for the respective conditions. We note that both functions have polynomial growth. Moreover, the integrand in \eqref{def:B11} is locally Lipschitz continuous, which allows us to apply Lemma~\ref{l: Riemann} and yields the convergence in probability of the sequence $\partial_{\theta_{1,k}} \partial_{\theta_{1, l}} S_n^N (\theta)/N$ for every $\theta$. To get tightness in $(C(\Theta{; \R}), \| \cdot\|_\infty)$, we use that uniformly in $\theta$ the partial derivatives with respect to $\theta_{i',j'}$,  $j'=1,\dots,p_{i'}$, $i'=1,2$, of the two functions have polynomial growth.

In the same way as above we get that for any $k = 1, ... , p_1$, $l = 1, ... , p_2$, once multiplied by $\sqrt{\Delta_n}/N$, 
$$
\partial_{\theta_{1,k}} \partial_{\theta_{2,l}} S_n^N (\theta) = 2 \sum_{i = 1}^N \sum_{j = 1}^n \frac{\partial_{\theta_{1,k}} b^i_{t_{j - 1}}(\theta_1) \partial_{\theta_{2,l}} c^i_{t_{j - 1}}(\theta_2)}{(c^i_{t_{j - 1}}(\theta_2))^2} (X^i_{t_j} - X^i_{t_{j-1}} - \Delta_n b^i_{t_{j - 1}}(\theta_1) ),
$$
converges to $
0$ uniformly in $\theta$ in probability. 

Finally, we have that for any $k,l = 1,\dots, p_2$, 
\begin{align*}
\partial_{\theta_{2,k}} \partial_{\theta_{2,l}} S^N_n (\theta) = \sum_{i=1}^N \sum_{j=1}^n &\Big\{
\frac{\partial_{\theta_{2,k}} \partial_{\theta_{2,l}} c^i_{t_{j-1}}(\theta_2) c^i_{t_{j-1}}(\theta_2) - \partial_{\theta_{2,k}} c^i_{t_{j-1}}(\theta_2) \partial_{\theta_{2,l}} c^i_{t_{j-1}}(\theta_2)}{(c^i_{t_{j-1}}(\theta_2))^2}\\
&+ \frac{2 \partial_{\theta_{2,k}} c^i_{t_{j-1}}(\theta_2)\partial_{\theta_{2,l}} c^i_{t_{j-1}}(\theta_2)-\partial_{\theta_{2,k}} \partial_{\theta_{2,l}} c^i_{t_{j-1}}(\theta_2)c^i_{t_{j-1}}(\theta_2)}{\Delta_n (c^i_{t_{j-1}}(\theta_2))^3}\\  
&\times (X^i_{t_j}-X^i_{t_{j-1}}- \Delta_n b^i_{t_{j-1}}(\theta_1))^2
\Big\},
\end{align*}
where the last factor can further be decomposed into $(X^i_{t_j}-X^i_{t_{j-1}}- \Delta_n b^i_{t_{j-1}}(\theta_1))^2 - \Delta_n c^i_{t_{j-1}}(\theta_{0,2})$ and $\Delta_n c^i_{t_{j-1}}(\theta_{0,2})$. We note that $ (\Delta_n/N) \partial_{\theta_{2,k}} \partial_{\theta_{2,l}} S^N_n(\theta)$ converges to
\begin{align*}
%{\color{blue}\Sigma}_{p_1+k,p_1+l}
\Sigma^{(2)}_{kl} (\theta) := \int_0^T \int_{\R} &\Big\{  \frac{\partial_{\theta_{2,k}} \partial_{\theta_{2,l}} c(\theta_2, x,\bar \mu_t)c(\theta_2, x,\bar \mu_t)-\partial_{\theta_{2,k}} c(\theta_2, x,\bar \mu_t)\partial_{\theta_{2,l}} c(\theta_2, x,\bar \mu_t)}{c(\theta_2, x,\bar \mu_t)^2}\\
&+ \frac{2 \partial_{\theta_{2,k}} c(\theta_2, x,\bar \mu_t)\partial_{\theta_{2,l}} c(\theta_2, x,\bar \mu_t) - \partial_{\theta_{2,k}} \partial_{\theta_{2,l}} c(\theta_2, x,\bar \mu_t)c(\theta_2, x,\bar \mu_t)}{c(\theta_2, x,\bar \mu_t)^3}\\ &\times c(\theta_{0,2}, x,\bar \mu_t) \Big\} \bar \mu_t (d x) d t
\end{align*}
uniformly in $\theta$ in probability. We will prove the uniform in $\theta$ convergence to the second term of $%{\color{blue}\Sigma}_{p_1+k,p_1+l}
\Sigma^{(2)}_{kl}(\theta)$ only: 
\begin{equation}\label{lim:B221}
\sum_{j=1}^n \chi^N_{n,j} (\theta) \xrightarrow{\P} \tilde %{\color{blue}\Sigma}_{p_1+k,p_1+l} 
\Sigma^{(2)}_{kl} (\theta) := \int_0^T \int_{\R} \tilde f (\theta_2, x,\bar \mu_t) c(\theta_{0,2}, x,\bar\mu_t) \bar \mu_t (d x) d t,    
\end{equation}
where 
$$
\chi^N_{n,j}(\theta) = \frac{1}{N} \sum_{i=1}^N \tilde f^i_{t_{j-1}} (\theta_2) (X^i_{t_j} - X^i_{t_{j-1}} - \Delta_n b^i_{t_{j-1}}(\theta_1))^2
$$
and function $\tilde f : \Theta_2 \times \R \times {\cal P} \to \R$ is given by $(2 (\partial_{\theta_{2,k}} c) (\partial_{\theta_{2,l}} c) - (\partial_{\theta_{2,k}} \partial_{\theta_{2,l}} c)  c) / c^3$.
For every $\theta$ the convergence in \eqref{lim:B221} follows from
$$
\sum_{j=1}^n \E_{t_{j-1}} [\chi^N_{n,j}(\theta)] \xrightarrow{\P} \tilde %{\color{blue}\Sigma}_{p_1+k,p_1+l} 
\Sigma^{(2)}_{kl} (\theta), \qquad \sum_{j=1}^n \E_{t_{j-1}}[(\chi^N_{n,j}(\theta))^2] \xrightarrow{\P} 0
$$
by  \cite[Lemma 9]{GenJac93}. Indeed, the above relations hold, because by  Lemma \ref{l: conditional expectation}(1), 
$$
\E_{t_{j-1}} [\chi^N_{n,j} (\theta)] = \frac{1}{N} \sum_{i=1}^N \tilde f^i_{t_{j-1}} (\theta_2) (\Delta_n c^i_{t_{j-1}} (\theta_{0,2}) + R^i_{t_{j-1}}(\Delta_n^{3/2})),
$$ 
by Jensen's inequality and Lemma \ref{l: conditional expectation}(2),
$$
\E_{t_{j-1}} [(\chi^N_{n,j}(\theta))^2] \le \frac{1}{N} \sum_{i=1}^N (\tilde f^i_{t_{j-1}}(\theta_2))^2 R^i_{t_{j-1}}(\Delta_n^2),
$$
by polynomial growth of $\partial^{i'}_{\theta_{2,j'}} c (\theta_2,\cdot)$, $i'=0,1,2$, $j'=1,\dots,p_2$, \textbf{\ref{as3}} and Point 1.\ of Lemma~\ref{l: moments},
$$
(\tilde f^i_{t_{j-1}}(\theta_2))^2 = R^i_{t_{j-1}}(1).%,
$$
%where we denote random variables by $R^i_{t_{j-1}}(\Delta^d)$ if they are ${\cal F}_{t_{j-1}}^N$-measurable and divided by $\Delta^d$ are uniformly in $j,i$ and $n,N$ bounded in $L^p$ for every $p$. 
The tightness in $(C(\Theta{;\R}), \| \cdot \|_\infty)$ follows from 
$\E [ \sup_\theta \|\nabla_\theta \sum_{j{=1}}^{{n}} \chi^N_{n,j}(\theta) \| ] = O(1)$. Indeed, we have  
\begin{align*}
\nabla_{\theta_1} \chi^N_{n,j} (\theta) &= - 2 \frac{\Delta_n}{N} \sum_{i=1}^N \nabla_{\theta_1} b^i_{t_{j-1}} (\theta_1)  \tilde f^i_{t_{j-1}} (\theta_2) (X^i_{t_j} - X^i_{t_{j-1}} - \Delta_n b^i_{t_{j-1}}(\theta_1) ),\\
\nabla_{\theta_2} \chi^N_{n,j} (\theta) &= \frac{1}{N} \sum_{i=1}^N \nabla_{\theta_2} \tilde f^i_{t_{j-1}}(\theta_2) (X^i_{t_j} - X^i_{t_{j-1}} - \Delta_n b^i_{t_{j-1}}(\theta_1))^2,
\end{align*}
where by polynomial growth of $\sup_{\theta_1} \|\nabla_{\theta_1} b (\theta_1,\cdot)\|$, $\sup_{\theta_2} |\partial_{\theta_{2,j'}}^{i'} c (\theta_2,\cdot)|$, $i'=0,1,2,3$, $j'=1,\dots,p_2$, and \textbf{\ref{as3}},
$$
\sup_\theta \|\nabla_{\theta_{1}} b^i_{t_{j-1}}(\theta_1)  \tilde f^i_{t_{j-1}} (\theta_2) \| = R^i_{t_{j-1}} (1), \qquad \sup_{\theta_2} \| \nabla_{\theta_2} \tilde f^i_{t_{j-1}}(\theta_2)\| = R^i_{t_{j-1}} (1)
$$
and 
$$
\sup_{\theta_1} |X^i_{t_j} - X^i_{t_{j-1}} - \Delta_n b^i_{t_{j-1}}(\theta_1)| \le  |X^i_{t_j}-X^i_{t_{j-1}}| + \Delta_n \sup_{\theta_1} |b^i_{t_{j-1}}(\theta_1)|
$$
with $\sup_{\theta_1} |b^i_{t_{j-1}}(\theta_1)| = R^i_{t_{j-1}}(1)$. Finally, we have
$\E [|X^i_{t_j}-X^i_{t_{j-1}}|^4] \le C \Delta_n^2$ uniformly in $i,j$ and $N,n$ by Lemma~\ref{l: moments}(2).

We conclude that the matrix 
$\Sigma^N_n (\theta)$ converges to $\Sigma(\theta) = %({\color{blue}\Sigma}_{kl} (\theta)) 
\operatorname{diag}(\Sigma^{(1)}(\theta),\Sigma^{(2)}(\theta))$ uniformly in $\theta$ and so at $\theta = \theta_0$ in probability. Hence,
\begin{align*}
\| \Sigma^N_n (\theta_0 + s (\hat \theta^N_n - \theta_0)) - \Sigma^N_n (\theta_0) \| 
&\le  o_\P (1) + \| \Sigma(\theta_0 + s (\hat \theta^N_n - \theta_0)) - \Sigma(\theta_0) \|,
\end{align*}
where  the uniform convergence in probability (in $s$) of the last term to $0$ follows from continuity of $\Sigma (\theta)$ at $\theta = \theta_0$ and consistency of the estimator sequence $\hat \theta^N_n$. 
\end{proof}

\section{Proof of technical results}{\label{s: proof technical}}

\subsection{Proof of Lemma \ref{l: moments}}

\begin{proof} 
\noindent
Proof of Lemma \ref{l: moments}(1). \\
We have, for any $i = 1, \dots , N$, $0 \le t \le T$, $p \ge 2$, 
\begin{align*}
\E[|X_t^i|^p] & \le \E \Big[ \Big|X_0^i + \int_0^t b_u^i(\theta_{0,1}) du + \int_0^t a_u^i(\theta_{0,2}) dW_u^i \Big|^p \Big] \\
& \le C \Big( \E[|X_0^i|^p] + t^{p - 1} \int_0^t \E[|b_u^i(\theta_{0,1})|^p] d u +  t^{\frac{p}{2} - 1} \int_0^t \E[|a_u^i(\theta_{0,2})|^p] d u \Big),
\end{align*}
where we have used the Burkholder-Davis-Gundy and Jensen inequalities. We observe that, as a consequence of the lipschitzianity gathered in \textbf{\ref{as2}}, for the true value of the parameter both coefficients are upper bounded by $C(1 + |X_u^i| + W_2(\mu_u, \delta_0))$. %It follows 
%\begin{equation}
%\E[|X_t^i|^p] \le c \E[|X_0^i|^p] +c (|t|^{p - 1} + |t|^{\frac{p}{2} - 1} ) [|t| +  \int_0^t  \big( \E[|X_u^i|^p] + \E[|W_1(\mu_u, \delta_0)|^p] \big) du ].
%\label{eq: start moments}
%\end{equation}
Due to 
%Minkowsky's 
Jensen's inequality, we have 
%$$\E[|W_1(\mu_u, \delta_0)|^p] \le \big(\frac{1}{N} \sum_{j = 1}^N (\E[|X_u^j|^p])^\frac{1}{p}\big)^p.$$
$$\E[W_2^p(\mu_u, \delta_0)] \le \frac{1}{N} \sum_{j = 1}^N \E [ |X_u^j|^p ] 
%= \E[|X_u^1|^p] 
= \E[|X_u^i|^p].
$$
The last identity follows from the fact that the particles are equally distributed. %It implies 
%$$\E[W^p_{\color{blue}2} (\mu_u, \delta_0)] \le \big(\frac{1}{N} \sum_{j = 1}^N (\E[|X_u^1|^p])^\frac{1}{p}\big)^p = \E[|X_u^1|^p] = \E[|X_u^i|^p].$$
%We plug it in \eqref{eq: start moments}, obtaining 
We obtain
\begin{equation}{\label{eq: end for p larger than 2}}
\E[|X_t^i|^p] \le C \Big( \E[|X_0^i|^p] +  (t^{p - 1} + t^{\frac{p}{2} - 1} ) \Big( t +  2\int_0^t  \E[|X_u^i|^p] du \Big)
\Big).
\end{equation}
We infer by Gronwall's lemma that 
$$
\E[|X_t^i|^p] \le 
%C( \E[|X_0^i|^p] +  t^{p} +  t^{\frac{p}{2}} ) \exp ( t^p +  t^{\frac{p}{2}} ).
C( \E[|X_0^i|^p] +  {T^{p}} +  {T^{\frac{p}{2}}} ) \exp ( {C' (T^{p} +  T^{\frac{p}{2}})} ).
$$
As the constants do not depend on $t \le T$ and $\E[|X_0^i|^p] < \infty$ by \textbf{\ref{as1}}, we have the wanted result for $p \ge 2$. 
Then, by a Jensen argument and the boundedness of the moments for $p \ge 2$, it follows the result also for $p< 2$.
%Then, regarding the case $p= 1$, acting only on the drift as to get \eqref{eq: end for p larger than 2} we have
%\begin{align*}
%\E[|X_t^i|] & \le \E[|X_0^i|] +c |t|  +  2c \int_0^t  \E[|X_u^i|] du  + \E[|\int_0^t a_u^i(\theta_{0, 2}) dW_u^i|] \\
%& \le \E[|X_0^i|] +c |t| +  2c \int_0^t  \E[|X_u^i|^2]^\frac{1}{2} du  + \E[\int_0^t |a_u^i(\theta_{0, 2})|^2 du]^\frac{1}{2},
%\end{align*}
%where we have used Cauchy-Schwarz inequality and Ito. We now act on the last integral as we previously did, recalling that
%$|a_u^i(\theta_{0, 2})|^2 \le c(1 + |X_u^i|^2 + W_1(\mu_u, \delta_0)^2)$. \\
%Then, it is easy to see $\E[|X_t^i|]< \infty$, having already showed that $\E[|X_t^i|^p]< \infty$ for any $p \ge 2$ and having assumed that $\E[|X_0^i|]< \infty$ as in \ref{as1}. \\
%To conclude, we use the interpolation theorem (see below Theorem 1.7 in Chapter 4 of \cite{Interpolation}), to get the wanted result also for any $q \in [1, 2]$. Indeed, interpolation theorem provides
%\begin{align*}
%\E[|X_t^i|^q] \le (\E[|X_t^i|])^\theta (\E[|X_t^i|^2]^\frac{1}{2})^{1 -\theta},
%\end{align*}
%with $\frac{1}{q} = \theta + \frac{1 - \theta}{2}$, hence $\theta = \frac{1}{q} -1$. As we have already showed that both $\E[|X_t^i|]$ and $\E[|X_t^i|^2]$ are bounded, it follows the boundedness of any moments. Then, $\E[W_p^q(\mu_t, \delta_0)] < \infty$ is a straightforward consequence of Minkowsky's inequality and the fact that the process $X$ has bounded moments of any order.\\
\\

\noindent
Proof of Lemma \ref{l: moments}(2). \\
We have for any $0 \le s < t \le T$, $p \ge 2$,
\begin{align*}
\E[|X_t^i - X_s^i|^p] & = \E \Big[ \Big|\int_s^t b_u^i(\theta_{0,1}) du + \int_s^t a_u^i(\theta_{0,2}) dW_u^i \Big|^p \Big] \\
& \le C \Big( (t-s)^{p - 1}\int_s^t \E[|b^i_u(\theta_{0,1})|^p] du +  (t - s)^{\frac{p}{2}-1} \int_s^t \E[|a^i_u(\theta_{0,2})|^p] ds \Big),
\end{align*}
where we have used the Jensen and Burkholder-Davis-Gundy inequalities. Because of \eqref{eq: pol growth} and the just shown Lemma \ref{l: moments}(1), the result follows letting $t-s \le 1$. \\

\noindent
Proof of Lemma \ref{l: moments}(3). \\
According to the definition of $R^i_s(1)$, we want to evaluate the $L^q$ norm of $\E_s [|X^i_t-X^i_s|^p]$.
For any $0 \le s < t \le T$ such that $t-s\le 1$ and $p \ge 2$, $q \ge 1$, 
%$\E_s [|X^i_t-X^i_s|^p]$ is ${\cal F}_s$-measurable, moreover, 
$$
\E \big[ \big| \E_s [ | X^i_t-X^i_s|^p ] \big|^q \big]^{\frac 1 q} \le \E [ |X^i_t-X^i_s|^{pq} ]^{\frac 1 q}  \le C (t-s)^{\frac{p}{2}}
$$
follows by conditional Jensen's inequality and Lemma \ref{l: moments}(2).
\\

\noindent
Proof of Lemma \ref{l: moments}(4).\\ 
This is a straightforward consequence of
\begin{equation}\label{eq: bound W2}
W_2^p (\mu_t,\mu_s) \le \Big( \frac{1}{N} \sum_{j=1}^N |X^j_t-X^j_s|^2 \Big)^{\frac{p}{2}} \le \frac{1}{N} \sum_{j=1}^N |X^j_t - X^j_s|^p
\end{equation}
by Jensen's inequality for any $0 \le s < t \le T$ such that $t-s \le 1$, $p \ge 2$ and Lemma \ref{l: moments}(2).\\
%This is a straightforward consequence of \eqref{eq: bound W1}, \eqref{eq: Minkowski} and Point 2. \\

\noindent
Proof of Lemma \ref{l: moments}(5).\\ 
It follows directly from \eqref{eq: bound W2}, where we use Minkowski's inequality as follows: %\eqref{eq: bound W1}, 
%\eqref{eq: Minkowski} 
$$
\E \big[ \big| \E_s [ W^p_2 (\mu_t,\mu_s) ] \big|^q \big]^{\frac{1}{q}} \le \frac{1}{N} \sum_{j=1}^N \E \big[ \big| \E_s [ |X^j_t-X^j_s|^{pq} ] \big| \big]^{\frac{1}{q}},  
$$
and then Lemma \ref{l: moments}(3).
\end{proof}

\subsection{Proof of Lemma \ref{l: Riemann}}
\begin{proof} 
\noindent 
Step 1. We prove that 
$$
\frac{\Delta_n}{N} \sum_{i=1}^N \sum_{j=1}^n f(X^{i,N}_{t_{j{-1,n}}}, \mu^N_{t_{j{-1,n}}}) - \frac{1}{N} \sum_{i=1}^N \int_0^T f(X^{i,N}_s, \mu^N_s) d s \xrightarrow{L^1} 0.
$$
Here we note $\Delta_n = t_{j,n}-t_{j-1,n}$ and decompose the  above integral into integrals over $[t_{j-1,n},t_{j,n})$. We can see that the above convergence
follows from 
%the uniform in $j, i, N$ convergence
$$
\sum_{j=1}^n \int_{t_{j-1,n}}^{t_{j,n}} \E [ | f(X^{i,N}_{t_{j-1,n}}, \mu^N_{t_{j-1,n}}) - f(X^{i,N}_s, \mu^N_s) | ] d s \to 0, \qquad N,n \to \infty,
$$
for fixed $i$,
which in turn follows using
%To get the last convergence we use
the condition \eqref{cond:Riemann}, Cauchy-Schwarz inequality and moment bounds in  Lemma \ref{l: moments}(1), (2) and (4). In particular, $\E [|X^{i,N}_{t_{j-1,n}} - X^{i,N}_s|^2] \le C \Delta_n$ for  all $t_{j-1,n} \le s \le t_{j,n}$, $j$ and $n,N$. \\

\noindent
Step 2. Next, let us prove that 
$$
\frac{1}{N} \sum_{i=1}^N \int_{{0}}^{{T}} f(X^{i,N}_s, \mu^N_s) d s - \frac{1}{N} \sum_{i=1}^N \int_0^T f(\bar X^i_s, \bar \mu_s) d s \xrightarrow{L^1} 0, \qquad N \to \infty,
$$
{where each $(\bar X^i_t)_{t\in [0,T]}$ satisfies \eqref{eq: McK} with $(W_t)_{t \in [0,T]} = (W^i_t)_{t \in [0,T]}$ and $\bar X^i_0 = X^{i,N}_0$.}
It suffices to prove %the uniform in $i, s$ convergence
$$
\int_0^T \E [ | f(X^{i,N}_s, \mu^N_s ) - f(\bar X^i_s, \bar \mu_s ) |] d s \to 0,
$$
where $i$ is fixed and the integral is over a bounded interval.
For this purpose, let us use again the condition \eqref{cond:Riemann} and the Cauchy-Schwarz inequality.  Following the same arguments as in the proof of Lemma \ref{l: moments}(1) and Gronwall lemma, it is easy to show that for all $p>0$ there exists $ C_p > 0$ such that for all $s, i, N$ it holds $\E [|\bar X^i_s|^p ] <  C_p$.
 Moreover we have
$$
\E [ |X^{i,N}_s - \bar X^i_s|^2 ] \le  \frac{C}{\sqrt{N}}
$$ 
for all $0 \le s \le {T}$ and $i, N$,  thanks to Theorem 3.20 in \cite{Review prop}, based on Theorem 1 of \cite{145 review}. We remark that, from the boundedness of the moments, the quantity $q$ appearing in the statement of Theorem 3.20 in \cite{Review prop} is larger than $4$. Hence, the rate $N^{- (q - 2)/q}$ is negligible compared to $N^{- 1/2}$.
%\\
%The results are given under our hypothesis, but the lipschitzianity in law of the coefficients is given with respect to the 2-Wasserstein distance. Acting as in proof of the above mentioned theorems it is not difficult to see that the results still hold true after having replaced the 2- Wasserstein distance with the 1-Wasserstein distance.} \\
The propagation of chaos stated above implies
$$
\E [ W^2_2 (\mu^N_s, \bar \mu_s) ] \le \frac{C}{\sqrt{N}}.
$$ 
Indeed, to get the last relation, we introduce the empirical measure $\bar \mu^N_s = N^{-1} \sum_{i=1}^N \delta_{\bar X^i_s}$ of the independent particle system at time $s$ and use the triangle inequality for $W_2$. Then %Minkowski's inequality gives 
$$
\E [ W_2^2 (\mu^N_s, \bar \mu^N_s) ] \le \frac{1}{N} \sum_{i=1}^N \E [|X^{i,N}_s-\bar X^i_s|^2] %\Big( \frac{1}{N} \sum_i \big(\E [|X^i_s - \bar X^i_s|^2 ] \big)^{\frac 1 2} \Big)^2 
\le  \frac{C}{\sqrt{N}},
$$
whereas Theorem~1 of {\cite{145 review}} implies
$$
\E [W^2_2(\bar \mu^N_s,\bar \mu_s)] \le \frac{C}{\sqrt{N}}.
$$

\noindent
Step 3. Finally, the law of large numbers gives 
$$
\frac{1}{N} \sum_{i=1}^N \int_0^T f(\bar X^i_s, \bar \mu_s) d s \xrightarrow{\P} \E \Big[ \int_0^T f(\bar X_s, \bar \mu_s) d s \Big],  \qquad N \to \infty.
$$
\end{proof}

\subsection{Proof of Lemma \ref{l: conditional expectation}}
\begin{proof}
We use the same notation as before. \\ 
\\
Proof of Lemma \ref{l: conditional expectation}(2).  We decompose $X^i_{t_j} - X^i_{t_{j-1}} - \Delta_n b^i_{t_{j-1}}(\theta_{0,1})$ into $A^i_{j,1}$ and $H^i_{j,2} := A^i_{j,2} + B^i_j$, where
\begin{equation}
\begin{gathered}
A^i_{j,1} := \int^{t_j}_{t_{j-1}} a^i_{t_{j-1}}(\theta_{0,2}) d W^i_s, \qquad 
A^i_{j,2} := \int^{t_j}_{t_{j-1}} (a^i_s(\theta_{0,2})-a^i_{t_{j-1}}(\theta_{0,2})) d W^i_s,\\
B^i_j := \int^{t_j}_{t_{j-1}} (b^i_s(\theta_{0,1})-b^i_{t_{j-1}}(\theta_{0,1})) d s,
\end{gathered}\label{eq: dynamics X}
\end{equation}
are the same as in {\eqref{def: BAnorm}, \eqref{def: A12norm}}. 
%the proof of \eqref{e: norm theta2}, in Proposition \ref{p: norm L}.

Firstly, we will show that for any $p \ge 2$,
%\begin{equation}\label{eq: bound BA}
%\E_{t_{j-1}}[|B^i_j|^p]=R_{t_{j-1}}(\Delta^{\frac{3}{2}p}), \qquad \E_{t_{j-1}}[|A^i_{j,2}|^p] = R_{t_{j-1}}(\Delta^p),
%\end{equation}
%whence
\begin{equation}\label{eq: bound R}
\E [|H^i_{j,2}|^p] \le C \Delta_n^p.
%\E_{t_{j-1}} [|R^i_{j,2}|^p] = R_{t_{j-1}}(\Delta^p)
\end{equation}
%according to the definition of $R_{t_{j - 1}}(\Delta^p)$ function provided in \eqref{eq: definition R}. 
Using Jensen's inequality and Lipschitz continuity of $b(\theta_1,\cdot)$ we get
\begin{align}
\E [|B^i_j|^p] & \le  \E \Big[ \Delta_n^{p - 1} \int_{t_{j-1}}^{t_j} |b^i_{s} (\theta_{0,1})- b^i_{t_{j - 1}} (\theta_{0,1})|^p d s \Big] \nonumber \\
& \le C \Delta_n^{p - 1} \int_{t_{j-1}}^{t_j} ( \E [|X_s^i - X_{t_{j - 1}}^i|^p] + \E [W_2^p(\mu_s, \mu_{t_{j - 1}})] ) ds \nonumber \\
& \le C \Delta_n^{p-1} \int_{t_{j-1}}^{t_j} (s-t_{j-1})^{\frac p 2} d s = C \Delta_n^{\frac{3}{2}p},\label{ineq: bound Bp} 
\end{align}
where the last inequality follows from Lemma \ref{l: moments}(2) and (4). 
%For any $q\ge 1$, use of conditional Jensen's inequality gives 
%$$
%\E \big[ \big| \E_{t_{j-1}} [|B^i_j|^p] \big|^q \big] \le \E \big[ \E_{t_{j-1}} [ |B^i_j|^{pq} ] \big] = \E [ |B^i_j|^{pq} ],
%$$
%hence, the first relation in \eqref{eq: bound BA}.
Further use of the Burkholder-Davis-Gundy and Jensen inequalities gives  
\begin{align}
\E [ |A^i_{j,2}|^p] &\le C \E \Big[ \Big(  \int_{t_{j-1}}^{t_j} | a^i_s (\theta_{0,2})-a^i_{t_{j-1}}(\theta_{0,2}) |^2 d s \Big)^{\frac p 2} \Big]\nonumber\\ 
&\le C \Delta_n^{\frac{p}{2}-1} \int_{t_{j-1}}^{t_j} \E [ | a^i_s (\theta_{0,2})-a^i_{t_{j-1}}(\theta_{0,2}) |^p ] d s\nonumber\\
&\le C \Delta_n^p,\label{e: bound A2}
\end{align}
where the last inequality follows from Lipschitz continuity of $a(\theta_2,\cdot)$ and  Lemma~\ref{l: moments}(2) and (4) as so does \eqref{ineq: bound Bp}. Hence, we have %also shown the second relation in \eqref{eq: bound BA}.
shown \eqref{eq: bound R}.

Next, we have
\begin{equation}\label{e: bound A1}
\E [|A^i_{j,1}|^p] = C \Delta_{{n}}^{\frac{p}{2}}  \E [ |a^i_{t_{j-1}} (\theta_{0,2})|^p ] \le C \Delta_n^{\frac p 2}
%\E_{t_{j-1}}[|A^i_{j,1}|^p] = C \Delta^{\frac{p}{2}}  |a^i_{t_{j-1}} (\theta_2)|^p = R_{t_{j-1}}(\Delta^{\frac p 2})
\end{equation} 
since we know the absolute moments of a centered normal distribution and have linear growth of $a(\theta_{0,2},\cdot)$, moment bounds in Lemma \ref{l: moments}(1). In particular, we note
$$ 
\E_{t_{j-1}}[(A^i_{j,1})^4] = 3 \Delta_n^2 (c^2)^i_{t_{j-1}}(\theta_{0,2}).
$$
Finally, we have
\begin{equation}\label{eq:X4}
\E_{t_{j - 1}}[(X^i_{t_j} - X^i_{t_{j - 1}} - \Delta_n b^i_{t_{j - 1}}(\theta_{0,1}))^4]
= 3 \Delta_n^2 (c^2)^i_{t_{j-1}}(\theta_{0,2}) + \sum_{k = 0}^3 \binom{4}{k}  \E_{t_{j - 1}} [ ( A^i_{j,1} )^k (H^i_{j,2})^{4-k} ].
\end{equation}
For any $k = 0,1,2,3$ and $q \ge 1$, using Jensen's inequality for conditional expectation, we get
\begin{align*}
\E \big[ \big| \E_{t_{j-1}} [ (A^i_{j,1})^k (H^i_{j,2})^{4-k} ] \big|^q \big] &\le \E \big[ \big| (A^i_{j,1})^k (H^i_{j,2})^{4-k} \big|^q \big] 
%&\le C \Delta^{\frac{k}{2} q + (4-k) q} = 
\le C \Delta_n^{(4-\frac{k}{2})q},
\end{align*}
where the last inequality follows from \eqref{e: bound A2}, \eqref{eq: bound R} using Cauchy-Schwarz inequality. Hence, the %worst 
term {converging to $0$ in $L^q$ at the slowest rate}
is the one for which $k= 3$. 
We therefore obtain that the remaining sum on the right hand side of \eqref{eq:X4} is an $R^{{i}}_{t_{j - 1}}(\Delta_n^{\frac{5}{2}})$ function.\\
\\
Proof of Lemma \ref{l: conditional expectation}(3).  This follows directly from \eqref{ineq: bound Bp} by decomposing the dynamics of $X^i$ as in \eqref{eq: dynamics X} and remarking that the stochastic integral is centered. \\ \\
%\\
%\textit{Proof Point 4} {\color{blue} (where do we use it?)} As the odd moments of the stochastic integral are zero, we obtain 
%\begin{align*}
%\E_{j - 1}[(X^i_{t_j} - X^i_{t_{j - 1}} - \Delta b^i_{t_{j - 1}}(\theta_1))^3] &= \E_{j - 1}[(\int_{t_{j-1}}^{t_j} a^i_{t_{j - 1}} (\theta_{0, 2}) dW^i_s)^3] + \E_{j - 1}[({R}^i_j)^3] \\
%&+ \sum_{k = 1}^2 \binom{3}{k}  \E_{j - 1}[(\int_{t_{j-1}}^{t_j} a^i_{t_{j - 1}} (\theta_{0, 2}) dW^i_s)^k({R}^i_j)^{3-k}] \\
%& = R_{t_{j - 1}}(\theta, \Delta^3) +  \sum_{k = 1}^2 R_{t_{j - 1}}(\theta, \Delta^{\frac{k}{2}})R_{t_{j - 1}}(\theta, \Delta^{3-k}).
%\end{align*}
%As it was in the proof of Point 2, the worst term in the sum is achieved for the biggest $k$, which corresponds to $k=2$. It yields
%$$\E_{j - 1}[(X^i_{t_j} - X^i_{t_{j - 1}} - \Delta b^i_{t_{j - 1}}(\theta_1))^3] \le R_{t_{j - 1}}(\theta, \Delta^{2}).$$
%\textit{Proof Point 5} {\color{blue} (where do we use it?)} From Burkholder-Davis-Gundy inequality and \eqref{eq: bound R} we easily obtain that, for any $p \ge 2$, 
%\begin{align*}
%\E_{j - 1}[|X^i_{t_j} - X^i_{t_{j - 1}} - \Delta b^i_{t_{j - 1}}(\theta_1)|^p] &\le c \E_{j - 1}[|\int_{t_{j-1}}^{t_j} a^i_{t_{j - 1}} (\theta_{0, 2}) dW^i_s|^p] + c\E_{j - 1}[|{R}^i_j|^p]  \\
%& \le R_{t_{j - 1}}(\theta, \Delta^{\frac{p}{2}})  + R_{t_{j - 1}}(\theta, \Delta^p).
%\end{align*}
Proof of Lemma \ref{l: conditional expectation}(1). We decompose
$X^i_{t_j} - X^i_{t_{j - 1}} - \Delta_n b^i_{t_{j - 1}}(\theta_{0,1})$ into 
$$
A^i_j := A^i_{j,1} + A^i_{j,2} = \int_{t_{j-1}}^{t_j} a^i_{s} (\theta_{0,2}) dW^i_s,
$$
and $B^i_j$ satisfying respectively $\E [|A^i_j|^{2p}] \le C \Delta_n^{p}$ and $\E [|B^i_j|^{2p}] \le C \Delta_n^{3p}$, whence $\E [ |A^i_j B^i_j|^p ] \le C \Delta_n^{2p}$ for any $p \ge 1$, see \eqref{ineq: bound Bp}-\eqref{e: bound A1}. We conclude that
\begin{align*}
\E_{t_{j - 1}}[(X^i_{t_j} - X^i_{t_{j - 1}} - \Delta_n b^i_{t_{j - 1}}(\theta_{0,1}))^2]
%&= \E_{t_{j-1}} [(A^i_j)^2] + 2 \E_{t_{j-1}} [A^i_j B^i_j] + \E_{t_{j-1}} [(B^i_j)^2]\\ &=
= \int_{t_{j-1}}^{t_j} \E_{t_{j - 1}}[c^i_{s} (\theta_{0,2})] ds + R^{{i}}_{t_{j-1}}(\Delta_n^2).    
\end{align*}
We are left to show that we can replace $\E_{t_{j - 1}}[c^i_{s} (\theta_{0,2})]$ with $c^i_{t_{j-1}} (\theta_{0,2})$ and that the remaining integral is an $R^{{i}}_{t_{j-1}}(\Delta_n^2)$ function.

Under \textbf{\ref{as7}} we have that for any $i$, 
%= 1,\dots,N$,
$$
(x_1,\dots,x_N) \mapsto c \Big(\theta_{0,2}, x_i,\frac{1}{N}\sum_{j=1}^N \delta_{x_j} \Big) = {\tilde a^2 %_{\theta_2}
} 
\Big( x_i, \frac{1}{N} \sum_{j=1}^N K%_{\theta_2} 
(x_i,x_j) \Big) =: g^i(x_1,\dots,x_N)
$$
is a %nice 
{twice continuously differentiable} function %say $g^i$ 
from $\R^N$ to $\R$. %{\color{blue}having all partial derivatives up to order two}. 
Given a vector 
$(X^1_s,\dots,X^N_s)_{s \in [0,T]}$ of processes, we denote %by 
$$
(\partial^l_{x_k} c)^i_s (\theta_{0,2}) := \partial^l_{x_k} g^i(X^1_s, \dots, X^N_s).
$$ 
%its $l$-th order partial derivative with respect to $x_k$ at $(X^1_s, \dots, X^N_s)$ 
%, where $l=0,1,2$, $k=1,\dots,N$.
We apply the multidimensional It\^o's formula to $g^i(X^1_s,\dots,X^N_s) = c^i_s (\theta_{0,2})$
as follows: 
\begin{align*}
c^i_s(\theta_{0,2}) - c^i_{t_{j-1}} (\theta_{0,2}) 
= \sum_{k=1}^N \int_{t_{j-1}}^s \Big( &(\partial_{x_k} c)^i_u (\theta_{0,2}) b^k_u (\theta_{0,1}) + \frac{1}{2} (\partial^2_{x_k} c)^i_u (\theta_{0,2}) c^k_u (\theta_{0,2}) \Big) d u\\
+ \sum_{k=1}^N \int_{t_{j-1}}^s &(\partial_{x_k} c)^i_u (\theta_{0,2}) a^k_u (\theta_{0,2}) d W^k_u.
\end{align*}
Since the driving $(W^1_u, \dots, W^N_u)_{u \in[t_{j-1},s]}$ is independent of ${\cal F}^{{N}}_{t_{j-1}}$, it follows that
\begin{align}
\E_{t_{j-1}} [&c^i_s(\theta_{0,2})] - c^i_{t_{j-1}} (\theta_{0,2}) \nonumber\\ = \E_{t_{j-1}} \Big[ \sum_{k=1}^N \int_{t_{j-1}}^s \Big( (\partial_{x_k} &c)^i_u (\theta_{0,2}) b^k_u (\theta_{0,1})
 + \frac{1}{2} (\partial^2_{x_k} c)^i_u (\theta_{0,2}) c^k_u (\theta_{0,2}) \Big) d u \Big].\label{eq:a2remainder}
\end{align}
To conclude, we need to bound {each} $(\partial^l_{x_k} c)^i_u (\theta_{0,2})$, $l=1,2$. To do that, we rely on the assumption about the dependence of the diffusion coefficient on the convolution with a probability measure gathered in \textbf{\ref{as7}}.
To compute the derivatives with respect to $x_k$ we need to consider two different cases, depending on whether $k \neq i$ or $k = i$. When $k \neq i$ we have $(\partial_{x_k} c)^i_u (\theta_{0,2}) = 2
a^i_u (\theta_{0,2}) (\partial_{x_k} a)^i_u (\theta_{0,2})$, where
\begin{equation}\label{eq:partialaki}
(\partial_{x_k} a)^i_u (\theta_{0,2}) := \partial_y \tilde a%_{\theta_2} 
\Big( X^i_u, \frac{1}{N} \sum_{j=1}^N K%_{\theta_2} 
(X^i_u, X^j_u) \Big) \frac{1}{N} \partial_y K%_{\theta_2} 
(X^i_u,X^k_u),
\end{equation}
while for $k = i$ we have $(\partial_{x_i} c)^i_u (\theta_{0,2}) = 2 a^i_u (\theta_{0,2}) (\partial_{x_i} a)^i_u (\theta_{0,2})$, where
\begin{align*}
(\partial_{x_i} a)^i_u (\theta_{0,2}) := \partial_x \tilde a%_{\theta_2} 
\Big( X^i_u, \frac{1}{N} \sum_{j=1}^N &K%_{\theta_2} 
(X^i_u, X^j_u) \Big) + \partial_y \tilde a%_{\theta_2} 
\Big( X^i_u, \frac{1}{N} \sum_{j=1}^N K%_{\theta_2} 
(X^i_u, X^j_u) \Big)\\
\times \Big( \frac{1}{N} \sum_{j=1}^N \partial_x &K%_{\theta_2} 
(X^i_u,X^j_u) + \frac{1}{N} \partial_y K%_{\theta_2} 
(X^i_u,X^i_u) \Big).
\end{align*}
From polynomial growth of the $l$-th order partial derivatives of %$ K_{\theta_2} (\cdot)$, $\tilde a_{\theta_2} (\cdot)$, 
$K, \tilde a$ for $l = 0, 1$, that of $b(\theta_{0,1},\cdot)$, moment bounds in Lemma~\ref{l: moments}{(1)} applying Jensen's inequality it follows that 
$\sum_{k=1}^N (\partial_{x_k} c)^i_u (\theta_{0,2}) b^k_u (\theta_{0,1})$ is bounded in $L^p$ for any $p\ge 1$ uniformly in $u,i$. We proceed similarly to compute $(\partial_{x_k}^2 c)^i_u(\theta_{0,2})$. Then from polynomial growth of the $l$-th order partial derivatives of
$K, \tilde a$ for $l=0,1,2$, moment bounds in Lemma~\ref{l: moments}(1) applying Jensen's inequality it follows that
$\sum_{{k}=1}^N (\partial^2_{x_k} c)^i_u (\theta_{0,2}) c^k_u (\theta_{0,2})$ is bounded in $L^p$ for any $p \ge 1$ uniformly in $u,i$.
For any $p \ge 1$, $t_{j-1} \le s \le t_j$, repeatedly applying Jensen's inequality to \eqref{eq:a2remainder} we get
$$
\E \big[ \big| \E_{t_{j-1}} [c^i_s (\theta_{0,2})] - c^i_{t_{j-1}}(\theta_{0,2}) \big|^p \big] \le C (s-t_{j-1})^p,
$$
whence 
$$
\E \Big[ \Big| \int_{t_{j-1}}^{t_j} \big( \E_{t_{j-1}} [c^i_s (\theta_{0,2})] - c^i_{t_{j-1}} (\theta_{0,2}) \big) d s \Big|^p \Big] \le C \Delta_n^{2p}.
$$
which completes the proof.
\end{proof}

{
\section*{Acknowledgements}
We are grateful to two anonymous referees for truly helpful comments and suggestions.
}

\newpage

%\newpage 

%Remarks:
%\begin{enumerate}
%    \item The convergence in \eqref{lim:B221} resembles that in \eqref{lim:con1} in Lemma \ref{lemma:cns}.  We may want to write a technical lemma for the uniform in $\theta$ convergence:
%$$
%\frac{1}{N} \sum_i \sum_j f^i_{t_{j-1}} (\theta) (X^i_{t_j} - X^i_{t_{j-1}} - \Delta b^i_{t_{j-1}}(\theta_{0,1}))^2 \to_\P \int \int f(x,\bar \mu_t,\theta) c(x,\bar \mu_t, \theta_{0,2}) \bar \mu_t (d x) d t
%$$ 
%and 
%\begin{align*}
%&\frac{1}{N} \sum_j \sum_j f^i_{t_{j-1}} (\theta) (X^i_{t_j}-X^i_{t_{j-1}}-\Delta b^i_{t_{j-1}}(\theta_{0,1})) \to_\P 0,\\
%&\frac{\Delta}{N} \sum_j \sum_j f^i_{t_{j-1}} (\theta) \to_\P \int \int f(x,\bar \mu_t, \theta) \bar \mu_t(d x) d t\\
%\end{align*}
%as in \cite{Kes97}.

 %   \item Actually, for $(\sqrt{\Delta}/N) \partial_{\theta_1} \partial_{\theta_2} S^N_n (\theta)$ and $(\Delta/N) \partial_{\theta_2}^2 S^N_n (\theta)$ in Proposition \ref{p: second derivatives contrast} it is enough to show the convergence at $\theta=\theta_0$ and tightness in $(C(\Theta), \| \cdot \|_\infty)$ only.
 %   \item In \ref{as7} we may try the form $$
 %   a (x,\mu,\theta_{0,2}) = \tilde a \Big( x, \int K_2(x-y,\theta_{0,2}) \mu (d y), \theta_{0,2} \Big).
 %   $$
 %   \item I guess that if $a$ does not depend on $\theta_{0,2}$ and $$
 %   b (x,\mu,\theta_{0,1}) = \tilde b \Big( x, \int K_1 (x-y,\theta_{0,1}) \mu (d y), \theta_{0,1} \Big)
 %  $$
 %   the condition $N \Delta \to 0$ in Theorem \ref{th: normality} can be replaced by $N \Delta^{3/2} \to 0$.
%\end{enumerate}


\begin{thebibliography}{99}

\bibitem{AbrNic} Abraham, K., \& Nickl, R. (2020). On statistical Calderón problems. Mathematical Statistics and Learning, 2(2), 165-216.

\bibitem{Sjs} Amorino, C., \& Gloter, A. (2020). Contrast function estimation for the drift parameter of ergodic jump diffusion process. Scandinavian Journal of Statistics, 47(2), 279-346.

\bibitem{Joint} Amorino, C., \& Gloter, A. (2021). Joint estimation for volatility and drift parameters of ergodic jump diffusion processes via contrast function. Statistical Inference for Stochastic Processes, 24(1), 61-148.

\bibitem{Bal12} Baladron, J., Fasoli, D., Faugeras, O., \& Touboul, J. (2012). Mean-field description and propagation of chaos in
networks of Hodgkin-Huxley and FitzHugh-Nagumo neurons. The Journal of Mathematical Neuroscience, 2(1), 1-50.

\bibitem{Vyt} Belomestny, D., Pilipauskaitė, V., \& Podolskij, M. (2023). Semiparametric estimation of McKean-Vlasov SDEs. %To appear in 
Annales de l'Institut Henri Poincar\'e, Probabilit\'es et Statistiques,
%arXiv preprint arXiv:2107.00539.
59(1), 79-96.

\bibitem{Bis} Bishwal, J. P. N. (2011). Estimation in interacting diffusions: Continuous and discrete sampling. Applied Mathematics, 2(9), 1154-1158.

\bibitem{Bos} Bossy, M., \& Jabir, J. F. (2017, July). On the wellposedness of some McKean models with moderated or singular diffusion coefficient. In International Symposium on BSDEs (pp. 43-87). Springer, Cham.

\bibitem{13Imp} Brugna, C., \& Toscani, G. (2015). Kinetic models of opinion formation in the presence of personal conviction. Physical Review E, 92(5), 052818.

%\bibitem{Interpolation} Bennett, C., Sharpley, R. C. (1988). Interpolation of operators (Vol. 129). Academic press.

\bibitem{Car19} Cardaliaguet, P., Delarue, F., Lasry, J.-M., \& Lions, P.-L. (2019). The master equation and the convergence
problem in mean field games. 
%In Annals of Mathematics Studies, vol. 201. 
Princeton University Press, Princeton.

\bibitem{Cat08} Cattiaux, P., Guillin, A., \& Malrieu, F. (2008). Probabilistic approach for granular media equations in the
non-uniformly convex case. Probability Theory and Related Fields, 140(1), 19-40.

\bibitem{Review prop} Chaintron, L. P., \& Diez, A. (2022). Propagation of chaos: a review of models, methods and applications. II. Applications.  arXiv preprint 	arXiv:2106.14812.

\bibitem{Cha17} Chazelle, B., Jiu, Q., Li, Q., \& Wang, C. (2017). Well-posedness of the limiting equation of a noisy consensus
model in opinion dynamics. Journal of Differential Equations, 263(1), 365-397.

\bibitem{24 imp} Chen, X. (2021). Maximum likelihood estimation of potential energy in interacting particle systems from single-trajectory data. Electronic Communications in Probability, 26, 1-13.

\bibitem{ComGen} Comte, F., \& Genon-Catalot, V. (2020). Nonparametric drift estimation for i.i.d.\ paths of stochastic differential equations. Annals of Statistics, 48(6), 3336-3365.

%\bibitem{6cuc} Cucker, F., \& Smale, S. (2007). Emergent behaviour in flocks. IEEE Transactions on Automatic Control 52, 852-862.

%\bibitem{7cuc} Cucker, F. \& Smale, S. (2007). On the mathematics of emergence. Japanese Journal of Mathematics 2, 197-227.

\bibitem{Del18} Delattre, M., Genon-Catalot, V., \& Lar\'edo, C. (2018) Approximate maximum likelihood estimation for stochastic differential equations with random effects in the drift and the diffusion. Metrika, 81, 953–983.

\bibitem{Mix} Delattre, M., Genon-Catalot, V., \& Larédo, C. (2018). Parametric inference for discrete observations of diffusion processes with mixed effects. Stochastic Processes and their Applications, 128(6), 1929-1957.

\bibitem{Marc} Della Maestra, L., \& Hoffmann, M. (2022). Nonparametric estimation for interacting particle systems: McKean–Vlasov models. Probability Theory and Related Fields, 182(1), 551-613.

\bibitem{Hof2} 
%Della Maestra, L., \& Hoffmann, M. (2022). The LAN property for McKean-Vlasov models in a mean-field regime. arXiv preprint arXiv:2205.05932.
{Della Maestra, L., \& Hoffmann, M. (2023). The LAN property for McKean–Vlasov models in a mean-field regime. Stochastic Processes and their Applications, 155, 109-146.}

{
\bibitem{Den20} Denis, C., Dion, C., \& Martinez, M. (2020). Consistent procedures for multiclass classification of discrete diffusion paths.
Scandinavian Journal of Statistics.

\bibitem{Den21} Denis, C., Dion-Blanc, C., \& Martinez, M. (2021). A ridge estimator of the drift from discrete repeated observations of the solutions of a stochastic differential equation. Bernoulli, 27(4), 2675-2713. 
}

\bibitem{Zanella} Djehiche, B., Gozzi, F., Zanco, G., \& Zanella, M. (2022). Optimal portfolio choice with path dependent benchmarked labor income: a mean field model. Stochastic Processes and their Applications, 145, 48-85.


\bibitem{DosReis} Dos Reis, G., Engelhardt, S., \& Smith, G. (2022). Simulation of McKean-Vlasov SDEs with super-linear growth. IMA Journal of Numerical Analysis, 42(1), 874-922.

\bibitem{AAP} Dos Reis, G., Salkeld, W., \& Tugaut, J. (2019). Freidlin-Wentzell LDP in path space for McKean-Vlasov equations and the functional iterated logarithm law. The Annals of Applied Probability, 29(3), 1487-1540.

%\bibitem{Du} Du, K., Jiang, Y., \& Li, J. (2021). Empirical approximation to invariant measures for McKean--Vlasov processes: mean-field interaction vs self-interaction. arXiv preprint arXiv:2112.14112.

%\bibitem{Cucker} Erban, R., Haskovec, J., \& Sun, Y. (2016). A Cucker-Smale model with noise and delay. SIAM Journal on Applied Mathematics, 76(4), 1535-1557.

\bibitem{Fer97} Fernandez, B., \& M\'el\'eard, S. (1997). A Hilbertian approach for fluctuations on the McKean-Vlasov model. Stochastic processes and their applications, 71(1), 33-53.

\bibitem{FloZmi} Florens-Zmirou, D. (1989). Approximate discrete-time schemes for statistics of diffusion processes. Statistics: A Journal of Theoretical and Applied Statistics, 20(4), 547-557.

\bibitem {For08} {Forman, J. L., \& Sørensen, M. (2008). The Pearson diffusions: A class of statistically tractable diffusion processes. Scandinavian Journal of Statistics, 35(3), 438-465.}

%\bibitem{FouIch} Fouque, J. P., \& Ichiba, T. (2013). Stability in a model of interbank lending. SIAM Journal on Financial Mathematics, 4(1), 784-803.

\bibitem{Fou13} Fouque, J.P., \& Sun, L.H. (2013). Systemic risk illustrated. Handbook on Systemic Risk, Eds J.P Fouque and J Langsam.

\bibitem{145 review} Fournier, N., \& Guillin, A. (2015). On the rate of convergence in Wasserstein distance of the empirical measure. Probability Theory and Related Fields, 162(3), 707-738.
%, Publisher: Springer.




%\bibitem{Gartner} Gärtner, J. (1988). On the McKean-Vlasov limit for interacting diffusions. Math. Nachr. 137, 197-248. 

\bibitem{Gen90} Genon-Catalot, V. (1990). Maximum contrast estimation for diffusion processes from discrete observations. Statistics, 21(1), 99-116.

\bibitem{GenJac93} Genon-Catalot, V., \& Jacod, J. (1993). On the estimation of the diffusion coefficient for multidimensional diffusion processes. Annales de l'institut Henri Poincar\'e (B) Probabilit\'es et Statistiques, 29, 119-151.

\bibitem{GenLar1} Genon-Catalot, V., \& Larédo, C. (2021). Parametric inference for small variance and long time horizon McKean-Vlasov diffusion models. Electronic Journal of Statistics, 15(2), 5811-5854.

\bibitem{GenLar2} Genon-Catalot, V., \& Lar\'edo, C. (2021). Probabilistic properties and parametric inference of small variance nonlinear self-stabilizing stochastic differential equations. Stochastic Processes and their Applications, 142, 513-548.

\bibitem{Giesecke} Giesecke, K., Schwenkler, G., \& Sirignano, J. (2019). Inference for large financial systems. Mathematical Finance, 1-44.

\bibitem{GLM} Gloter, A., Loukianova, D., \& Mai, H. (2018). Jump filtering and efficient drift estimation for L\'evy-driven SDEs. The Annals of Statistics, 46(4), 1445.

\bibitem{Gobet 2002} Gobet, E. (2002). LAN property for ergodic diffusions with discrete observations. Annales de l'Institut Henri Poincare (B) Probability and Statistics, 38(5), 711-737.

\bibitem{smile} Guyon, J., \& Henry-Labordere, P. (2011). The smile calibration problem solved. Available at SSRN 1885032.

\bibitem{Gyo} Gyongy, I. (1986). Mimicking the one-dimensional marginal distributions of processes having an Ito Differential. Probability Theory and Related Fields, 71, 501-516.

\bibitem{HalHey80} Hall, P., \& Heyde, C. (1980). Martingale limit theory and its application. Academic Press, New York.

\bibitem{IbrHas} Ibragimov, I. A., \& Has’ Minskii, R. Z. (2013). Statistical estimation: ssymptotic theory (Vol. 16).
Springer Science \& Business Media.

\bibitem{Kall} Kallenberg, O. (1997). Foundations of modern probability (Vol. 2). Springer-Verlag, New York.

\bibitem{Kas90} Kasonga, R.A. (1990). Maximum likelihood theory for large interacting systems. SIAM Journal on Applied Mathematics, 50(3), 865-875.

%\bibitem{JouTse} Jourdain, B., Tse, A.: Central limit theorem over non-linear functionals of empirical measures with applications to the mean-field fluctuation of interacting particle systems (2020). arXiv:2002.01458

\bibitem{Kes97} Kessler, M. (1997). Estimation of an ergodic diffusion from discrete observations. Scandinavian Journal of Statistics, 24(2), 211-229.

\bibitem{Lac} Lacker, D., Shkolnikov, M., \& Zhang, J. (2020). Inverting the Markovian projection, with an application to local stochastic volatility models. The Annals of Probability, 48(5), 2189-2211.

\bibitem{Liu} Liu, M., \& Qiao, H. (2022). Parameter estimation of path-dependent McKean-Vlasov stochastic differential equations.
%, arXiv Prepr., (2020).
Acta Mathematica Scientia 42, 876-886. 

\bibitem{Mal01} Malrieu, F. (2001). Logarithmic Sobolev inequalities for some nonlinear PDE’s. Stochastic Processes and their Applications, 95(1), 109-132.

\bibitem{MarRos} Marie, N., \& Rosier, A. (2022). Nadaraya–Watson estimator for I.I.D.\ paths of diffusion processes. Scandinavian Journal of Statistics.

\bibitem{60imp} McKean, H. P. (1966). A class of Markov processes associated with nonlinear parabolic equations. Proceedings of the
National Academy of Sciences of the United States of America, 56(6), 1907-1911.

\bibitem{61imp} McKean, H. P. (1967). Propagation of chaos for a class of non-linear parabolic equations. Stochastic Differential Equations (Lecture Series in Differential Equations, Session 7, Catholic Univ., 1967), 41-57.

\bibitem{Mel96} M\'el\'eard, S. (1996) Asymptotic behaviour of some interacting particle systems; McKean-Vlasov and Boltzmann models. In Probabilistic models for nonlinear partial differential equations (pp. 42-95). Springer, Berlin, Heidelberg.

\bibitem{Mog99} Mogilner, A., \& Edelstein-Keshet, L. (1999). A non-local model for a swarm. Journal of mathematical biology, 38(6), 534-570.

\bibitem{Alm} Monter, S. A. A., Shkolnikov, M., \& Zhang, J. (2019). Dynamics of observables in rank-based models and performance of functionally generated portfolios. The Annals of Applied Probability, 29(5), 2849-2883.

\bibitem{65imp} Motsch, S., \& Tadmor, E. (2014). Heterophilious dynamics enhances consensus. SIAM review, 56(4), 577-621.

%\bibitem{Vacc} Müller, J., Tellier, A., \& Kurschilgen, M. (2021). A model of opinion dynamics with echo chambers explains the spatial distribution of vaccine hesitancy. arXiv preprint arXiv:2112.10230.

%\bibitem{Musco} Musco, C., Ramesh, I., Ugander, J., \& Witter, R. T. (2021). How to quantify polarization in models of opinion dynamics. arXiv preprint arXiv:2110.11981.

\bibitem{Nickl1} Nickl, R. (2020). Bernstein-von Mises theorems for statistical inverse problems I: Schrodinger equation. Journal of the European Mathematical Society, 22(8), 2697-2750.

\bibitem{Nickl2} Nickl, R. (2017). On Bayesian inference for some statistical inverse problems with partial differential equations. Bernoulli News, 24(2), 5-9.

\bibitem{Sharrock} Sharrock, L., Kantas, N., Parpas, P., \& Pavliotis, G. A. (2021). Parameter estimation for the McKean-Vlasov stochastic differential equation. arXiv preprint arXiv:2106.13751.

\bibitem{Shi} Shimizu, Y. (2006). $M$-Estimation for discretely observed ergodic diffusion processes with infinitely many jumps. Statistical Inference for Stochastic Processes, 9, 179-225.

%\bibitem{Shimizu thesis} Shimizu, Y. (2007). Asymptotic Inference for Stochastic Differential Equations with Jumps from Discrete Observations and Some Practical Approaches (Doctoral dissertation, University of Tokyo).

\bibitem{79imp} Sznitman, A.-S. (1991). Topics in propagation of chaos. In Ecole d'et\'e de probabilit\'es de Saint-Flour XIX-1989
%Lect. Notes Math., 1464 , 
(pp. 165-251). Springer, Berlin, Heidelberg.

\bibitem{van98} Van der Vaart, A. W. (1998). Asymptotic statistics. Cambridge University Press, Cambridge.

\bibitem{Wen} Wen, J., Wang, X., Mao, S., \& Xiao, X. (2016). Maximum likelihood estimation of McKean-Vlasov stochastic differential equation and its application. Applied Mathematics and Computation, 274, 237-246.

\bibitem{Yos92} Yoshida, N. (1992). Estimation for diffusion processes from discrete observation. Journal of Multivariate
Analysis, 41(2), 220-242.

\end{thebibliography}
\end{document}